\documentclass[11pt]{article}

\usepackage{amsmath, amsthm, amssymb, amsfonts}
\usepackage[normalem]{ulem} 
\usepackage{verbatim} 
\usepackage{hyperref}
\usepackage{xcolor}
\hypersetup{
     colorlinks,
     linkcolor={red!50!black},
     citecolor={blue!50!black},
     urlcolor={blue!80!black}
}

\usepackage{longtable} 
\usepackage[center]{caption}
\usepackage{diagbox}
\usepackage{xypic}
\usepackage{mathtools}

\usepackage{tikz-cd}

\usepackage{enumitem}
\setitemize{itemsep=1pt}
\setenumerate{itemsep=1pt}

\theoremstyle{plain}
\newtheorem{thm}{Theorem}

\newtheorem{cor}{Corollary}

\newtheorem{lemma}{Lemma}
\newtheorem{prop}{Proposition}
\newtheorem{proposition}[prop]{Proposition}

\newtheorem{defn}{Definition}
\newtheorem{definition}[defn]{Definition}

\newtheorem{lemmastar}{Lemma*}[lemma]

\theoremstyle{definition}

\newcommand{\BA}{{\mathbb{A}}}

\newcommand{\BC}{{\mathbb{C}}}

\newcommand{\BH}{{\mathbb{H}}}

\newcommand{\BL}{{\mathbb{L}}}

\newcommand{\BN}{{\mathbb{N}}}

\newcommand{\BQ}{{\mathbb{Q}}}

\newcommand{\BT}{{\mathbb{T}}}

\newcommand{\BZ}{{\mathbb{Z}}}

\newcommand{\CA}{{\mathcal A}}

\newcommand{\CC}{{\mathcal C}}

\newcommand{\CE}{{\mathcal E}}
\newcommand{\CF}{{\mathcal F}}

\newcommand{\CO}{{\mathcal O}}

\newcommand{\CS}{{\mathcal S}}

\newcommand{\CW}{{\mathcal W}}

\newcommand{\CZ}{{\mathcal Z}}

\newcommand{\Ff}{{\mathfrak{f}}}

\newcommand{\Fp}{{\mathfrak{p}}}

\newcommand{\ra}{{\ \longrightarrow\ }}

\newcommand{\blangle}{\Big\langle}
\newcommand{\brangle}{\Big\rangle}
\newcommand{\pr}{\mathop{\rm pr}\nolimits}
\newcommand{\Bl}{\mathop{\rm Bl}\nolimits}
\newcommand{\Pic}{\mathop{\rm Pic}\nolimits}

\newcommand{\Mbar}{{\overline M}}
\newcommand{\Cone}{\mathop{\rm Cone}\nolimits}
\newcommand{\Cocone}{\mathop{\rm Cocone}\nolimits}
\newcommand{\Spec}{\mathop{\rm Spec}\nolimits}
\newcommand{\QJac}{\mathop{\rm QJac}\nolimits}
\newcommand{\qpar}{\hbar} 
\newcommand{\pt}{{\omega}}
\newcommand{\ev}{\mathop{\rm ev}\nolimits}
\newcommand{\p}{\mathbb{P}}
\providecommand{\scr}{\mathcal} 
\newcommand{\pidiscr}{{\CW}}
\DeclareMathOperator{\Hilb}{\mathsf{Hilb}}

\DeclareMathOperator{\Sym}{Sym}
\DeclareMathOperator{\id}{id}

\DeclareMathOperator{\Hom}{Hom}
\newcommand{\Km}{\operatorname{Km}}
\newcommand{\AAA}{{\mathsf{A}}}
\newcommand{\im}{\mathop{\rm Im}\nolimits}
\newcommand{\len}{\mathop{\rm len}\nolimits}

\allowdisplaybreaks[2] 

\begin{document}
\title{Gromov-Witten invariants of\\the Hilbert schemes of points of a K3 surface}
\author{Georg Oberdieck}
\date{October 2015}
\maketitle
\begin{abstract}
We study the enumerative geometry of rational curves
on the Hilbert schemes of points of a K3 surface.

Let $S$ be a K3 surface and let $\Hilb^d(S)$ be the Hilbert scheme of $d$ points of $S$.
In case of elliptically fibered K3 surfaces $S \to \p^1$,
we calculate genus~$0$ Gromov-Witten invariants of $\Hilb^d(S)$,
which count rational curves incident to two generic fibers
of the induced Lagrangian fibration $\Hilb^d(S) \to \p^d$.
The generating series of these invariants is the Fourier expansion of
a power of the Jacobi theta function times a modular form, hence of a Jacobi form.

We also prove results for genus $0$ Gromov-Witten invariants of $\Hilb^d(S)$ for several other natural incidence conditions.
In each case, the generating series is again a Jacobi form.
For the proof we evaluate Gromov-Witten invariants of the Hilbert scheme of $2$ points of $\p^1 \times E$, where $E$ is an elliptic curve.

Inspired by our results, we conjecture a formula
for the quantum multiplication with divisor classes
on $\Hilb^d(S)$ with respect to primitive curve classes.
The conjecture is presented in terms of natural operators acting on the Fock space of~$S$.
We prove the conjecture in the first non-trivial case $\Hilb^2(S)$.
As a corollary, we find that the full genus $0$ Gromov-Witten theory of $\Hilb^2(S)$ in primitive classes is governed by Jacobi forms.

We present two applications.
A conjecture relating genus~$1$ invariants of $\Hilb^d(S)$ to the Igusa cusp form was proposed
in joint work with R.~Pandharipande in \cite{K3xE}. Our results prove the conjecture in case $d=2$.
Finally, we present a conjectural formula for the number of hyperelliptic curves
on a K3 surface passing through $2$ general points.
\end{abstract}

\newpage
\setcounter{tocdepth}{1}
\tableofcontents
\addtocounter{section}{-1}

\section{Introduction} \label{Intro}
\subsection{The Yau-Zaslow formula}
Let $S$ be a smooth projective K3 surface, and let
$\beta_h \in H_2(S,\BZ)$ be a primitive effective curve class
of self-intersection $\beta_h^2 = 2 h - 2$.
The Yau-Zaslow formula \cite{YZ}
predicts the number $\mathsf{N}_h$
of rational curves in class $\beta_h$
in the form of the generating series
\begin{equation} \sum_{h \geq 0} \mathsf{N}_h\, q^{h-1} = \frac{1}{q} \prod_{m \geq 1} \frac{1}{(1-q^m)^{24}} \,. \label{YZ} \end{equation}
The right hand side is the reciprocal of
the Fourier expansion of a classical modular form of weight 12,
the modular discriminant
\begin{equation}
\Delta(\tau) = q \prod_{m \geq 1} (1-q^m)^{24} \label{Delta}
\end{equation}
where $q = \exp(2 \pi i \tau)$ and $\tau \in \BH$.
The prediction \eqref{YZ}
was proven by Beauville \cite{Beauville2} and Chen \cite{ChenK3} using the compactified Jacobian,
and by Bryan and Leung \cite{BL} using Gromov-Witten theory.
It is the starting point of further research
into the enumerative geometry of algebraic curves on K3 surfaces,
see for example \cite{MPT, KKV, KKP} and the references therein.

The Hilbert scheme of $d$ points on $S$, denoted
\[ \Hilb^d(S), \]
is the moduli space of zero-dimensional subschemes of $S$ of length $d$,
see \cite{Lehn, Nakajima} for an introduction.
It is a non-singular projective variety of dimension $2d$,
which is simply-connected and carries a holomorphic symplectic form.
In case $d = 1$ we recover the original surface,
\[ \Hilb^1(S) = S, \]
while for $d \geq 2$ the Hilbert schemes $\Hilb^d(S)$ may be thought of as analogues of K3 surfaces in higher dimensions.

In this paper we study the enumerative geometry of rational curves on the Hilbert scheme of points $\Hilb^d(S)$ for all $d \geq 1$.
In particular, we obtain a generalization of the Yau-Zaslow formula \eqref{YZ}.

\subsection{Gromov-Witten invariants}  \label{Section_Definition_of_Gromov_Witten}
For all $\alpha \in H^{\ast}(S ; \BQ)$ and $i > 0$ let 
\[ \Fp_{-i}(\alpha) : H^{\ast}(\Hilb^d(S);\BQ) \to H^{\ast}(\Hilb^{d+i}(S);\BQ), \ \gamma \mapsto \Fp_{-i}(\alpha) \gamma \]
be the Nakajima creation operator
obtained by adding length $i$ punctual subschemes incident to a cycle Poincare dual to $\alpha$.
By a result of Grojnowski \cite{GrojH} and Nakajima \cite{Nakajima}
the cohomology of $\Hilb^d(S)$ is completely described by the cohomology of $S$
via the action of the operators $\Fp_{-i}(\alpha)$ on the vacuum vector 
\[ 1_S \in H^{\ast}(\Hilb^0(S);\BQ) = \BQ. \]

Let $\pt$ be the class of a point on $S$. For $\beta \in H_2(S;\BZ)$, define the class
\[ C(\beta) = \Fp_{-1}(\beta) \Fp_{-1}(\pt_S)^{d-1} 1_S \, \, \in H_2(\Hilb^d(S);\BZ) \,. \]
If $\beta = [C]$ for a curve $C \subset S$,
then $C(\beta)$ is the class of the curve obtained by fixing $d-1$ distinct points away from $C$
and letting a single point move on~$C$.
For brevity, we often write $\beta$ for $C(\beta)$.
For $d \geq 2$ let
\[ A = \Fp_{-2}(\pt_S) \Fp_{-1}(\pt_S)^{d-2} 1_S \]
be the class of an exceptional curve -- the locus of 2-fat points centered at a point $P \in S$ plus $d-2$ distinct points away from $P$.
For $d \geq 2$ we have
\[ H_2(\Hilb^d(S);\BZ) = \big\{ \beta + kA\ \big|\ \beta \in H_2(S;\BZ), k \in \BZ \big\}. \] \\

Let $\beta + k A \in H_2(\Hilb^d(S))$ be a non-zero effective curve class and consider the moduli space
\begin{equation} \Mbar_{g,m}( \Hilb^d(S), \beta + kA) \label{rigrge22r} \end{equation}
of $m$-marked stable maps{\footnote{The domain of a {\em stable map}
is always taken here to be connected.}}
$f : C \to \Hilb^d(S)$ of genus $g$ and class $\beta + k A$.
Since $\Hilb^d(S)$ carries a holomorphic symplectic $2$-form,
the virtual class on \eqref{rigrge22r} defined by ordinary Gromov-Witten theory vanishes \cite{KL}.
A modified reduced theory was defined in \cite{GWNL}
and gives rise to a non-zero \emph{reduced} virtual class
\[ [ \Mbar_{g,m}( \Hilb^d(S), \beta + kA) ]^{\text{red}} \]
of dimension $(1-g) (2d - 3) + 1$, see also \cite{STV,Pridham}.
Let
\[ \ev_i : \Mbar_{g,m}( \Hilb^d(S), \beta + kA) \to \Hilb^d(S), \quad i = 1, \dots, n \]
be the evaluation maps. The \emph{reduced Gromov-Witten invariant} of $\Hilb^d(S)$ in genus $g$ and class $\beta + kA$
with primary insertions 
\[ \gamma_1, \dots, \gamma_m \in H^{\ast}(\Hilb^d(S)) \]
is defined by
\begin{equation}
\big\langle\, \gamma_1, \dots, \gamma_m\, \big\rangle^{\Hilb^d(S)}_{g,\beta + k A} =
\int_{[\overline{M}_{g,m}(\Hilb^d(S), \beta + kA)]^{\text{red}}}
\ev_1^{\ast}(\gamma_1) \cup \dots \cup \ev_m^{\ast}(\gamma_m) \,. \label{bbbm} \end{equation}
In case $d = 1$ and $k \neq 0$, the moduli space $\overline{M}_{g,m}(\Hilb^d(S), \beta + kA)$
is empty by convention and the invariant \eqref{bbbm} is defined to vanish.


\subsection{The Yau-Zaslow formula in higher dimensions} \label{YZ_section_statement_of_results}
Let $\pi : S \to \p^1$ be an elliptically fibered K3 surface and let
\[ \pi^{[d]} : \Hilb^d(S) \ra \Hilb^d(\p^1) = \p^d \,, \]
be the induced \emph{Lagrangian} fibration with generic fiber a smooth Lagrangian torus. Let
\[ L_z \subset \Hilb^d(S) \]
denote the fiber of $\pi^{[d]}$ over a point $z \in \p^d$.

Let $F \in H_2(S;\BZ)$ be the class of a fiber of $\pi$, and let $\beta_h$ be a primitive effective curve class on $S$ with
\[ F \cdot \beta_h = 1 \quad \text{ and } \quad \beta_h^2 = 2h - 2 \,. \]
For points $z_1, z_2 \in \p^d$ and for all $d \geq 1$ and $k \in \BZ$, define the Gromov-Witten invariant
\begin{align*}
  \mathsf{N}_{d,h,k} &= \big\langle  L_{z_1}, L_{z_2} \big\rangle_{\beta_h + kA}^{\Hilb^d(S)} \\
&= 
\int_{[\Mbar_{0,2}(\Hilb^d(S), \beta_h + kA)]^{\text{red}}} \ev_1^{\ast}(L_{z_1}) \cup  \ev_2^{\ast}(L_{z_2})
\end{align*}
which (virtually) counts the number of rational curves in class $\beta_h + kA$ incident to the Lagrangians $L_{z_1}$ and $L_{z_2}$.
The first result of this paper is a complete evaluation of the invariants $\mathsf{N}_{d,h,k}$.

Define the Jacobi theta function
\begin{equation} \label{FFFdef}
F(z,\tau) = \frac{\vartheta_1(z,\tau)}{\eta^3(\tau)}
 = (y^{1/2} + y^{-1/2}) \prod_{m \geq 1} \frac{ (1 + yq^m) (1 + y^{-1}q^m)}{ (1-q^m)^2 }
\end{equation}
considered as a formal power series in the variables
\[ y = -e^{2 \pi i z} 
 \quad \text{ and } \quad q = e^{2 \pi i \tau} \]
where $|q|<1$.

\begin{thm} \label{MThm0} For all $d \geq 1$, we have
\begin{equation}
 \sum_{h \geq 0}\sum_{k \in \BZ} \mathsf{N}_{d,h,k} y^k q^{h-1}\ =\ F(z,\tau)^{2d-2} \cdot \frac{1}{\Delta(\tau)} \label{005}
\end{equation}
under the variable change $y = - e^{2\pi i z}$ and $q = e^{2 \pi i \tau}$.
\end{thm}

The right hand side of \eqref{005} is the Fourier expansion
of a Jacobi form\footnote{see \cite{EZ} for an introduction} of index $d-1$.
For $d = 1$ the class $A$ vanishes on $S$ and
by convention only the term $k = 0$ is taken in the sum on the left side of \eqref{005}.
Then, \eqref{005} specializes to the Yau-Zaslow formula \eqref{YZ}.

\subsection{Further Gromov-Witten invariants} \label{Section_More_evaluations_Introduction}
Let $S$ be a smooth projective K3 surface, let $\beta_h \in H_2(S,\BZ)$ be a primitive curve class of square 
\[ \beta_h^2 = 2h-2, \]
and let $\gamma \in H^2(S, \BZ)$ be a cohomology class with $\gamma \cdot \beta_h = 1$ and $\gamma^2 = 0$.
We define three sets of invariants.

For $d \geq 2$, consider the classes
\begin{align*}
C(\gamma) & = \Fp_{-1}(\gamma) \Fp_{-1}(\pt)^{d-1} 1_S \\
A & = \Fp_{-2}(\pt) \Fp_{-1}(\pt)^{d-2} 1_S
\end{align*}
which were defined in Section \ref{Section_Definition_of_Gromov_Witten}.
Define the first two invariants
\[
\mathsf{N}_{d,h,k}^{(1)} = \blangle\, C(\gamma)\, \brangle_{\beta_h + kA}^{\Hilb^d(S)}, \quad \quad \quad
\mathsf{N}_{d,h,k}^{(2)} = \blangle\, A\, \brangle_{\beta_h + kA}^{\Hilb^d(S)}
\]
counting rational curves incident to a cycle of class $C(\gamma)$ and $A$ respectively.

For a point $P \in S$ consider the incidence scheme of $P$,
\[ I(P) = \{\, \xi \in \Hilb^d(S)\ |\ P \in \xi\, \} \,. \]
For generic points $P_1, \dots, P_{2d-2} \in S$ define the third invariant
\begin{equation} \mathsf{N}_{d,h,k}^{(3)} = \blangle\, I(P_1),\, \dots,\, I(P_{2d-2})\, \brangle^{\Hilb^d(S)}_{\beta_h + kA}. \label{GW33invs}\end{equation}

By geometric recursions, the invariants $\mathsf{N}_{d,h,k}^{(i)}, i=1,2,3$ and $\mathsf{N}_{d,h,k}$ determine
the full Gromov-Witten theory of $\Hilb^2(S)$ in genus $0$.
In case $d = 2$ the invariants \eqref{GW33invs} are also related to counting hyperelliptic curves on a K3 surface
passing through $2$ generic points, see Section \ref{Section_hyperelliptic_curves} below. 

The following theorem provides a full evaluation of the invariants $\mathsf{N}_{d,h,k}^{(i)}$.
Consider the formal variables
\[ y = - e^{2 \pi i z} \quad \text{ and } \quad q = e^{2 \pi i \tau} \]
expanded in the region $|y|<1$ and $|q| < 1$,
and the function
\begin{equation} \label{G_Function_def}
\begin{aligned}
G(z, \tau) & = F(z,\tau)^2 \left( y \frac{d}{dy} \right)^2 \log( F(z,\tau) ) \\
& = F(z,\tau)^2 \cdot \bigg\{ \frac{y}{(1+y)^2} - \sum_{d \geq 1} \sum_{m | d} m \big( (-y)^{-m} + (-y)^m \big) q^d \bigg\} \\
& = 1 + (y^{-2} + 4 y^{-1} + 6 + 4 y^{1} + y^2) q \\
& \quad \quad \quad \quad \quad \quad + (6y^{-2} + 24y^{-1}+ 36 + 24y + 6y^2)q^2 + \, \dots \ .
\end{aligned}
\end{equation}

\begin{thm} \label{MThm} For all $d \geq 2$, we have
\begin{align*}
 \sum_{h \geq 0} \sum_{k \in \BZ} \mathsf{N}^{(1)}_{d,h,k} y^k q^{h-1} & = G(z,\tau)^{d-1}\frac{1}{\Delta(\tau)} \\[4pt]
 \sum_{h \geq 0} \sum_{k \in \BZ} \mathsf{N}^{(2)}_{d,h,k} y^k q^{h-1} & = \frac{1}{2-2d} \Big( y \frac{d}{dy} \big( G(z,\tau)^{d-1} \big) \Big) \frac{1}{\Delta(\tau)} \\[4pt]
 \sum_{h \geq 0} \sum_{k \in \BZ} \mathsf{N}^{(3)}_{d,h,k} y^k q^{h-1} & = \frac{1}{d} \binom{2d-2}{d-1} \Big( q \frac{d}{dq} F(z,\tau)\Big)^{2d-2} \frac{1}{\Delta(\tau)}
\end{align*}
under the variable change $y = - e^{2 \pi i z}$ and $q = e^{2 \pi i \tau}$.
\end{thm}

\subsection{Quantum cohomology}
\subsubsection{Reduced quantum cohomology}
Let $\qpar$ be a formal parameter with $\qpar^2 = 0$.
The reduced quantum cohomology of $\Hilb^d(S)$ is a formal deformation of
the ordinary cup-product multiplication in $H^{\ast}(\Hilb^d(S))$
defined by
\begin{equation} \langle a \ast b, c \rangle = \langle a \cup b, c \rangle + \qpar \sum_{\beta > 0} \langle a,b,c \rangle_{0,\beta}^{\Hilb^d(S)} q^\beta \,, \label{rgrgrg} \end{equation}
where $\langle a,b \rangle = \int_{\Hilb^d(S)} a \cup b$ is the standard intersection form,
$\beta$ runs over all non-zero elements of the cone ${\rm Eff}_{\Hilb^d(S)}$ of effective curve classes in $\Hilb^d(S)$,
the symbol $q^{\beta}$ denotes the corresponding element in the semi-group algebra,
and $\langle a,b,c \rangle_{0,\beta}^{\Hilb^d(S)}$ denote the \emph{reduced} genus $0$ Gromov-Witten invariants of $\Hilb^d(S)$ in class $\beta$;
see \cite{QGQC} for the related case of equivariant quantum cohomology.

By the WDVV equation for reduced virtual classes (see Appendix~\ref{section_reducedWDVV}),
equality \eqref{rgrgrg} defines a commutative and associative product on
\[ H^{\ast}(\Hilb^d(S), \BQ) \otimes \BQ[[ {\rm Eff}_{\Hilb^d(S)} ]] \otimes \BQ[\qpar]/\qpar^2, \]
which we call the reduced quantum cohomology ring
\begin{equation} Q H^{\ast}( \Hilb^d(S) ) \,. \label{15fsdfsf} \end{equation}


The ordinary cohomology ring structure on $H^{\ast}( \Hilb^d(S) , \BQ)$
has been explicitly determined by Lehn and Sorger in \cite{LS_K3}.
In this paper, we put forth several conjectures and results about its quantum deformation \eqref{15fsdfsf}.
Our results will concern only the quantum multiplication with a divisor class on $\Hilb^d(S)$.
In other cases \cite{Lehn2,LQW,MO2, OP, QGQC}, this has been the first step towards a more complete understanding.
We will also restrict to \emph{primitive} classes~$\beta$ below.

\subsubsection{Elliptic K3 surfaces}
Let $\pi : S \to \p^1$ be an elliptic K3 surface with a section,
and let $B$ and $F$ denote the class of a section and fiber respectively.
For every $h \geq 0$, we set
\[ \beta_h = B + hF \,. \]
For $d \geq 1$ and cohomology classes $\gamma_1, \dots, \gamma_m \in H^{\ast}(\Hilb^d(S);\BQ)$, define the quantum bracket
\begin{equation*}
 \big\langle \gamma_1, \dots, \gamma_m \big\rangle_q^{\Hilb^d(S)}
= \sum_{h \geq 0} \sum_{k \in \BZ} y^k q^{h-1} \big\langle \gamma_1, \dots, \gamma_m \big\rangle^{\Hilb^d(S)}_{\beta_h + k A} \,. 
\end{equation*}
Define the \emph{primitive} quantum multiplication $\ast_{\text{prim}}$ by
\begin{equation} \big\langle a , b \ast_{\text{prim}} c \big\rangle\ 
=\ \big\langle a , b \cup c \big\rangle
+ \qpar \cdot \big\langle a, b, c \big\rangle_q \,. \label{vsdfdsjf} \end{equation}
for all $a,b,c$. Since $\qpar^2 = 0$, different curve classes $\beta$ dont interact, and
$\ast_{\text{prim}}$ defines a commutative and associative product on
\begin{equation}  H^{\ast}(\Hilb^d(S),\BQ) \otimes \BQ((y))((q)) \otimes \BQ[\qpar]/\qpar^2 \,. \label{jfnvief} \end{equation}

The main result of Section \ref{Section_Quantum_Cohomology} is a conjecture for
the primitive quantum multiplication with divisor classes. 
By the divisor axiom and by deformation invariance
the conjecture explicitly predicts the full $2$-point genus~$0$
Gromov-Witten theory of all Hilbert schemes of points of a K3 surface in primitive classes.
By direct calculations using the WDVV equation and the evaluations of Section~\ref{Section_More_Evaluations},
we prove the conjecture in case $\Hilb^2(S)$.

\subsubsection{Quasi-Jacobi forms} \label{Section_Quantum_Cohomology_Jacobi_Forms}
Let $(z,\tau) \in \BC \times \BH$. 
The ring $\QJac$ of quasi-Jacobi forms is defined as the linear subspace
\[ \QJac \subset \BQ[ F(z,\tau), E_2(\tau), E_4(\tau), \wp(z,\tau), \wp^{\bullet}(z,\tau), J_1(z,\tau)] \]
of functions which are holomorphic at $z=0$ for generic $\tau$;
here $F(z,\tau)$ is the Jacobi theta function \eqref{FFFdef},
$E_{2k}$ are the classical Eisenstein series, $\wp$ is the Weierstrass elliptic function,
$\wp^{\bullet}$ is its derivative with respect to $z$,
and $J_1$ is the logarithmic derivative of $F$ with respect to $z$, see Appendix \ref{Appendix_Quasi_Jacobi_Forms}.

We will identify a quasi Jacobi form $\psi \in \QJac$ with its power series expansions in the variables
\[ q = e^{2 \pi i \tau} \quad \text{ and } \quad y = - e^{2 \pi i z}. \]
The space $\QJac$ is naturally graded by index $m$ and weight $k$:
\[ \QJac = \bigoplus_{m \geq 0} \bigoplus_{k \geq -2m} \QJac_{k,m} \]
with finite-dimensional summands $\QJac_{k,m}$.

Based on the proven case of $\Hilb^2(S)$ and effective calculations for $\Hilb^d(S)$ for any $d$, we have the following
results that link curve counting on $\Hilb^d(S)$ to quasi-Jacobi forms.

\begin{thm} \label{jac_thm} For all $\mu, \nu \in H^{\ast}(\Hilb^2(S))$, we have
\[ \langle \mu, \nu \rangle^{\Hilb^2(S)}_q = \frac{\psi(z,\tau)}{\Delta(\tau)} \]
for a quasi-Jacobi form $\psi(z,\tau)$ of index $1$ and weight $\leq 6$.
\end{thm}

Since $\Mbar_{0}(\Hilb^2(S), \gamma)$ has virtual dimension $2$ for all $\gamma$,
Theorem \ref{jac_thm} implies that the full genus $0$ Gromov-Witten theory of $\Hilb^2(S)$ in primitive classes
is governed by quasi-Jacobi forms.

\vspace{8pt}
\noindent{\bf Conjecture J.} {\em For $d \geq 1$ and for all $\mu, \nu \in H^{\ast}(\Hilb^d(S))$,
we have
\[ \langle \mu, \nu \rangle^{\Hilb^d(S)}_q = \frac{\psi(z,\tau)}{\Delta(\tau)} \]
for a quasi-Jacobi form $\psi(z,\tau)$ of index $d-1$ and weight $\leq 2 + 2d$.}
\vspace{8pt}

A sharper formulation of Theorem \ref{jac_thm} and Conjecture J specifying the weight appears in Lemma \ref{index_weight_lemma_jacforms}.

\subsection{Application 1: Genus 1 invariants of $\Hilb^d(S)$}
Let $S$ be a K3 surface and let $\beta_h \in H^2(S, \BZ)$ be a primitive curve class of
square $\beta_h^2 = 2h-2$.
Let $(E,0)$ be a nonsingular elliptic curve with origin $0\in E$,
and let
\begin{equation} \overline{M}_{(E,0)}(\Hilb^d(S), \beta_h + kA) \label{ofovfvfvgd} \end{equation}
be the fiber of the forgetful map
\[ \Mbar_{1,1}(\Hilb^d(S), \beta_h + k A) \to \Mbar_{1,1} \,. \]
over the moduli point $(E,0) \in \Mbar_{1,1}$. Hence, \eqref{ofovfvfvgd} is the moduli space parametrizing
stable maps to $\Hilb^d(S)$ with 1-pointed domain with complex structure {\em fixed} after stabilization to be $(E,0)$.
The moduli space \eqref{ofovfvfvgd} carries a reduced virtual class of dimension~$1$.

For $d>0$ consider the reduced Gromov-Witten potential
\begin{equation} \label{wssw}
 \mathsf{H}_d(y,q) = 
\sum_{k \in \BZ} \sum_{h\geq 0}y^k q^{h-1} 
\int_{[ \overline{M}_{(E,0)}(\Hilb^d(S), \beta_h + kA) ]^{\text{red}}}  \text{ev}_0^*(\beta_{h,k}^\vee)\,,
\end{equation}
where the divisor class $\beta_{h,k}^\vee \in H^2(\Hilb^d(S),\mathbb{Q})$ is chosen to satisfy
\begin{equation}\label{ffqq}
\int_{\beta_h + kA} \beta_{h,k}^\vee  = 1 \, . 
\end{equation}
The invariants \eqref{wssw} virtually count the number of maps
from the elliptic curve~$E$ to the Hilbert scheme $\Hilb^d(S)$ in the classes $\beta_h + kA$.
By degenerating $E$ to a nodal curve, resolving and using the divisor equation,
the series $\mathsf{H}_d(y,q)$ is seen to not depend on the choice of $\beta_{h,k}^{\vee}$.

The case $d=1$ of the series $\mathsf{H}_d(y,q)$ is determined by the Katz-Klemm-Vafa formula \cite{MPT}.
In case $d=2$ we have the following result.

\begin{prop} \label{rgergegssa}Under the variable change $y = - e^{2 \pi i z}$ and $q = e^{2 \pi i \tau}$,
\[ \mathsf{H}_2(y,q) =
 F(z,\tau)^2 \cdot \left( 54 \cdot \wp(z,\tau) \cdot E_2(\tau) - \frac{9}{4} E_2(\tau)^2  + \frac{3}{4} E_4(\tau) \right) \frac{1}{\Delta(\tau)}
\]
\end{prop}

In joint work with Rahul Pandharipande a correspondence between curve counting on $\Hilb^d(S)$ and
the enumerative geometry of the product Calabi-Yau $S \times E$ was proposed in \cite{K3xE}.
This in turn lead to a explicit conjecture for $\mathsf{H}_d(y,q)$ for all $d$
in terms of the reciprocal of the Igusa cusp form $\chi_{10}$.
Proposition \eqref{rgergegssa} verifies this conjecture in case $d=2$.

\subsection{Application 2: Hyperelliptic curves} \label{Section_hyperelliptic_curves}
A projective nonsingular curve $C$ of genus $g\geq 2$ is {\em hyperelliptic}
if $C$ admits a degree 2 map to $\p^1$,
\[ C \rightarrow \p^1\,. \]
The locus of hyperelliptic curves in the moduli space $M_g$ of non-singular curves of genus $g$ is a closed substack of codimension $g-2$.
Let
\[ \mathcal{H}_g \in H^{2(g-2)}(\Mbar_g, \BQ) \]
be the stack fundamental class
of the closure of nonsingular hyperelliptic curves inside $\Mbar_g$.
By results of Faber and Pandharipande \cite{FPM}, $\mathcal{H}_g$ is
a tautological class \cite{FP13} of codimension~$g-2$. 

Let $S$ be a K3 surface, let $\beta_h \in H^2(S)$ be a primitive curve class of square $\beta_h^2 = 2h-2$,
and let 
\[ \pi : \Mbar_g(S,\beta_h) \to \Mbar_g \]
be the forgetful map from the moduli space of genus $g$ stable maps to $S$ in class $\beta_h$.
A virtual count of
genus $g \geq 2$ hyperelliptic curves on $S$ in class $\beta_h$ passing through
$2$ general points is defined by the integral
\[ \mathsf{H}_{g, h} = \int_{ [ \Mbar_{g,2}(S, \beta_h) ]^{\text{red}} } \pi^{\ast}(\mathcal{H}_g) \ev_1^{\ast}( \omega ) \ev_2^{\ast}( \omega ), \]
where $\omega \in H^4(S, \BZ)$ is the class of a point.


In \cite{Grab} T.~Graber used the genus~$0$ Gromov-Witten theory of $\Hilb^2(\p^2)$ to obtain
enumerative results on hyperelliptic curves in $\p^2$.
A similar strategy has been applied for $\p^1 \times \p^1$ \cite{Pon07}
and for abelian surfaces \cite{Ros14, BOPY} (modulo a transversality result).
By arguments parallel to the abelian case \cite{BOPY},
Theorem \ref{MThm} above leads to the following prediction for $\mathsf{H}_{g, h}$.
%
%

Let $\Delta(\tau) = q \prod_{m \geq 1} (1-q^m)^{24}$ be the modular discriminant and let 
\begin{align*}
F(z,\tau) & = u \, \exp \left( {\sum_{k \geq 1}} \frac{(-1)^k B_{2k}}{2k (2k)!} E_{2k}(\tau) u^{2k} \right)
\end{align*}
be the Jacobi theta function which appeared already in \eqref{FFFdef} expanded in the variables
\begin{equation} q = e^{2 \pi i \tau} \quad \text{ and } \quad u = 2 \pi z \,. \label{wwrewrew} \end{equation}

\vspace{8pt}
\noindent{\bf Conjecture H.} Under the variable change \eqref{wwrewrew},
\[ \sum_{h \geq 0} \sum_{g \geq 2} u^{2g+2} q^{h-1} \mathsf{H}_{g, h} =
 \left( q \frac{d}{dq} F(z,\tau) \right)^2 \cdot \frac{1}{\Delta(\tau)}
\]
\vspace{8pt}

By a direct verification using results of \cite{BL, MPT} and an explicit expression \cite{HM} for
\[ \mathcal{H}_3 \in H^2(\Mbar_3, \BQ)\, , \]
Conjecture H holds in the first non-trivial case $g = 3$.
Similar conjectures relating the Gromov-Witten count of $r$-gonal curves on the K3 surface $S$
to the genus $0$ Gromov-Witten invariants of $\Hilb^r(S)$ can be made.

The virtual counts $\mathsf{H}_{g, h}$ have contributions from the boundary of the moduli space,
and do \emph{not} correspond to the actual, enumerative count of hyperelliptic curves on $S$.
For example, $\mathsf{H}_{3,1} = - \frac{1}{4}$
is both rational and negative. For $h \geq 0$ BPS numbers $\mathsf{h}_{g, h}$ of genus $g$ hyperelliptic curves on $S$ in class $\beta_h$ are defined by the expansion
\begin{equation} \sum_{g \geq 2} \mathsf{h}_{g, h}
\left( 2 \sin( u/2)  \right)^{2g+2}
= \sum_{g \geq 2}\, \mathsf{H}_{g,h} \, u^{2g+2} \,.
\label{BPS_expansion}
\end{equation}
The invariants $\mathsf{h}_{g, h}$ are expected to yield the enumerative count
of genus~$g$ hyperelliptic curves in class $\beta_h$ on a generic K3 surface $S$ carrying a curve class~$\beta_h$, compare \cite[Section 0.2.4]{BOPY}.

The invariants $\mathsf{h}_{g,h}$ vanish for $h=0,1$.
The first non-vanishing values of $\mathsf{h}_{g, h}$ are presented in Table \ref{hyptable} below.
The distribution of the non-zero values in Table \ref{hyptable}
matches precisely the results of Ciliberto and Knutsen in \cite[Theorem~0.1]{CK14}:
there exist curves on a generic K3 surface in class $\beta_h$ with normalization a hyperelliptic curve of genus $g$ if and only if
\begin{equation*}
h\ \geq\ g + \Big\lfloor \frac{g}{2} \Big\rfloor \Big( g - 1 - \Big\lfloor \frac{g}{2} \Big\rfloor \Big) \,.
\end{equation*}

\begin{figure}[h]
\begin{longtable}{| c | c | c | c | c | c | c | c | c | c | c |}
\hline
\diagbox[width=1.1cm,height=0.7cm]{$h$}{$g$} 
& $2$ & $3$ & $4$ & $5$ & $6$   \\
\nopagebreak \hline
$2$ & $1$ & $0$ & $0$ & $0$ & $\phantom{0}0\phantom{0}$ \\
$3$ & $36$ & $0$ & $0$ & $0$ & $0$ \\
$4$ & $672$ & $6$ & $0$ & $0$ & $0$ \\
$5$ & $8728$ & $204$ & $0$ & $0$ & $0$ \\
$6$ & $88830$ & $3690$ & $9$ & $0$ & $0$ \\
$7$ & $754992$ & $47160$ & $300$ & $0$ & $0$ \\
$8$ & $5573456$ & $476700$ & $5460$ & $0$ & $0$ \\
$9$ & $36693360$ & $4048200$ & $70848$ & $36$ & $0$ \\
$10$ & $219548277$ & $29979846$ & $730107$ & $1134$ & $0$ \\
$11$ & $1210781880$ & $198559080$ & $6333204$ & $19640$ & $0$ \\
$12$ & $6221679552$ & $1197526770$ & $47948472$ & $244656$ & $36$ \\
$13$ & $30045827616$ & $6666313920$ & $324736392$ & $2438736$ & $1176$ \\
$14$ & $137312404502$ & $34612452966$ & $2002600623$ & $20589506$ & $20895$ \\
$15$ & $597261371616$ & $169017136848$ & $11396062440$ & $152487720$ & $265860$ \\
\hline
\caption{The first values for the counts $\mathsf{h}_{g,h}$ of hyperelliptic
curves of genus~$g$ and class $\beta_h$ on a generic K3 surface $S$ passing through $2$ general points,
as predicted by Conjecture H and the BPS expansion \eqref{BPS_expansion}.} \label{hyptable}
\end{longtable}
\end{figure}

\subsection{Plan of the paper}
In Section 1, we introduce the bare notational necessities and prove a few general lemmas.
In Section 2 we prove Theorem \ref{MThm0} by reducing to an elliptic K3 surface with 24 rational nodal fibers
and by comparision with rational curve counts on a Kummer K3.
In Section 3 and 4 we prove Theorem \ref{MThm} by reducing the statement 
to a calculation of Gromov-Witten invariants of $\Hilb^2(\p^1 \times E)$.
This approach is mainly independent from the Kummer K3 geometry used in Section 2,
and yields a second proof of Theorem~\ref{MThm0}.
In Section 5, we present the conjectures and results on the quantum cohomology ring of $\Hilb^d(K3)$.
Here we also prove Theorem~\ref{jac_thm} and Proposition~\ref{rgergegssa}.
In Appendix A, we present the precise form of the WDVV equations for reduced invariants.
In Appendix B, we introduce the notion of a quasi-Jacobi form, and list numerical results related to the conjectures of Section 5.

\subsection{Acknowledgements}
I would like to thank the following people.
First and foremost, my advisor Rahul Pandharipande for suggesting the topic and for all his support over the years.
Aaron Pixton for guessing a key series needed in the proof.
Qizheng Yin for pointing out the connection between $\Hilb^2(K3)$ and the Kummer geometry.
And Ben Bakker, \mbox{Jim Bryan}, Lothar G\"ottsche, Simon H\"aberli, Jonas Jermann,
Alina Marian, \mbox{Eyal Markman}, Davesh Maulik, Andrew Morrison,
Martin Raum, Emanuel Scheidegger, Timo Sch\"urg,
Qizheng Yin and Paul Ziegler for various discussions and comments related to the topic.

The author was supported by the Swiss National Science Foundation grant SNF 200021\_143274.

\section{Preliminaries}
Let $S$ be a smooth projective surface and let $\Hilb^d(S)$ be the Hilbert scheme of $d$ points of $S$.
By definition, $\Hilb^0(S)$ is a point parametrizing the empty subscheme.

\subsection{Notation}
We always work over $\BC$.
All cohomology coefficients are in $\BQ$ unless denoted otherwise.
We let $[V]$ denote the homology class of an algebraic cycle~$V$.

On a connected smooth projective variety $X$,
we will freely identify cohomology and homology classes by Poincare duality.
We write
\begin{align*}
\pt = \pt_X & \in H^{2 \dim(X)}(X;\BZ), \\
e = e_X & \in H^0(X;\BZ)
\end{align*}
for the class of a point and the fundamental class of $X$ respectively.
Using the degree map we identify the top cohomology class with the underlying ring:
\[ H^{2 \dim(X)}(X, \BQ) \equiv \BQ. \]
The tangent bundle of $X$ is denoted by $T_X$.

A homology class $\beta \in H_2(X, \BZ)$ is an \emph{effective curve class}
if $X$ admits an algebraic curve $C$ of class $[C] = \beta$.
The class $\beta$ is \emph{primitive} if it is indivisible in $H_2(X,\BZ)$.

\subsection{Cohomology of $\Hilb^d(S)$} \label{cohhil}
\subsubsection{The Nakajima basis}
Let $(\mu_1, \dots, \mu_l)$ with $\mu_1 \geq \ldots \geq \lambda_l \geq 1$
be a partition
and let 
\[ \alpha_1, \dots, \alpha_l \in H^{\ast}(S; \BQ) \]
be cohomology classes on $S$.
We call the tuple
\begin{equation} \mu = \big( (\mu_1, \alpha_1), \dots, (\mu_l, \alpha_l) \big) \label{iejfisfsg} \end{equation}
a cohomology-weighted partition of size $|\mu| = \sum \mu_i$.

If the set $\{ \alpha_1, \dots, \alpha_l \}$ is ordered, we call \eqref{iejfisfsg}
ordered if for all $i \leq j$
\[ \mu_i \geq \mu_j \quad  \text{ or } \quad ( \mu_i = \mu_j\ \text{ and }\ \alpha_i \geq \alpha_j ) \,. \]
For $i > 0$ and $\alpha \in H^{\ast}(S ; \BQ)$, let
\begin{equation*} \Fp_{-i}(\alpha) : H^{\ast}( \Hilb^d(S) , \BQ) \to H^{\ast}( \Hilb^{d+i}(S) , \BQ ) \end{equation*}
be the Nakajima creation operator \cite{N2},
and let
\[ 1_S \in H^{\ast}( \Hilb^0(S), \BQ) = \BQ \]
be the vacuum vector.
A cohomology weighted partition \eqref{iejfisfsg}
defines the cohomology class
\[ \Fp_{-\mu_1}(\alpha_1) \dots \Fp_{-\mu_l}(\alpha_l) \, 1_S \, \in\, H^{\ast}( \Hilb^{|\mu|}(S) ) \,. \]

Let $\alpha_1, \dots, \alpha_p$ be a homogeneous ordered basis of $H^{\ast}(S ; \BQ)$.
By a theorem of Grojnowski \cite{GrojH} and Nakajima \cite{N2},
the cohomology classes associated to all
ordered cohomology weighted partitions of size $d$
with cohomology weighting by the $\alpha_i$
not repeating factors $(\alpha_j,k)$ with $\alpha_j$ odd,
form a basis of the cohomology $H^{\ast}(\Hilb^d(S) ; \BQ)$.

\subsubsection{Special cycles} \label{special_cycles}
We will require several natural cycles and their cohomology classes.
In the definitions below, we set $\Fp_{-m}(\alpha)^{k} = 0$ whenever $k < 0$.

\vspace{10pt}
\noindent \textbf{(i) The diagonal}\\[4pt]
The diagonal divisor
\[ \Delta_{\Hilb^d(S)} \subset \Hilb^d(S) \]
is the reduced locus of subschemes $\xi \in \Hilb^d(S)$ such that ${\rm len}(\CO_{\xi,x}) \geq 2$ for some $x \in S$.
It has cohomology class
\begin{equation*} [ \Delta_{\Hilb^d(S)} ]\, =\, \frac{1}{(d-2)!} \, \Fp_{-2}(e) \Fp_{-1}(e)^{d-2} 1_S \, =\, -2 \cdot c_1(\CO_{S}^{[d]}), \end{equation*}
where we let $E^{[d]}$ denote the tautological bundle on $\Hilb^d(S)$ associated to a vector bundle $E$ on $S$, see \cite{Lehn2,Lehn}.

\vspace{10pt}
\noindent \textbf{(ii) The exceptional curve}\\[4pt]
Let $\Sym^d(S)$ be the $d$-th symmetric product of $S$ and let
\[ \rho : \Hilb^{d}(S) \to \Sym^d(S),\ \xi \mapsto \small{\sum}_{x \in S} \len(\CO_{\xi,x}) x \]
be the Hilbert-Chow morphism.

For distinct points $x, y_1, \dots, y_{d-2} \in S$ where $d \geq 2$,
the fiber of $\rho$ over 
\[2x + \textstyle{\sum}_i y_i \in \Sym^d(S) \]
is isomorphic to $\p^1$ and called an \emph{exceptional curve}. For all $d$ define the cohomology class
\[ A \, = \, \Fp_{-2}(\pt) \Fp_{-1}(\pt)^{d-2} 1_S \,, \]
where $\pt \in H^4(S, \BZ)$ is the class of a point on $S$.
If $d \geq 2$ every exceptional curve has class $A$.

\vspace{10pt}
\noindent \textbf{(iii) The incidence subschemes}\\[4pt]
Let $z \subset S$ be a zero-dimensional subscheme.
The incidence scheme of $z$ is the locus
\begin{equation*} I(z) = \{\ \xi \in \Hilb^d(S) \ |\ z \subset \xi\ \} \end{equation*}
endowed with the natural subscheme structure.

\vspace{10pt}
\noindent \textbf{(iv) Curve classes}\\
For $\beta \in H_2(S)$ and $a,b \in H_1(S)$, define
\begin{equation} \label{xiudemjck}
\begin{alignedat}{2}
C(\beta) & = \Fp_{-1}(\beta) \Fp_{-1}(\pt)^{d-1} 1_S\ & & \in H_{2}(\Hilb^d(S)), \\
C(a, b) & = \Fp_{-1}(a) \Fp_{-1}(b) \Fp_{-1}(\pt)^{d-2} 1_S \ & & \in H_{2}(\Hilb^d(S)) \,.
\end{alignedat}
\end{equation}
In unambiguous cases, we write $\beta$ for $C(\beta)$.
By Nakajima's theorem, the assignment \eqref{xiudemjck} induces for $d \geq 2$ the isomorphism
\begin{align*}
H_2(S, \BQ) \oplus \wedge^2 H_1(S, \BQ) \oplus \BQ & \to H_2( \Hilb^d(S) ; \BQ) \\
(\beta, a \wedge b, k) & \mapsto \beta + C(a,b) + k A \,.
\end{align*}
If $d \leq 1$ and we write
\[  \beta + \sum_i C(a_i,b_i) + k A \, \in H_2(\Hilb^d(S), \BQ) \]
for some $\beta, a_i, b_i, k$, we \emph{always assume} $a_i = b_i = 0$ and $k = 0$. If $d = 0$, we also assume $\beta = 0$.
This will allow us to treat $\Hilb^d(S)$ simultaneously for all~$d$ at once, see for example Section \ref{Section_Curves_in_Hilbd}.

\vspace{10pt}
\noindent \textbf{(v) Partition cycles}\\[4pt]
Let $V \subset S$ be a subscheme, let $k \geq 1$ and consider the diagonal embedding
\[ \iota_k : S \to \Sym^k(S) \]
and the Hilbert-Chow morphism
\[ \rho : \Hilb^k(S) \to \Sym^k(S). \]
The $k$-\emph{fattening} of $V$ is the subscheme
\[ V[k] = \rho^{-1}( i_k(V) ) \subset \Hilb^k(S) \,. \]

Let $d= d_1 + \dots + d_r$ be a partition of $d$ into integers $d_i \geq 1$,
and let
\[ V_1, \dots, V_r \subset S \]
be pairwise disjoint subschemes on~$S$.
Consider the open subscheme
\begin{equation}
U = \big\{ (\xi_1, \dots, \xi_r) \in \Hilb^{d_1}(S) \times \cdots \times \Hilb^{d_r}(S) \ |\ \xi_i \cap \xi_j = \varnothing \text{ for all } i \neq j \big\} \label{111}
\end{equation}
and the natural map $\sigma : U \to \Hilb^d(S)$,
which sends a tuple of subschemes $(\xi_1, \dots, \xi_r)$
defined by ideal shaves $I_{\xi_i}$
to the subscheme $\xi \in \Hilb^d(S)$ defined by the ideal sheaf \mbox{$I_{\xi_1} \cap \dots  \cap I_{\xi_r}$}.
We often use the shorthand notation\footnote{
For functions $f_i : X \to \Hilb^{d_i}(S),i=1,\dots, r$ with $(f_1, \dots,f_r) : X \to U$ we also use
$f_1 + \ldots + f_r = \sigma \circ (f_1, \dots, f_r) : X \to \Hilb^d(S)$.
}
\begin{equation} \sigma(\xi_1, \dots, \xi_r)\, =\, \xi_1 + \dots + \xi_r. \label{116} \end{equation}
We define the \emph{partition cycle} as
\begin{equation} V_1[d_1]\, \cdots \, V_r[d_r]\, =\, \sigma(\, V_1[d_1] \times \dots \times V_r[d_r]\, ) \subset \Hilb^d(S). \label{505} \end{equation}
By \cite[Thm 9.10]{Nakajima},
the subscheme \eqref{505} has cohomology class
\[ \Fp_{-d_1}(\alpha_1) \cdots \Fp_{-d_r}(\alpha_r) 1_S \in H^{\ast}(\Hilb^d(S)), \]
where $\alpha_i = [V_i]$ for all $i$.

\subsection{Curves in $\Hilb^d(S)$} \label{Section_Curves_in_Hilbd}
\subsubsection{Cohomology classes}
Let $C$ be a projective curve and let $f : C \to \Hilb^d(S)$ be a map.
Let $p : \CZ_d \to \Hilb^d(S)$ be the universal subscheme and let $q : \CZ_d \to S$ be the universal inclusion.
Consider the fiber diagram
\begin{equation}
  \label{pullback_diagram}
\begin{tikzcd}
\widetilde{C} \ar{r}{\widetilde{f}} \ar{d}{\widetilde{p}} & \CZ_d \ar{d}{p} \ar{r}{q} & S \\
C \ar{r}{f} & \Hilb^d(S)
\end{tikzcd}
\end{equation}
and let $f' = q \circ \widetilde{f}$.
The embedded curve $\widetilde{C} \subset C \times S$ is flat of degree $d$ over $C$.
By the universal property of $\Hilb^d(S)$, we can recover $f$ from $\widetilde{C}$.
Here, even when $C$ is a smooth connected curve,
$\widetilde{C}$ could be disconnected, singular and possibly non-reduced.

\begin{lemma} \label{pullback_lemma}
 Let $C$ be a reduced projective curve and let $f : C \to \Hilb^d(S)$ be a map with
\begin{equation} f_{\ast} [C] = \beta + \sum_{j} C(\gamma_j, \gamma_j') + k A \label{xxvv} \end{equation}
for some $\beta \in H_2(S), \gamma_j, \gamma_j' \in H_1(S)$ and $k \in \BZ$. Then,
\[ ( q \circ \widetilde{f} )_{\ast} [\widetilde{C}] = \beta \,. \]
\end{lemma}
\begin{proof}
We may assume $d \geq 2$ and $C$ irreducible. Since $\widetilde{p}$ is flat,
\[ f'_{\ast} [\widetilde{C}] = f'_{\ast} \widetilde{p}^{\ast} [C] = q_{\ast} p^{\ast} f_{\ast} [C]. \]
Therefore, the claim of Lemma \ref{pullback_lemma} follows from \eqref{xxvv} and
\[ q_{\ast} p^{\ast} A = 0, \quad q_{\ast} p^{\ast} C(\beta) = \beta, \quad q_{\ast} p^{\ast} C(a,b) = 0 \]
for all $\beta \in H_2(S)$ and $a,b \in H_1(S)$.
By considering an exceptional curve of class $A$, one finds $q_{\ast} p^{\ast} A = 0$.
We will verify $q_{\ast} p^{\ast} C(\beta) = \beta$; the equation $q_{\ast} p^{\ast} C(a,b) = 0$ is similar.

Let $U \subset S^d$ be the open set defined in \eqref{111} and let $\sigma : U \to \Hilb^d(S)$ be the sum map.
We have $C(\beta) = \sigma_{\ast}( \pt^{d-1} \times \beta )$. Consider the fiber square
\[
\begin{tikzcd}
\widetilde{U} \ar{r} \ar{d}{p'} & \CZ_d \ar{d}{p} \ar{r}{q} & S \\
U \ar{r} & \Hilb^d(S) \,.
\end{tikzcd}
\]
Let $\Delta_{i,d+1} \subset S^d \times S$ be the $(i,d+1)$ diagonal.
Then $\widetilde{U} \subset S^d \times S$ is the disjoint union $\bigcup_{i=1, \dots,d} \Delta_{i,d+1} \cap (U \times S)$.
Therefore
\begin{align*} q_{\ast} p^{\ast} C(\beta)
& = q_{\ast} p^{\ast} \sigma_{\ast} (\pt^{d-1} \times \beta) \\
& = \pr_{d+1 \ast} p^{\prime \ast} (\pt^{d-1} \times \beta) \\
& = \sum_{i = 1}^{d} \pr_{d+1 \ast} ([\Delta_{i,d+1}] \cdot (\pt^{d-1} \times \beta \times e_S)) \\
& = \beta \,. \qedhere
\end{align*}
\end{proof}

\begin{lemma} \label{eulcharlemma}
Let $C$ be a smooth, projective, connected curve of genus $g$
and let $f: C \to \Hilb^d(S)$ be a map of class \eqref{xxvv}. Then
\[ k = \chi(\CO_{\widetilde{C}}) - d (1-g) \]
\end{lemma}
\begin{proof}
The intersection of $f_{\ast} [C]$ with the diagonal class $\Delta = -2 c_1( \CO_S^{[d]} )$ is $-2k$. Therefore
\[ k = \deg(c_1(\CO_{S}^{[d]}) \cap f_{\ast}[C]) = \deg( f^{\ast} \CO_{S}^{[d]} ) = \chi( f^{\ast} \CO_{S}^{[d]} ) - d (1-g), \]
where we used Riemann-Roch in the last step. Since we have
\[ f^{\ast} \CO_{S}^{[d]} = f^{\ast} p_{\ast} q^{\ast} \CO_S = \widetilde{p}_{\ast} \widetilde{f}^{\ast} q^{\ast} \CO_S = \widetilde{p}_{\ast} \CO_{\widetilde{C}} \]
and $\widetilde{p}$ is finite, we obtain $\chi( f^{\ast} \CO_{S}^{[d]} ) = \chi( \widetilde{p}_{\ast} \CO_{\widetilde{C}} ) = \chi( \CO_{\widetilde{C}} )$.
\end{proof}

\begin{cor} \label{bounded}
Let $\gamma \in H_2(\Hilb^d(S), \BZ)$ and let
$\Mbar_0(\Hilb^d(S), \gamma)$ be the moduli space of stable maps of genus $0$ in class $\gamma$.
Then for $m \ll 0$,
\[ \Mbar_0(\Hilb^d(S), \gamma + m A) = \varnothing \]
\end{cor}
\begin{proof}
Let $f : \p^1 \to \Hilb^d(S)$ be a map in class $\gamma + mA$.
The cohomology class of the corresponding curve 
$\widetilde{C} = f^{\ast} \CZ_d \subset \p^1 \times S$
is independent of $m$.
Hence, the holomorphic Euler characteristic $\chi(\CO_{\widetilde{C}})$ is bounded from below by a constant independent of $m$.
Therefore, by Lemma \ref{eulcharlemma}, we find $m$ to be bounded from below when the domain curve is $\p^1$.
Since an effective class $\gamma + m A$ decomposes in at most finitely many ways in a sum of effective classes, the claim is proven.
\end{proof}

\subsubsection{Irreducible Components} \label{irreducible_components} \label{Section_irreducible_components}
Let $f:C \to \Hilb^d(S)$ be a map and consider the fiber diagram
\[
\begin{tikzcd}
\llap{$\widetilde{C} ={}$} f^{\ast} \CZ_d \ar{r} \ar{d}{\widetilde{p}} & \CZ_d \ar{d}{p} \\
C \ar{r}{f} & \Hilb^d(S) \,,
\end{tikzcd}
\]
where  $p : \CZ_d \to \Hilb^d(S)$ is the universal family.\vspace{6pt}

\begin{definition} The map $f$ is \emph{irreducible}, if $f^{\ast} \CZ_d$ is irreducible. \end{definition}

Let $d \geq 1$ and let $f : C \to \Hilb^d(S)$ be a map from a connected non-singular projective curve $C$.
Consider the (reduced) irreducible components
\[ G_1, \dots, G_r \]
of the curve $\widetilde{C} = f^{\ast} \CZ_d$, and let
\[ \xi = \cup_{i \neq j}\, \widetilde{p}(G_i \cap G_j)\, \subset\, C \]
be the image of their intersection points under $\widetilde{p}$.
Every \emph{connected} component~$D$ of $\widetilde{C} \setminus \widetilde{p}^{-1}(\xi)$ is an irreducible curve and flat over $C \setminus \xi$.
Since $C$ is a non-singular curve, also the closure $\overline{D}$ is flat over $C$,
and by the universal property of $\Hilb^{d'}(S)$ yields an associated irreducible map
\[ C \to \Hilb^{d'}(S) \]
for some $d' \leq d$.
Let $\phi_1, \dots, \phi_r$ be the irreducible maps associated
to all connected components of $\widetilde{C} \setminus \widetilde{p}^{-1}(\xi)$.
We say $f$ \emph{decomposes into the irreducible components} $\phi_1, \dots, \phi_r$.

Conversely, let $\phi_i : C \to \Hilb^{d_i}(S), i=1,\dots,n$
be irreducible maps with
\begin{itemize}
 \item $\sum_i d_i = d$,
 \item $\phi_i^{\ast} \CZ_{d_i} \cap \phi_{j}^{\ast} \CZ_{d_j}$ is of dimension $0$ for all $i \neq j$.
\end{itemize}
Let $U$ be the open subset defined in \eqref{111}. The map 
\[ (\phi_1, \dots, \phi_n) : C \ra \Hilb^{d_1}(S) \times \dots \times \Hilb^{d_n}(S) \]
meets the complement of $U$ in a finite number of points $x_1, \dots, x_m \in C$.
By smoothness of $C$, the composition 
\[ \sigma \circ (\phi_1, \dots, \phi_n) : C \setminus \{ x_1, \dots, x_m \} \ra \Hilb^d(S) \]
extends uniquely to a map $f : C \to \Hilb^d(S)$.

A direct verification shows that the two operations above are inverse to each other. We write
\[ f = \phi_1 + \dots + \phi_r \]
for the decomposition of $f$ into the irreducible components $\phi_1, \dots, \phi_r$.\vspace{5pt}

Let $\beta, \beta_i \in H_2(S)$, $\gamma_{j}, \gamma'_{j}, \gamma_{i,j}, \gamma'_{i,j} \in H_1(S)$ and $k, k_i \in \BZ$ such that
\begin{align*}
 f_{\ast}[C] & = C(\beta) + \sum_j C(\gamma_j, \gamma'_j) + k A\ \in H_{2}(\Hilb^d(S))\\
 \phi_{i \ast}[C] & = C(\beta_i) + \sum_j C(\gamma_{i,j}, \gamma'_{i,j}) + k_i A\ \in H_2(\Hilb^{d_i}(S)) \,.
\end{align*}
\begin{lemma} \label{hom_add_up} We have
\begin{itemize}
 \item $\sum_i \beta_i = \beta \, \in H_2(S;\BZ)$
 \item $\sum_{i,j} \gamma_{i,j} \wedge \gamma'_{i,j} = \sum_j \gamma_j \wedge \gamma'_j\, \in \bigwedge^2 H_1(S)$ \,.
\end{itemize}
\end{lemma}
\begin{proof}
This follows directly from \cite[Theorem 9.10]{Nakajima}.
\end{proof}

\section{The Yau-Zaslow formula in higher dimensions} \label{YZ_formula_in_higher_dim} \label{Section_higher_dimensional_Yau_Zaslow}

\subsection{Overview}
In the remainder of section \ref{YZ_formula_in_higher_dim} we give a proof of Theorem \ref{MThm0}.
The proof proceeds in the following steps.

In section \ref{sec_ellK3case} we use the deformation theory of K3 surfaces
to reduce Theorem \ref{MThm0} to an evaluation on a specific elliptic K3 surface $S$.
Here, we also analyse rational curves on $\Hilb^d(S)$ and prove a few Lemmas.
This discussion will be used also later on.

In section \ref{basic_case}, we study the structure of the moduli space of stable maps which are incident to the Lagrangians $L_{z_1}$ and $L_{z_2}$.
The main result is a splitting statement (Proposition~\ref{W0splitprop}),
which reduces the computation of Gromov-Witten invariants
to integrals associated to fixed elliptic fibers.

In section \ref{section_Kummer_evaluation}, we evaluate these remaining integrals
using the geometry of the Kummer K3 surfaces, the Yau-Zaslow formula
and a theta function associated to the $\mathsf{D}_4$ lattice.

\subsection{The Bryan-Leung K3} \label{sec_ellK3case}
\subsubsection{Definition} \label{Section_BLK3_Defn}
Let $\pi : S \to \p^1$ be an elliptic K3 surface
with a unique section $s : \p^1 \to S$ and 24 rational nodal fibers.
We call $S$ a \emph{Bryan-Leung K3 surface}.

Let $x_1, \dots, x_{24} \in \p^1$ be the basepoints of the nodal fibers of $\pi$,
let $B_0$ be the image of the section $s$, and let
\[ F_x \subset S \]
denote the fiber of $\pi$ over a point $x \in \p^1$.

The Picard group 
\[ \Pic(S) = H^{1,1}(S;\BZ) = H^{2}(S;\BZ) \cap H^{1,1}(S;\BC) \]
is of rank $2$ and generated by the section class $B$ and the fiber class $F$.
We have the intersection numbers
$B^2 = -2$,\ $B \cdot F = 1$ and $F^2 = 0$.
Hence for all $h \geq 0$ the class 
\begin{equation} \beta_h = B + hF \in H_2(S;\BZ) \label{BL_betah} \end{equation}
is a primitive and effective curve class of square $\beta_h^2 = 2h - 2$.

The projection $\pi$ and the section $s$ induce maps of Hilbert schemes
\[ \pi^{[d]} : \Hilb^d(S) \to \Hilb^d(\p^1) = \p^d, \quad \quad s^{[d]} : \p^d \to \Hilb^d(S), \]
such that $\pi^{[d]} \circ s^{[d]} = \id_{\p^d}$.
The map $s^{[d]}$ is an isomorphism
from $\Hilb^d(\p^1)$ to the locus of subschemes of $S$, which are contained in $B_0$.
This gives natural identifications
\[ \p^d = \Hilb^d(\p^1) = \Hilb^d( B_0 ) \,, \]
that we will use sometimes. In unambiguous cases we also write $\pi$~and~$s$ for $\pi^{[d]}$ and $s^{[d]}$ respectively.

\subsubsection{Main statement revisited} \label{K3statements}
For $d \geq 1$ and cohomology classes $\gamma_1, \dots, \gamma_m \in H^{\ast}(\Hilb^d(S);\BQ)$ define the quantum bracket
\begin{equation*}
 \big\langle \gamma_1, \dots, \gamma_m \big\rangle_q^{\Hilb^d(S)}
= \sum_{h \geq 0} \sum_{k \in \BZ} y^k q^{h-1} \blangle \gamma_1, \dots, \gamma_m \brangle^{\Hilb^d(S)}_{\beta_h + k A} \,,
\end{equation*}
where the bracket on the right hand side was defined in \eqref{bbbm}.

\begin{thm} \label{ellthm} For all $d \geq 1$,
\[ \blangle \Fp_{-1}(F)^d 1_S \ , \ \Fp_{-1}(F)^d 1_S \brangle^{\Hilb^d(S)}_q \ =\ \frac{F(z,\tau)^{2d-2}}{\Delta(\tau)}, \]
where $q = e^{2 \pi i \tau}$ and $y = -e^{2 \pi i z}$.   
\end{thm}
We begin the proof of Theorem \ref{ellthm} in Section \ref{basic_case}.

Let $\pi' : S' \to \p^1$ be any elliptic K3 surface, and let $F'$ be the class of a fiber of $\pi'$. A fiber
of the induced Lagrangian fibration
\[ \pi^{\prime [d]} : \Hilb^d(S') \to \p^d \]
has class $\Fp_{-1}(F')^d 1_S$. Hence, Theorem~\ref{MThm0} implies Theorem~\ref{ellthm}.
The following Lemma shows that conversely Theorem~\ref{ellthm} also implies Theorem~\ref{MThm0}, and hence
the claims in both Theorems are equivalent.

\begin{lemma} \label{fijerfmimv}
Let $S$ be the fixed Bryan-Leung K3 surface defined in Section~\ref{Section_BLK3_Defn},
and let $\beta_h = B + hF$ be the curve class \eqref{BL_betah}.

Let $S'$ be a K3 surface with a primitive curve class $\beta$ of square $2h-2$, and let $\gamma \in H^2(S', \BZ)$
be any class with $\beta \cdot \gamma = 1$ and $\gamma^2 = 0$. Then
\[
\blangle  \Fp_{-1}(\gamma)^d 1_{S'}\, ,\, \Fp_{-1}(\gamma)^d 1_{S'} \brangle_{\beta + kA}^{\Hilb^d(S')}
=
\blangle \Fp_{-1}(F)^d 1_S \, , \, \Fp_{-1}(F)^d 1_S \brangle^{\Hilb^d(S)}_{\beta_h + k A} \,.
\]
\end{lemma}

\begin{proof}[Proof of Lemma \ref{fijerfmimv}]
We will construct an algebraic deformation from $S'$
to the fixed K3 surface $S$ such that $\beta$ deforms to $\beta_h$ through classes of Hodge type $(1,1)$, and $\gamma$ deforms to $F$.
By the deformation invariance of reduced Gromov-Witten invariants
the claim of Lemma \ref{fijerfmimv} follows.

Let $E_8(-1)$ be the negative $E_8$ lattice, let $U$ be the hyperbolic lattice
and consider the K3 lattice
\[ \Lambda = E_8(-1)^{\oplus 2} \oplus U^{\oplus 3} \,. \]
Let $e,f$ be a hyperbolic basis for one of the $U$ summands of $\Lambda$ and let 
\[ \phi: \Lambda \overset{\cong}{\longrightarrow} H^{2}(S;\BZ) \]
be a fixed marking with $\phi(e) = B + F$ and $\phi(f) = F$. We let 
\[ b_h = e + (h-1) f \]
denote the class corresponding to $\beta_h = B + hF$ under $\phi$.

The orthogonal group of $\Lambda$ is transitive on primitive vectors of the same square, see \cite[Lemma 7.8]{GH} for references.
Hence there exists a marking 
\[ \phi' : \Lambda \overset{\cong}{\longrightarrow} H^2(S';\BZ) \]
such that $\phi'(b_h) = \beta$.
Let $g = \phi^{\prime -1}(\gamma) \in \Lambda$
be the vector that corresponds to the class $\gamma$ under $\phi'$.
The span 
\[ \Lambda_0 = \langle g, b_h \rangle \subset \Lambda \] 
defines a hyperbolic sublattice of $\Lambda$ which, by unimodularity, yields the direct sum decomposition
\[ \Lambda = \Lambda_0 \oplus \Lambda_0^{\perp}. \]
Because the irreducible unimodular factors of a unimodular lattice are unique up to order, we find
\[ \Lambda_0^{\perp} \cong E_8(-1)^{\oplus 2} \oplus U^{\oplus 2} \,. \]
Hence there exists a lattice isomorphism $\sigma : \Lambda \to \Lambda$ with $\sigma(b_h) = b_h$ and $\sigma(g) = f$.
Replacing $\phi'$ by $\phi' \circ \sigma^{-1}$, we may therefore assume $\phi'(b_h) = \beta$ and $\phi'(f) = \gamma$.

Since the period domain $\Omega$ associated to $b_h$ is connected,
there exists a curve inside $\Omega$ connecting the period point of $S'$ to the period point of $S$.
Restricting the universal family over $\Omega$ to this curve, we obtain a deformation with the desired properties.
\end{proof}

\subsubsection{Rational curves in $\Hilb^d(S)$} \label{Section_Rational_curves}
Let $h \geq 0$ and let $k$ be an integer.
We consider rational curves on $\Hilb^d(S)$
in the classes $\beta_h + kA$ and $h F + kA$.

\vspace{10pt}
\noindent \textbf{Vertical maps}\\[4pt]
Let $u_1, \dots, u_d \in \p^1$ be points such that
\begin{itemize}
 \item $u_i$ is not the basepoint of a nodal fiber of $\pi : S \to \p^1$ for all $i$,
 \item the points $u_1, \dots, u_n$ are pairwise distinct.
\end{itemize}
Then, the fiber of $\pi^{[d]}$ over
$u_1 + \dots + u_d \in \Hilb^d(\p^1)$ is isomorphic to
the product of smooth elliptic curves
\[ F_{u_1} \times \ldots \times F_{u_d} \,. \]
The subset of points in $\Hilb^d(\p^1)$ whose preimage under $\pi^{[d]}$ is not of this form is the divisor
\begin{equation} \pidiscr = I(x_1) \cup \dots I(x_{24}) \cup \Delta_{\Hilb^d(\p^1)} \subset \Hilb^d(\p^1), \label{501} \end{equation}
where $x_1, \dots, x_{24}$ are the basepoints of the nodal fibers of $\pi$,
$I(x_i)$ is the incidence subscheme, and $\Delta_{\Hilb^d(\p^1)}$ is the diagonal, see Section \ref{special_cycles}.
Since a fiber of $\pi^{[d]}$ over a point $z \in \p^d$ is non-singular
if and only if $z \notin \pidiscr$, we call $\pidiscr$ the \emph{discriminant} of $\pi^{[d]}$.

Consider a stable map 
$f : C \to \Hilb^d(S)$
of genus $0$ and class $h F + kA$.
Since the composition
\[ \pi^{[d]} \circ f : C \to \Hilb^d(\p^1) \]
is mapped to a point, and since non-singular elliptic curves
do not admit non-constant rational maps, we have the following Lemma.

\begin{lemma} \label{200}
Let $f : C \to \Hilb^d(S)$ be a non-constant genus $0$ stable map in class $h F + kA$. 
Then the image of $\pi^{[d]} \circ f$ lies in the discriminant $\pidiscr$.
\end{lemma}

\vspace{10pt}
\noindent \textbf{Non-vertical maps}\\[4pt]
Let $f : C \to \Hilb^d(S)$ be a stable genus $0$ map in class 
$f_{\ast} [C] = \beta_h + kA$.
The composition 
\[ \pi^{[d]} \circ f : C \ra \p^d \]
has degree $1$ with image a line 
\[ L \subset \p^d. \]
Let $C_0$ be the unique irreducible component of $C$
on which $\pi \circ f$ is non-constant.
We call
$C_0 \subset C$
the \emph{distinguished component} of~$C$.

Since $C_0 \cong \p^1$, we have a decomposition
\begin{equation*} f|_{C_0} = \phi_0 + \dots + \phi_r \label{115} \end{equation*}
of $f|_{C_0}$ into irreducible maps
$\phi_i : C_0 \to \Hilb^{d_i}(S)$
where $d_i$ are positive integers such that $d = d_0 + \dots + d_r$, see Section \ref{irreducible_components}.
By Lemma~\ref{hom_add_up},
exactly one of the maps $\pi^{[d_i]} \circ \phi_i$ is non-constant;
we assume this map is $\phi_0$.

\begin{lemma} \label{117} Let $\pidiscr$ be the discriminant of $\pi^{[d]}$.
If $L \nsubseteq \pidiscr$, then
\begin{itemize}
 \item[(i)] $d_i = 1$ for all $i \in \{1,\dots, r\}$,
 \item[(ii)] $\phi_i : C_0 \to S$ is constant for all $i \in \{ 1,\dots,r \}$,
 \item[(iii)] $\phi_0 : C_0 \to \Hilb^{d_0}(S)$ is an isomorphism onto a line in $\Hilb^{d_0}(B_0)$.
\end{itemize}
\end{lemma}

\begin{proof} Assume $L \nsubseteq \pidiscr$.
\begin{enumerate}
\item[(i)] If $d_i \geq 2$, then $\pi^{[d_i]} \circ \phi_i$ maps $C_0$ into $\Delta_{\Hilb^{d_i}(\p^1)}$.
Hence
\[ \pi^{[d]} \circ f = \textstyle{\sum}_i \pi^{[d_i]} \circ \phi_i \]
maps $C_0$ into $\Delta_{\Hilb^d(\p^1)} \subset \CW$.
Since $L = \pi^{[d]} \circ f(C_0)$, we find $L \subset \CW$, which is a contradiction.

\item[(ii)] If $\phi_i : C_0 \to S$ is non-constant,
then $\pi \circ \phi_i$
maps $C_0$
to a basepoint of a nodal fiber of $\pi : S \to \p^1$.
By an argument identical to (i) this implies $L \subset \CW$, which is a contradiction. Hence, $\phi_i$ is constant.

\item[(iii)]
The universal family of curves on the elliptic K3 surface $\pi : S \to \p^1$ in class $\beta_{h} = B + hF$
is the $h$-dimensional linear system 
\[ |\beta_h| = \Hilb^h(\p^1) = \p^h \,. \]
Explicitly, an element $z \in \Hilb^h(\p^1)$ corresponds to the comb curve
\begin{equation} B_0 + \pi^{-1}(z) \subset S \,, \label{307} \end{equation}
where $\pi^{-1}(z)$ denotes the fiber of $\pi$ over the subscheme $z \subset \p^1$.

Let $\CZ_d \to \Hilb^d(S)$ be the universal family
and consider the fiber diagram
\[
\begin{tikzcd}
\widetilde{C_0} \ar{r}{\widetilde{f}} \ar{d}{\widetilde{p}} & \CZ_d \ar{d}{p} \ar{r}{q} & S \\
C_0 \ar{r}{f} & \Hilb^d(S) \,.
\end{tikzcd}
\]

By Lemma \ref{pullback_lemma},
the map $f' = q \circ \widetilde{f} : \widetilde{C_0} \to S$
is a curve in the linear system $|\beta_{h'}|$
for some $h' \leq h$.
Its image is therefore a comb of the form \eqref{307}.

Let $G_0$ be the irreducible component of $\widetilde{C_0}$ such that $\pi \circ f'|_{G_0}$ is non-constant.
The restriction
\begin{equation} \widetilde{p}|_{G_0} : G_0 \to C_0 \label{fvmskfmvdf} \end{equation}
is flat.
Since $\pi \circ f' : \widetilde{C}_0 \to \p^1$ has degree~$1$,
the curve $\widetilde{C}_0$ has multiplicity~$1$ at $G_0$,
and the map to the Hilbert scheme of $S$ associated to \eqref{fvmskfmvdf} is equal to $\phi_0$.

Since $G_0$ is reduced and $f'|_{G_0} : G_0 \to S$ maps to $B_0$,
the map $\phi_0$ maps with degree $1$ to $\Hilb^{d_0}(B_0)$.
The proof of (iii) is complete. \qedhere
\end{enumerate}
\end{proof}

\vspace{10pt}
\noindent \textbf{The normal bundle of a line}\\[4pt]
Let $s^{[d]} : \Hilb^d(\p^1) \hookrightarrow \Hilb^d(S)$ be the section, and consider the normal bundle
\[ N = s^{[d] \ast} T_{\Hilb^d(S)} \big/ T_{\Hilb^d(\p^1)}. \]
\begin{lemma} \label{normaltoL} For every line $L \subset \Hilb^d(\p^1)$,
\[ T_{\Hilb^d(S)}\big|_{L} = T_{\Hilb^d(\p^1)}\big|_{L} \oplus N\big|_{L} \]
with $N\big|_{L} = T^{\vee}_{\Hilb^d(\p^1)}\big|_{L} = \CO_L(-2) \oplus \CO_L(-1)^{\oplus (d-1)}$.
\end{lemma}
\begin{proof}
Because the embedding $s^{[d]} : \Hilb^d(\p^1) \hookrightarrow \Hilb^d(S)$
has the right inverse $\pi^{[d]}$, the restriction
\[ T_{\Hilb^d(S)} \big|_{\Hilb^d(\p^1) } \]
splits as a direct sum of the tangent and normal bundle of $\Hilb^d(B_0)$.

The vanishing $H^0( \p^d, \Omega_{\p^d}^2 ) = 0$
implies that the holomorphic symplectic form on $\Hilb^d(S)$
restricts to $0$ on $\Hilb^d(\p^1)$ and hence, by non-degeneracy,
induces an isomorphism
\[ T_{\Hilb^d(\p^1)} \to N^{\vee} \,. \]
Since $T_{\Hilb^d(\p^1)}\big|_{L} = \CO_L(1)^{\oplus (d-1)} \oplus \CO_L(2)$,
the proof is complete.
\end{proof}

\subsection{Analysis of the moduli space} \label{basic_case}
\subsubsection{Overview} \label{Section_basic_case_overview}
Let $S$ be the fixed elliptic Bryan-Leung K3 surface,
let $z_1, z_2 \in \Hilb^d(\p^1)$ be generic points,
and for $i \in \{1,2\}$ let
\[ Z_{i} = \pi^{[d] -1}(z_i) \subset \Hilb^d(S) \]
be the fiber of $\pi^{[d]}$ over $z_i$. The subscheme $Z_i$ has class $[Z_i] = \Fp_{-1}(F)^d 1_S$.
Let
\[ \ev : \Mbar_{0,2}( \Hilb^d(S), \beta_h + kA) \ra \Hilb^d(S) \times \Hilb^d(S) \]
be the evaluation map from the moduli space of genus $0$ stable maps in class $\beta_h = B + hF$,
and define the moduli space
\begin{equation*} M_Z = M_Z(h,k) = \ev^{-1}(Z_1 \times Z_2) \label{cvcvdfvf} \end{equation*}
parametrizing maps which are incident to $Z_1$ and $Z_2$.

In Section \ref{basic_case}, we begin the proof of Theorem \ref{ellthm}
by studying the moduli space $M_Z$ and its virtual class.
First, we prove that $M_Z$ is naturally isomorphic to a product of
moduli spaces associated to specific fibers of the elliptic fibration $\pi \colon S \to \p^1$.
Second, we show that the virtual class splits as a product of virtual classes on each factor.
Both results are summarized in Proposition~\ref{W0splitprop}.
As a consequence, Theorem \ref{ellthm} is reduced
to the evaluation of a series $F^{\textup{GW}}(y,q)$ encoding integrals
associated to specific fibers of $\pi$.

\subsubsection{The set-theoretic product} \label{basic_case_settheoretic_splitting} \label{Section_settheoretic_splitting}
Consider a stable map
\[ [f : C \to \Hilb^d(S), p_1, p_2]\, \in\, M_Z \]
with markings $p_1,p_2 \in C$.
By definition of $M_Z$, we have
\[ \pi^{[d]}(f(p_1)) = z_1, \quad \quad \pi^{[d]}(f(p_2)) = z_2 .\]
Hence, the image of $C$ under $\pi^{[d]} \circ f$ is the unique line
\[ L \subset \p^d \]
incident to the points $z_1, z_2 \in \p^d$.
Because $z_1, z_2 \in \p^d$ are generic, also $L$ is generic.
In particular, since $z_1 \cap z_2 = \varnothing$, we have
\begin{equation} L \nsubseteq I(x) \quad \text{ for all } x \in \p^1 . \label{Lnon_deg}\end{equation}

Let $C_0$ be the distinguished irreducible component of $C$ on which $\pi \circ f$ is non-constant.
By \eqref{Lnon_deg}, the restriction $f|_{C_0}$ is irreducible, and by Lemma~\ref{117}~(iii), the map $f|_{C_0}$ is
an isomorphism onto the embedded line
\[ L \subset \Hilb^d(\p^1) \overset{s}{\subset} \Hilb^d(S). \]
We will identify $C_0$ with $L$ via this isomorphism.

Let $x_1, \dots, x_{24} \in \p^1$ be the basepoints of the nodal fibers of $\pi$, and let 
\[ y_1, \dots, y_{2d-2} \in \p^1 \]
be the points such that $2 y_i \subset z$ for some $z \in L$.
For $x \in \p^1$, let 
\[ \widetilde{x} = I(x) \cap L \ \in \ \Hilb^d(\p^1) \]
denote the unique point on $L$ which is incident to $x$.
Then, the points 
\begin{equation} \label{L_with_W_int_points} \widetilde{x}_1, \dots, \widetilde{x}_{24}, \widetilde{y}_1, \dots, \widetilde{y}_{2d-2} \end{equation}
are the intersection points of $L$ with
the discriminant of $\pi^{[d]}$ defined in \eqref{501}.
Hence, by Lemma \ref{200}, components of $C$ can be attached to $C_0$ only
at the points~\eqref{L_with_W_int_points}. 
Consider the decomposition
\begin{equation} C = C_0 \cup A_1 \cup \dots \cup A_{24} \cup B_1 \cup \dots \cup B_{2d-2}, \label{ogorgdfgfdg} \end{equation}
where $A_i$ and $B_j$ are the components of $C$ attached to the points
$\widetilde{x}_i$ and $\widetilde{y}_j$ respectively.
We consider the restriction of $f$ to $A_i$ and $B_j$ respectively.

\begin{enumerate}
 \item[$A_i$:] Let $\widetilde{x_i} = x_i + w_1 + \dots + w_{d-1}$ for some points $w_\ell \in \p^1$.
By genericity of~$L$, the $w_\ell$ are basepoints of smooth elliptic fibers.
Hence, $f|_{A_i}$ decomposes as
\begin{equation} f|_{A_i} = \phi + w_1 + \dots + w_{d-1}, \label{201} \end{equation}
where $w_\ell \in \p^1 \subset S$ for all $\ell$ denote constant maps,
and $\phi : A_i \to F_{x_i}$ is a map to $i$-th nodal fiber which sends $\tilde{x}_i$ to the point $s(x_i) \in S$.

\item[$B_j$:] Let $\widetilde{y_j} = 2 y_j + w_1 + \dots + w_{d-2}$ for some points $w_\ell \in \p^1$.
Then, $f|_{B_j}$ decomposes as
\begin{equation} f|_{B_j} = \phi + w_1 + \dots + w_{d-2}, \label{212} \end{equation}
where $\phi : B_j \to \Hilb^2(S)$ maps to the fiber $(\pi^{[2]})^{-1}(2y)$ and sends the point $\widetilde{y}_j \in L \equiv C_0$ to $s(2 y_j)$.
\end{enumerate}

Since $L$ is independent of $f$, we conclude that the moduli space $M_Z$ is \emph{set-theoretically}\footnote{i.e. the set of $\BC$-valued points of $M_Z$ is a product}
a product of moduli spaces of maps of the form $f|_{A_i}$ and $f|_{B_j}$.
The next step is to prove the splitting is \emph{scheme-theoretic}.

\subsubsection{Deformation theory} \label{defsplii}
Let $[f : C \to \Hilb^d(S),p_1, p_2] \in M_Z$ be a point and let
\begin{equation}
 \xymatrix{
\widehat{C} \ar@<+6pt>[d]^p \ar[r]^{\widehat{f}} & \Hilb^d(S) \\
\Spec( \BC[\epsilon]/\epsilon^2 ) \ar@<+0pt>[u] \ar@<+6pt>[u]^{\widehat{p_1}, \widehat{p}_2}
}
\label{221}
\end{equation}
be a first order deformation of $f$ inside $M_Z$.
In particular, $p$ is a flat map, $\widehat{p}_1, \widehat{p}_2$ are sections of $p$,
and $\widehat{f}$ restricts to $f$ at the closed point.

Consider the decomposition \eqref{ogorgdfgfdg} and let $\widetilde{x}_i$ for $i = 1,\dots, 24$ and $\widetilde{y}_j$ for $j =  1,\dots, 2d-2$
be the node points $A_i \cap C_0$ and $B_j \cap C_0$ respectively.
\begin{lemma} \label{doesnotsmooth}
The deformation \eqref{221} does not resolve the nodal points 
$\widetilde{x}_1, \dots, \widetilde{x}_{24}$ and $\widetilde{y}_1, \dots, \widetilde{y}_{2d-2}$.
\end{lemma}

\begin{proof}
Assume $\widehat{f}$ smoothes the node $\widetilde{x}_i$ for some $i$.
Let $\CZ_d \to \Hilb^d(S)$ be the universal family and consider the pullback diagram
\[
\begin{tikzcd}
\mathllap{f^{\ast} \CZ_d = }\ \widetilde{C} \ar{r} \ar{d} & \CZ_d \ar{d} \ar{r} & S \\
C \ar{r}{f} & \Hilb^d(S)
\end{tikzcd}
\]
Let $E$ be the connected component of $f|_{A_i}^{\ast} \CZ_d$,
which defines the non-constant map $\phi$ in the decomposition \eqref{201},
and let $G_0 = f|_{C_0}^{\ast} \CZ_d$.
Then, the projection $\widetilde{C} \to C$ is \'etale at the intersection point $q = G_0 \cap E$,

The deformation $\widehat{f} : \widehat{C} \to \Hilb^d(S)$ induces the deformation 
\[ K = \widehat{f}^{\ast} \CZ_d \ra \Spec \bigl( \BC[\epsilon]/\epsilon^2 \bigr) \]
of the curve $\widetilde{C}$.
Since $\widehat{C}$ smoothes $\widetilde{x}_i$ and $\widetilde{C} \to C$ is \'etale near $q$, the deformation~$K$ resolves~$q$.
Then, the natural map $K \to S$ defines a deformation of the curve $\widetilde{C} \to S$ which resolves $q$.
Since $\widetilde{C} \to S$ has class $\beta_h$, such a deformation can not exist by the geometry of the linear system $|\beta_h|$.
Hence, $\widehat{f}$ does not smooth the node $\widetilde{x}_i$.

Assume $\widehat{f}$ smoothes the node $\widetilde{y}_j$ for some $j$.
We follow closely the argument of T.~Graber in \cite[page 19]{Grab}.
Let $F_{y_j}$ be the fiber of $\pi : S \to \p^1$ over~$y_j$, let
\[ D(F_{y_j}) = \{ \xi \in \Hilb^d(S)\ |\ \xi \cap F_{y_j} \neq 0 \} \]
be the divisor of subschemes with non-zero intersection with $F_{y_j}$,
and consider the divisor
\[ D = \Delta_{\Hilb^d(S)} + D(F_{y_j}) \,. \]
Let $C_1$ be the irreducible component of $C$ that attaches to $C_0$ at $q = \widetilde{y}_j$,
and let $C_2$ be the union of all irreducible components of $B_j$ except~$C_1$.
The curves $C_2$ and $C_1$ intersect in a finite number of nodes $\{ q_i \}$.
The deformation $\widehat{f}$ resolves the node $q$ and may also resolve some of the $q_i$.

The first order neighborhood $\widetilde{C_1}$ of $C_1$
in the total space of the deformation $\widehat{C}$
can be identified with the first order neighborhood of $\p^1$ in the total space of the bundle $\CO(- \ell)$,
where $\ell \geq 1$ is the number of nodes on $C_1$ which are smoothed by $\widehat{f}$.
Let
\[ f' : \widetilde{C}_1 \to \Hilb^d(S) \]
be the induced map on $\widetilde{C}_1$.
We consider the case, where $f'|_{C_1}$ is a degree $k \geq 1$ map to the exceptional curve at $\widetilde{y}_j$.
The general case is similar.

Let $N$ be the pullback of $\CO(D)$ by $f' : \widetilde{C}_1 \to \Hilb^d(S)$,
and let $s \in H^0(\widetilde{C}, N)$ be the pullback of the section of $\CO(D)$ defined by $D$.
The bundle $N$ restricts to $\CO(-2k)$ on $C_1$.
By \cite[page 20]{Grab}, giving $N$ and $s$ is equivalent to an element of the vector space
\[ \Hom_{\CO_{C_1}}(\CO(-\ell), f|_{C_1}^{\ast} \CO(D)), \]
of dimension $\ell-2k+1 \leq \ell-1$.

The neighborhood $\widetilde{C_1}$ intersects $C_0$ in a double point.
Since $C_0$ intersects the divisor $D$ transversely, $s$ is non-zero on $\widetilde{C}_1$.
Let $q_1, \dots, q_{\ell-1}$ be the other nodes on $C_1$ which get resolved by $\widehat{f}$.
Since $C_2 \subset D$, the section $s$ vanishes at $q_1, \dots, q_{\ell-1}$.
By dimension reasons, we find $s = 0$. This contradicts the non-vanishing of $s$. 
Hence, $\widehat{f}$ does not smooth the node $\widetilde{y}_j$.
\end{proof}

By Lemma \ref{doesnotsmooth}, any first order (and hence any infinitesimal)
deformation of $[f : C \to \Hilb^d(S), p_1, p_2] \in M_Z$ inside $M$
preserves the decomposition
\[ C = C_0 \cup_i A_i \cup_j B_j \]
and therefore induces a deformation of the restriction
\begin{equation} f|_{C_0} : C_0 \overset{\cong}{\ra} L \subset \Hilb^d(S). \label{frmgerg} \end{equation}
By Lemma \ref{normaltoL}, every deformation of $L \subset \Hilb^d(S)$
moves the line $L$ in the projective space $\Hilb^{d}(B_0)$.
Since any deformations of $f$ inside $M_Z$ must stay incident to $Z_1, Z_2 \subset \Hilb^d(S)$,
we conclude that such deformations induce the constant deformation of \eqref{frmgerg}.
The image line $f(C_0)$ stays completely fixed.

\subsubsection{The product decomposition} \label{splitting} \label{Section_splitting}
For $h > 0$ and for $x \in \p^1$ a basepoint of a nodal fiber of $\pi : S \to \p^1$, let
\[ M^{\textup{(N)}}_{x}(h) \]
be the moduli space of $1$-marked genus $0$ stable maps to $S$ in class $h F$
which map the marked point to $x$.
Hence, $M^{(N)}_{x}(h)$ parametrizes degree~$h$ covers of the nodal fiber $F_{x}$.
By convention, $M^{(N)}_{x}(0)$ is taken to be a point. 

For $h \geq 0,\, k \in \BZ$ and for $y \in \p^1$ a basepoint of a smooth fiber of $\pi$, let
\begin{equation} M^{\textup{(F)}}_{y}(h,k) \label{234_1} \end{equation}
be the moduli space of $1$-marked genus $0$ stable maps to $\Hilb^2(S)$ in class $h F + k A$
which map the marked point to $s^{[2]}(2 y)$.
By convention, $M^{(F)}_{y}(0,0)$ is taken to be a point.

Let $T$ be a connected scheme and consider a family
\begin{equation}
\begin{tikzcd}
C \ar{r}{F} \ar{d} & \Hilb^d(S) \\ T
\end{tikzcd}
\label{fam_dekfodkfo}
\end{equation}
of stable maps in $M_Z$.
By Lemma \ref{doesnotsmooth}, the curve $C \to T$ allows a decomposition
\[ C = C_0 \cup A_1 \cup \dots \cup A_{24} \cup B_1 \cup \dots \cup B_{2d-2}, \]
where $C_0$ is the distinguished component of $C$
and the components $A_i$ and $B_j$ are attached to $C_0$ at the points $\tilde{x}_i$ and $\tilde{y}_j$ respectively.

The restriction of the family \eqref{fam_dekfodkfo} to the components $A_i$ (resp. $B_j$)
defines a family in the moduli space
$M^{\textup{(N)}}_{x_i}(h_{x_i})$ (resp. $M^{\textup{(F)}}_{y_j}(h_{y_j}, k_{y_j})$)
for some $h_{x_i}$ (resp. $h_{y_j}, k_{y_j}$).
Since, by Section \ref{Section_Curves_in_Hilbd}, the line $f(C_0) = L$ has class
\begin{equation*} [L] = B - (d-1) A \ \in H_2( \Hilb^d(S), \BZ), \end{equation*}
and by the additivity of cohomology classes under decomposing (Lemma~\ref{hom_add_up}),
we must have $\sum_i h_{x_i} + \sum_j h_{y_j} = h$ and $\sum_j k_{y_j} = k + (d-1)$.
Let
\begin{equation}
\Psi: M_{Z}
\ra
\bigsqcup_{\textbf{h}, \textbf{k}} \bigg( \prod_{i=1}^{24} M^{\textup{(N)}}_{x_i}(h_{x_i})
\times \prod_{j=1}^{2d-2} M^{\textup{(F)}}_{y_j}(h_{y_j}, k_{y_j}) \bigg)
\label{spl1}. \end{equation}
be the induced map on moduli spaces,
where the disjoint union runs over all 
\begin{equation}
\label{spl1_indexing}
\begin{aligned}
\textbf{h} & = (h_{x_1}, \dots, h_{x_{24}}, h_{y_1}, \dots, h_{y_{2d-2}} ) \in (\BN^{\geq 0})^{ \{ x_i, y_j \}} \\
\textbf{k} & = (k_{y_1}, \dots, k_{y_{2d-2}}) \in \BZ^{2d-2}
\end{aligned}
\end{equation}
such that 
\begin{equation} \sum_i h_{x_i} + \sum_j h_{y_j} = h \quad \text{ and } \quad \sum_j k_{y_j} = k + (d-1) \,. \label{spl2_indexing} \end{equation}
Since $L \subset \Hilb^d(S)$ is fixed under deformations,
we can glue elements of the right hand side of \eqref{spl1} to $C_0$
and obtain a map in $M_{Z}$.
By a direct verification, the induced morphism on moduli spaces is the inverse to $\Psi$.
Hence, $\Psi$ is an isomorphism.

\subsubsection{The virtual class} \label{anvirclass} \label{Section_analysis_of_virtual_class}
Let $Z_1, Z_2$ be the Lagrangian fibers of $\pi^{[d]}$ defined in Section \ref{Section_basic_case_overview},
and let $Z = Z_1 \times Z_2$. We consider the fiber square
\begin{equation} \label{iweufruf}
\begin{tikzcd}
M_Z \ar{r}{j} \ar{d}{p} & M \ar{d}{\ev} \\
Z \ar{r}{i} & ( \Hilb^d(S) )^2,
\end{tikzcd}
\end{equation}
where $M = \Mbar_{0,2}( \Hilb^d(S), \beta_h + kA)$.
The map $i$ is the inclusion of a smooth subscheme of codimension $2d$.
Hence, the restricted virtual class
\begin{equation} [M_Z]^{\text{vir}} = i^{!} [M]^{\text{red}} \label{204} \end{equation}
is of dimension $0$.
By the push-pull formula we have
\begin{equation}
\int_{[M_Z]^{\text{vir}}} 1\ =\ \big\langle \Fp_{-1}(F)^d 1_S , \Fp_{-1}(F)^d 1_S  \big\rangle^{\Hilb^d(S)}_{\beta_h+kA} \,.
\label{virt_pushpull_mddfds}
\end{equation}
Let $\Psi$ be the splitting morphism \eqref{spl1}.
We will show that $\Psi_{\ast} [ M_Z ]^{\text{vir}}$
splits naturally as a product of virtual cycles.

Let $\BL_X$ denote the cotangent complex on a space $X$.
Let
$E^{\bullet} \to \BL_{M}$
be the reduced perfect obstruction theory on $M$,
and let $F^{\bullet}$ be the cone of the map
\[ p^{\ast} i^{\ast} \Omega_{(\Hilb^d(S))^2} \ra j^{\ast} E^{\bullet} \oplus p^{\ast} \Omega_Z \]
induced by the diagram \eqref{iweufruf}.
The cone $F^{\bullet}$ maps to $\BL_{M_Z}$ and defines a perfect obstruction theory on $M_Z$.
By \cite[Proposition 5.10]{BF}, the associated virtual class is $[M_Z]^{\text{vir}}$.

Let $[f : C \to \Hilb^d(S), p_1, p_2] \in M_Z$ be a point.
For simplicity, we consider all complexes on the level of tangent spaces at the moduli point $[f]$.
Let $E_{\bullet}$ and $F_{\bullet}$ denote the derived duals of $E^{\bullet}$ and $F^{\bullet}$ respectively.

We recall the construction of $E_{\bullet}$, see \cite{GWNL,STV}.
Consider the semi-regularity map
\begin{equation} b : R \Gamma( C, f^{\ast} T_{\Hilb^d(S)} ) \to V[-1] \label{semiregu_map} \end{equation}
where $V = H^0(\Hilb^d(S), \Omega^2_{\Hilb^d(S)})^{\vee}$,
and recall the ordinary (non-reduced) perfect obstruction theory of $M$ at the point $[f]$,
\[ E_{\bullet}^{\text{vir}} = \Cone\Big( R \Gamma( C, \BT_{C} (-p_1 - p_2) ) \to R \Gamma( C, f^{\ast} T_{\Hilb^d(S)} ) \Big), \]
where $\BT_C = \BL_{C}^{\vee}$ is the tangent complex on $C$.
Then, by the vanishing of the composition
\begin{equation} R \Gamma( C, \BT_{C} (-p_1 - p_2) ) \to R \Gamma( C, f^{\ast} T_{\Hilb^d(S)} ) \xrightarrow{b} V[-1], \label{virclass_vanishing} \end{equation}
the map \eqref{semiregu_map} induces a morphism $\overline{b} : E_{\bullet}^{\text{vir}} \to V[-1]$
with co-cone $E_{\bullet}$.

By a diagram chase, $F_{\bullet}$ is the co-cone of
\[ (\overline{b}, d \ev) \colon E_{\bullet}^{\text{vir}} \to V[-1] \oplus N_{Z, (z_1, z_2)} \]
where $z_1,z_2$ are the basepoints of the Lagrangian fiber $Z_1,Z_2$ respectively,
$N_{Z,(z_1,z_2)}$ is the normal bundle of $Z$ in $\Hilb^d(S)^2$ at $(z_1, z_2)$,
and $d \ev$ is the differential of the evaluation map.
Since taking the cone and co-cone commutes, the complex $F_{\bullet}$ is therefore the cone of
\begin{equation} \gamma: R \Gamma(C, \BT_C(-p_1-p_2)) \to K, \label{210} \end{equation}
where
\begin{equation} K = \Cocone\Big[ (b, d \ev) : R \Gamma(C, f^{\ast} T_{\Hilb^d(S)}) \to V[-1] \oplus N_{Z, (z_1, z_2)} \Big] \,. \label{206} \end{equation}

Consider the decomposition
\begin{equation} C = C_0 \cup A_1 \cup \dots \cup A_{24} \cup B_1 \cup \dots \cup B_{2d-2}, \label{rfofdjf} \end{equation}
where the components $A_i$ and $B_j$ are attached to $C_0$ at the points $\tilde{x}_i$ and~$\tilde{y}_j$ respectively.
Tensoring $R \Gamma( C, \BT_{C} (-p_1 - p_2) )$ and $K$
against the partial renormalization sequence associated to decomposition \eqref{rfofdjf},
we will show that the dependence on $L$ cancels in the cone of \eqref{210}.

The map $(b, d\!\ev)$ fits into the diagram
\begin{equation} \label{207}
\xymatrix{
R \Gamma(C, f^{\ast} T_{\Hilb^d(S)}) \ar[r]^u \ar[d]^{(b, d \ev)} & R \Gamma(L, f^{\ast} T_{\Hilb^d(S)}) \ar[d]^{v  = (b, d \ev)} \\
V[-1] \oplus N_{Z, (z_1, z_2)} \ar[r]^{(\sigma, \id)} & V[-1] \oplus N_{Z, (z_1, z_2)},
}
\end{equation}
where $u$ is the restriction map and $\sigma$ is the induced map\footnote{$\sigma$ is the inverse to the natural
isomorphism in the other direction induced by the sequence of surjections $H^1(C,\Omega_C) \ra \oplus H^1(C_i, \Omega_{C_i}) \ra H^1(C,\omega_C) \ra 0$.}.
By Lemma \ref{normaltoL}, the co-cone of $v$ is $R \Gamma( \BT_L(-p_1 - p_2) )$.

The partial normalization sequence of $C$
with respect to $\widetilde{x}_i$ and $\widetilde{y}_j$ is
\begin{equation} 
0 \ra \CO_C \ra \CO_L \oplus_{D \in \{ A_i, B_j \}} \CO_{D} \ra \oplus_{s \in \{ \widetilde{x}_i, \widetilde{y}_j \}} \CO_{C,s} \ra 0 \label{normalization} \,.
\end{equation}
Tensoring \eqref{normalization} with $f^{\ast} T_{\Hilb^d(S)}$,
applying $R \Gamma( \cdot )$ and factoring with \eqref{207},
we obtain the exact triangle
\begin{equation} K \ra R \Gamma(L, \BT_{L}(-p_1 - p_2)) \oplus_D R \Gamma(D, f_{|D}^{\ast} T_{\Hilb^d(S)}) \ra \oplus_s T_{\Hilb^d(S),s} \ra K[1]. \label{209} \end{equation}

For each node $t \in C$, let $N_t$ (resp. $T_{t}$)
be the tensor product (resp. the direct sum)
of the tangent spaces to the branches of $C$ at $t$.
Tensoring \eqref{normalization} with $\BT_C(-p_1 - p_2)$ and applying $R \Gamma( \cdot )$, we obtain the exact triangle
\begin{equation} R \Gamma \BT_C(-p_1 - p_2) \to R \Gamma(\BT_L(-p_1-p_2)) \oplus_{D} R \Gamma(\BT_{D}) \oplus_t N_{t}[-1] \to \oplus_t T_{t} \to \ldots \,. \label{208} \end{equation}
By the vanishing of \eqref{virclass_vanishing} (applied to $C = L$), the sequence \eqref{208} maps naturally to \eqref{209}.
Consider the restriction of this map to the summand $R \Gamma( \BT_L(-p_1-p_2) )$ which appears in the second term of  \eqref{208},
\[ \varphi : R \Gamma( \BT_L(-p_1-p_2) ) \to R \Gamma(L, \BT_{L}(-p_1 - p_2)) \oplus_D R \Gamma(D, f_{|D}^{\ast} T_{\Hilb^d(S)}) \,. \]
Then, the composition of $\varphi$ with the projection to $R \Gamma(L, \BT_{L}(-p_1 - p_2))$ is the identity.
Hence, $F_{\bullet} = \Cone(\gamma)$ admits the exact sequence
\begin{equation} F_{\bullet} \ra \oplus_D G_D \overset{\psi}{\ra} \oplus_D H_D \ra F_{\bullet}[1], \label{211} \end{equation}
where $D$ runs over all $A_i$ and $B_j$, and
\begin{align*}
G_D & = \Cone\Big[ R \Gamma(\BT_{D}) \oplus_t N_{t}[-1] \ra R \Gamma(D, f_{|D}^{\ast} T_{\Hilb^d(S)}) \Big]\\
H_D & = \Cone\Big[ \oplus_t T_{t} \ra \oplus_t T_{\Hilb^d(S),t} \Big].
\end{align*}
Here $t = t(D) = D \cap C_0$ is the attachment point of the component $D$.

The map $\psi$ in \eqref{211} maps the factor $G_D$ to $H_D$ for all $D$.
For $D = A_i$ consider the decomposition
\[ f|_{A_i} = \phi + w_1 + \dots + w_{d-1}. \]
The trivial factors which arise in $G_D$ and $H_D$
from the tangent space of $\Hilb^d(S)$ at the points $w_1, \dots, w_{d-1}$
cancel each other in $\Cone(G_D \to H_D)$. Hence
$\Cone(G_D \to H_D)$ only depends on $\phi : C \to S$, and therefore
only on the image of $[f]$ in the factor $M^{\textup{(N)}}_{x_i}(h_{x_i})$,
where $M^{\textup{(N)}}_{x_i}(h_{x_i})$ is the moduli space defined in Section \ref{splitting}.
The case $D = B_j$ is similar.

Hence, $F_{\bullet}$ splits into a sum of complexes pulled back from each factor
of the product splitting \eqref{spl1}.
Since $F_{\bullet}$ is a perfect obstruction theory on~$M$,
the complexes on each factor are perfect obstruction theories. Let
\[ [M^{\textup{(N)}}_{x_i}(h_{x_i})]^{\text{vir}} \quad \text{ and } \quad [ M^{\textup{(F)}}_{y_j}(h_{y_j}, k_{y_j}) ]^{\text{vir}} \]
be their virtual classes respectively. We have proved the following.

\begin{prop} \label{W0splitprop} Let $\Psi$ be the splitting morphism \eqref{spl1}. Then,
$\Psi$ is an isomorphism and we have
\[ \Psi_{\ast} [M_{Z}]^{\text{vir}}
= \sum_{\textbf{h}, \textbf{k}} \left(
\prod_{i=1}^{24} [M^{\textup{(N)}}_{x_i}(h_{x_i})]^{\text{vir}} \times \prod_{j=1}^{2d-2} [ M^{\textup{(F)}}_{y_j}(h_{y_j}, k_{y_j}) ]^{\text{vir}} \right)
\]
where the sum is over the set \eqref{spl1_indexing} satisfying \eqref{spl2_indexing}.
\end{prop}

\subsubsection{The series $F^{\textup{GW}}$} \label{concl_basic_case}
We consider the left hand side of Theorem \ref{ellthm}. By \eqref{virt_pushpull_mddfds}, we have
\begin{equation*}\label{ellthm_eval_44231}
\blangle \Fp_{-1}(F)^d 1_S \ , \ \Fp_{-1}(F)^d 1_S \brangle^{\Hilb^d(S)}_q
=\ \sum_{h \geq 0} \sum_{k \in \BZ} y^k q^{h-1} \int_{[M_{Z}(h,k)]^{\text{vir}}} 1 \,.
\end{equation*}
By Proposition \ref{W0splitprop}, this equals
\begin{align*}
& \sum_{\substack{h \geq 0 \\ k \in \BZ}} y^k q^{h-1}
\sum_{\substack{ (\textbf{h}, \textbf{k}) \\ \sum_i h_{x_i} + \sum_j h_{y_j} = h \\ \sum_{j} k_{y_j} = k + (d-1) }}
\bigg( \prod_{i=1}^{24} \int_{[M^{\textup{(N)}}_{x_i}(h_{x_i})]^{\text{vir}}}1 \bigg)
\cdot
\bigg( \prod_{j=1}^{2d-2} \int_{[ M^{\textup{(F)}}_{y_j}(h_{y_j}, k_{y_j}) ]^{\text{vir}}} 1 \bigg) \\
=\ & y^{-(d-1)}q^{-1}
\bigg( \prod_{i=1}^{24} \sum_{h_{x_i} \geq 0} q^{h_{x_i}} \int_{[M^{\textup{(N)}}_{x_i}(h_{x_i})]^{\text{vir}}}1 \bigg) \\
& \qquad \qquad \qquad \qquad \qquad \qquad \qquad 
\times \bigg( \prod_{j=1}^{2d-2} \sum_{\substack{h_{y_j} \geq 0 \\ k_{y_j} \in \BZ}} y^{k_{y_j}} q^{h_{y_j}} \int_{[ M^{\textup{(F)}}_{y_j}(h_{y_j}, k_{y_j}) ]^{\text{vir}}} 1 \bigg) \\
=\ & \Big( \prod_{i = 1}^{24} \sum_{h \geq 0} q^{h - \frac{1}{24}} \int_{[M^{\textup{(N)}}_{x_i}(h)]^{\text{vir}}}1 \Big) \cdot
 \Big( \prod_{i = 1}^{2d-2} \sum_{\substack{ h \geq 0 \\ k \in \BZ }} q^{h} y^{k - \frac{1}{2}} \int_{[ M^{\textup{(F)}}_{y_j}(h, k) ]^{\text{vir}}} 1 \Big) \,.
\end{align*}
The integrals in the first factor were calculated by Bryan and Leung
in their proof of the Yau-Zaslow conjecture \cite{BL}. The result is
\begin{equation}
 \sum_{h \geq 0} q^h \int_{[M^{\textup{(N)}}_{x_i}(h)]^{\text{vir}}}1 \  =\  \prod_{m \geq 0} \frac{1}{1-q^m} \,. \label{nodal_fiber_contribution}
\end{equation}
By deformation invariance, the integrals 
\[ \int_{[ M^{\textup{(F)}}_{y_j}(h, k) ]^{\text{vir}}} 1 \]
only depend on $h$ and $k$. Define the generating series 
\begin{equation} F^{\textup{GW}}(y,q) = \sum_{h \geq 0} \sum_{k \in \BZ} q^{h} y^{k - \frac{1}{2}} \int_{[ M^{\textup{(F)}}_{y_j}(h, k) ]^{\text{vir}}} 1 \,. \label{FGW}\end{equation}
By our convention on $M^{\textup{(F)}}_{y_j}(0,0)$, the $y^{-1/2} q^0$-coefficient of $F^{\textup{GW}}$ is $1$.

Let $\Delta(q) = q \prod_{m \geq 1} (1-q^m)^{24}$ be the modular discriminant $\Delta(\tau)$ considered
as a formal expansion in the variable $q = e^{2 \pi i \tau}$. We conclude
\[ \blangle \Fp_{-1}(F)^d 1_S \ , \ \Fp_{-1}(F)^d 1_S \brangle^{\Hilb^d(S)}_q = \frac{F^{\textup{GW}}(y,q)^{2d-2}}{\Delta(q)} \,. \]
The proof of Theorem~\ref{ellthm} now follows directly from Theorem~\ref{thm_F_evaluation} below.

\subsection{Evaluation of $F^{\textup{GW}}$ and the Kummer K3} \label{section_Kummer_evaluation}
Let $F$ be the theta function which already appeared in Section \ref{YZ_section_statement_of_results},
\begin{equation*}
F(z,\tau) = \frac{\vartheta_1(z,\tau)}{\eta^3(\tau)}
 = (y^{1/2} + y^{-1/2}) \prod_{m \geq 1} \frac{ (1 + yq^m) (1-y^{-1}q^m)}{ (1-q^m)^2 } \,,
\end{equation*}
where $q = e^{2 \pi i \tau}$ and $y = - e^{2 \pi i z}$.
\begin{thm} \label{thm_F_evaluation}
Under the variable change $q = e^{2 \pi i \tau}$ and $y = -e^{2 \pi i z}$,
\[ F^{\textup{GW}}(y,q) = F(z,\tau). \]
\end{thm}

In Section \ref{section_Kummer_evaluation} we present a proof of Theorem \ref{thm_F_evaluation} using the Kummer K3 surface and the Yau-Zaslow formula.
An independent proof is given in Section~\ref{Section_Hilb2P1xE} through the geometry of $\Hilb^2(\p^1 \times E)$, where $E$ is an elliptic curve.

The Yau-Zaslow formula was used in the geometry of Kummer K3 surfaces before by S. Rose \cite{Ros14}
to obtain virtual counts of hyperelliptic curves on abelian surfaces.
While the geometry used in \cite{Ros14} is similar to our setting,
the closed formula of Theorem \ref{thm_F_evaluation} in terms
of the Jacobi theta function $F$ is new.
For example, Theorem \ref{thm_F_evaluation} yields
a new, closed formula for hyperelliptic curve counts on an abelian surface, see \cite{BOPY}.

\subsubsection{The Kummer K3}
Let $\AAA$ be an abelian surface. The \emph{Kummer} of $\AAA$ is the blowup
\begin{equation} \rho : \Km(\AAA) \to \AAA / \pm 1 \label{bldn_kummer} \end{equation}
of $\AAA / \pm 1$ along its 16 singular points.
It is a smooth projective K3 surface.
Alternatively, consider the composition
\[ s : \Hilb^2(\AAA) \ra \Sym^2(\AAA) \ra \AAA \]
of the Hilbert-Chow morphism with the addition map.
Then, $\Km(\AAA)$ is the fiber of $s$ over the identity element $0_\AAA \in \AAA$,
\begin{equation} \Km(\AAA) = s^{-1}(0_\AAA) \,. \label{Kummer_cnstruction_23142}\end{equation}

Let $E$ and $E'$ be generic elliptic curves and let 
\[ \AAA = E \times E' .\]
Let $t_1, \dots, t_4$ and $t'_1, \dots, t'_4$ denote the 2-torsion points
of $E$ and $E'$ respectively. The \emph{exceptional curves} of $\Km(\AAA)$ are the divisors
\[ A_{ij} = \rho^{-1}(\, (t_i, t'_j) \,), \ \ \  i,j = 1, \dots, 4 \,.\]

The projection of $\AAA$ to the factor $E$ induces the elliptic fibration
\[ p : \Km(\AAA) \ra \AAA / \pm 1 \ra E / \pm 1 = \p^1. \]
Hence, $\Km(\AAA)$ is an elliptically fibered K3 surface.
Similarly, we let $p' : \Km(\AAA) \to \p^1$ denote the fibration induced by the projection $\AAA \to E'$.
Since $E$ and $E'$ are generic, the fibration $p$ has exactly 4 sections 
\[ s_1, \dots, s_4 : \p^1 \to \Km(\AAA) \]
corresponding to the torsion points $t'_1, \dots, t_4'$ of $E'$. We write
$B_i \subset \Km(\AAA)$
for the image of $s_i$,
and we let $F_x$ denote the fiber of $p$ over $x \in \p^1$

Let $y_1, \dots, y_4 \in \p^1$ be the image of the $2$-torsion points $t_1, \dots, t_4 \in E$
under $E \to E/\pm 1=\p^1$. The restriction
\[ p: \Km(\AAA) \setminus \{ F_{y_1}, \dots, F_{y_4} \} \ra \p^1 \setminus \{ y_1, \dots, y_4 \} \]
is an isotrivial fibration with fiber $E'$.
For $i \in \{1,\dots, 4\}$, the fiber $F_{y_i}$ of $p$ over the points $y_i$ is singular with divisor class
\[ F_{y_i} = 2 T_i + A_{i1} + \dots + A_{i4}, \]
where $T_i$ denotes the image of the section of $p' : \Km(\AAA) \to \p^1$ corresponding to the $2$-torsion points $t_i$.
We summarize the notation in figure \ref{Kummer_diagram}.

\begin{figure}[ht]
\begin{center}
\begin{tikzpicture}
\draw[thick]
   (1.0,0) node[below, align=center]{$T_1$}
-- (1.0,1.0)
-- (1.0,2.5) node[above right, align=center]{\ $\vdots$}
-- (1.0,4.0) node[above right, align=left]{$A_{12}$}
-- (1.0,5.5) node[above right, align=left]{$A_{11}$}
-- (1.0,6.5);
\draw[thick]
   (2.5,0) node[below, align=center]{$T_2$}
-- (2.5,5.5) node[above right, align=left]{$A_{21}$}
-- (2.5,6.5);
\draw[thick]
   (4.0,0) node[below, align=center]{$T_3$}
-- (4.0,5.5) node[above right, align=left]{$\cdots$}
-- (4.0,6.5);
\draw[thick]
   (5.5,0) node[below, align=center]{$T_4$}
-- (5.5,6.5);
\draw[thick] 
   (0,1.0) node[left, align=right]{$B_4$}
-- (6.5,1.0);
\draw[thick]
   (0,2.5) node[left, align=right]{$B_3$}
-- (6.5,2.5);
\draw[thick]
   (0,4.0) node[left, align=right]{$B_2$}
-- (6.5,4.0);
\draw[thick]
   (0,5.5) node[left, align=right]{$B_1$}
-- (6.5,5.5);
\draw[thick]
   (0, -2.0) node[left, align=right]{$\p^1 = E/\pm 1$}
-- (6.5, -2.0);
\draw [fill] (1.0,5.5) circle [radius=0.1];
\draw [fill] (1.0,4.0) circle [radius=0.1];
\draw [fill] (2.5,5.5) circle [radius=0.1];
\draw [fill] (1.0,-2.0) circle [radius=0.06];
\node [below] at (1,-2) {$y_1$};
\draw [fill] (2.5,-2.0) circle [radius=0.06];
\node [below] at (2.5,-2) {$y_2$};
\draw [fill] (4.0,-2.0) circle [radius=0.06];
\node [below] at (4.0,-2) {$y_3$};
\draw [fill] (5.5,-2.0) circle [radius=0.06];
\node [below] at (5.5,-2) {$y_4$};
\draw[thick, ->]
   (3.25, -1.0)
-- (3.25, -1.5);
\node [right] at (3.35, -1.25) {$p$};
\draw[thick]
   (8.0, 0.0)
-- (8.0, 6.5) node[above, align=center]{$E'/\pm 1$};
\draw[thick, ->]
   (6.75, 3.25)
-- (7.5, 3.25);
\node [above] at (7.125, 3.35) {$p'$};
\end{tikzpicture}
\caption{The Kummer K3 of $\AAA = E \times E'$}
\label{Kummer_diagram}
\end{center}
\end{figure}
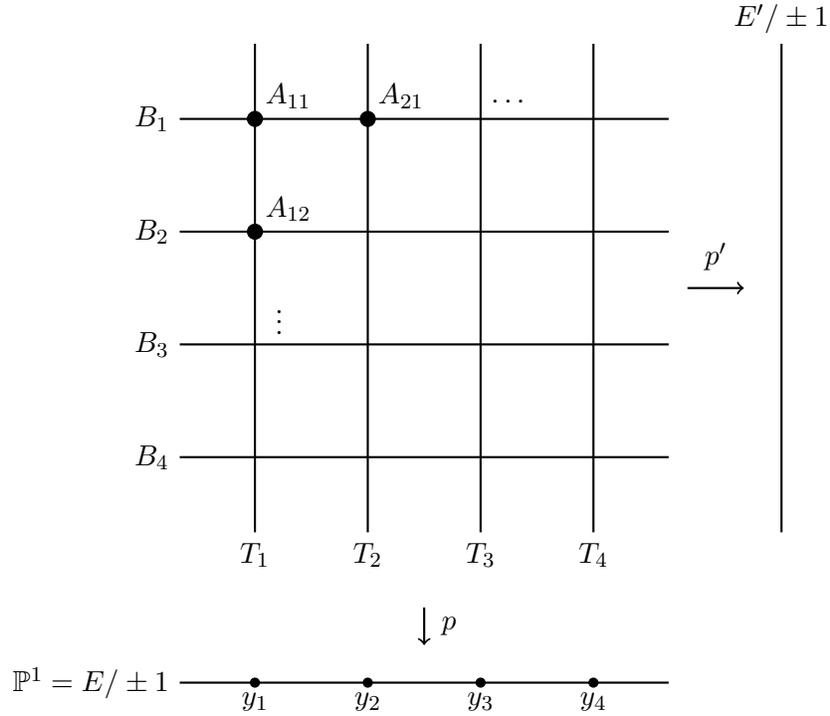

Let $F$ and $F'$ be the class of a fiber of $p$ and $p'$ respectively.
We have the intersections
\[ F^2 = 0, \quad F \cdot F' = 2, \quad F'^2 = 0 \]
and
\[ F \cdot A_{ij} = F' \cdot A_{ij} = 0, \quad A_{ij} \cdot A_{k \ell} = -2\, \delta_{ik} \delta_{j \ell}
\ \ \ \ \ \text{ for all } i,j,k, \ell \in \{1, \dots, 4 \} \,. \]
By the relation
\begin{equation} \label{241_3}
\begin{aligned}
F & = 2 T_i + A_{i1} + A_{i2} + A_{i3} + A_{i4} \\
F' & = 2 B_i + A_{1i} + A_{2i} + A_{3i} + A_{4i}
\end{aligned}
\end{equation}
for $i \in \{1, \dots, 4 \}$ this determines the intersection numbers of all the divisors above.

\subsubsection{Rational curves and $F^{\textup{GW}}$}
Let $\beta \in H_2( \Km(\AAA), \BZ)$ be an effective curve class and let
\[ \big\langle\, 1 \, \big\rangle_{0, \beta}^{\Km(\AAA)} =
 \int_{[ \Mbar_{0}( \Km(\AAA), \beta) ]^{\text{red}} } 1
\]
denote the genus $0$ Gromov-Witten invariants of $\Km(\AAA)$.
For an integer $n \geq 0$ and a tuple $\textbf{k} = (k_{ij})_{i,j = 1, \dots, 4}$ of half-integers $k_{ij} \in \frac{1}{2} \BZ$,
define the class
\[ \beta_{n, \textbf{k}} = \frac{1}{2} F' + \frac{n}{2} F + \sum_{i,j = 1}^{4} k_{ij} A_{ij}\ \in H_2( \Km(\AAA) , \BQ). \]
We write $\beta_{n, \textbf{k}} > 0$, if $\beta_{n, \textbf{k}}$ is effective.
\begin{prop} \label{Kummer_prop_1} We have
\[ \sum_{\substack{n, \textbf{k} \\ \beta_{n,\textbf{k}} > 0}} \big\langle\, 1 \, \big\rangle_{0, \beta_{n,\textbf{k}}}^{\Km(\AAA)} q^n y^{\sum_{i,j} k_{ij}}
\ =\  4 \cdot F^{\textup{GW}}(y,q)^{4}, \]
where the sum runs over all $n \geq 0$ and $\textbf{k} = (k_{ij})_{i,j} \in (\frac{1}{2} \BZ)^{4 \times 4}$
for which $\beta_{n,\textbf{k}}$ is an effective curve class.
\end{prop}

\begin{proof} Let $f : C \to \Km(\AAA)$ be a genus $0$ stable map in class $\beta_{n,\textbf{k}}$.
By genericity of $E$ and $E'$ the fibration $p$ has only the sections
$B_1, \dots, B_4$.
Since $p \circ f$ has degree $1$, the image divisor of $f$ is then of the form
\[ \im(f) = B_{\ell} + D' \]
for some $1 \leq \ell \leq 4$ and a divisor $D'$, which is contracted by $p$.
Since the fibration $p$ has fibers isomorphic to $E'$ away from the points $y_1, \dots,  y_4 \in \p^1$,
the divisor $D'$ is supported on the singular fibers $F_{y_i}$. Hence, there exist non-negative integers
\[ a_i, \ \ i=1,\dots,4 \quad \text{ and } \quad b_{ij}, \ \ i,j = 1, \dots, 4 \]
such that
\[ \im(f) = B_{\ell} + \sum_{i=1}^4 a_i T_i + \sum_{i,j = 1}^{4} b_{ij} A_{ij} . \]
Let $C_0$ be the component of $C$ which gets mapped by $f$ isomorphically to $B_{\ell}$, and let 
$D_i$ be the component of $C$, that maps into the fiber $F_{y_i}$. Then,
\begin{equation} C = C_0 \cup D_1 \cup \dots \cup D_4, \label{2345} \end{equation}
with pairwise disjoint $D_i$.
Under $f$ the intersection points $C_0 \cap D_j$ gets mapped to $s_{\ell}(y_j)$, where $s_{\ell} : \p^1 \to \Km(\AAA)$
denotes the $\ell$-th section of $p$.

By arguments similar to the proof of Lemma \ref{doesnotsmooth} or by the geometry of the linear system $|\beta_{n,\textbf{k}}|$,
the nodal points $C_0 \cap D_j$ do not smooth under infinitesimal deformations of $f$.
The decomposition \eqref{2345} is therefore preserved under infinitesimal deformations.
This implies that the moduli spaces
$\Mbar_{0}( \Km(\AAA), \beta_{n, \textbf{k}})$
admits the decomposition
\begin{equation} \Mbar_{0}( \Km(\AAA), \beta_{n, \textbf{k}})\ =\ 
\bigsqcup_{\ell = 1}^{4}\ \bigsqcup_{n = n_1 + \dots + n_4}\ \prod_{i=1}^{4}\ M_{y_i}^{(\ell)}(n_i, (k_{ij} + \frac{1}{2} \delta_{j \ell})_j ),
\label{242_1} \end{equation}
where $M_{y_i}^{(\ell)}(n_i, (k_{ij})_j )$ is the moduli space of stable $1$-pointed genus $0$ maps to $\Km(\AAA)$ in class
\begin{equation*} \frac{n_i}{2} F + \sum_{j = 1}^4 k_{ij} A_{ij} \end{equation*}
and with marked point mapped to $s_{\ell}(y_i)$.
The term $\frac{1}{2} \delta_{j \ell}$ appears in \eqref{242_1} since
\[ B_{\ell} = \frac{1}{2} ( F' - A_{1 \ell} - A_{2 \ell} - A_{3 \ell} - A_{4 \ell}). \]

For $n_i \geq 0$ and $k_i \in \BZ / 2$, let
\begin{equation}
M^{(\ell)}_{y_i}(n_i,k_i) = \bigsqcup_{\substack{k_{i1}, \dots, k_{i4} \in \BZ/2 \\ k_i = k_{i1} + \dots + k_{i4}}} M_{y_i}^{(\ell)}(n_i, (k_{ij})_j ) \,.
\label{moduli_space_Kummer_case}
\end{equation}

be the moduli space parametrizing stable $1$-pointed genus $0$ maps to $\Km(\AAA)$ in class
$\frac{n_i}{2} F + \sum_{j} k_{ij} A_{ij}$
for some $k_{ij}$ with $\sum_{j} k_{ij} = k_i$ and such that the marked points maps to $s^{\ell}(y_i)$.

Let $n \geq 0$ and $k \in \BZ/2$ be fixed. Taking the union of \eqref{242_1} over all~$\textbf{k}$
such that $k = \sum_{i,j} k_{ij}$, interchanging sum and product and reindexing, we get
\begin{equation}
\bigsqcup_{\substack{ \textbf{k} \colon \sum_{i,j} k_{ij} = k}} \Mbar_{0}( \Km(\AAA), \beta_{n, \textbf{k}})
=
\bigsqcup_{\ell = 1}^{4}\ \bigsqcup_{\substack{n = n_1 + \dots + n_4 \\ k + 2 = k_1 + \dots + k_4}} \prod_{i=1}^{4} M^{(\ell)}_{y_i}(n_i,k_i)
\label{242_2}
\end{equation}
By arguments essentially identical to those in Section \ref{anvirclass} the moduli space $M^{(\ell)}_{y_i}(n_i,k_i)$ carries a natural virtual class
\begin{equation} [ M^{(\ell)}_{y_i}(n_i,k_i) ]^{\text{vir}} \label{virtual_class_XYZ}\end{equation}
of dimension $0$ such that the splitting \eqref{242_2} holds also for virtual classes:
\begin{equation}
\bigsqcup_{\substack{ \textbf{k} \colon \sum_{i,j} k_{ij} = k}} [ \Mbar_{0}( \Km(\AAA), \beta_{n, \textbf{k}}) ]^{\text{red}}
=
\bigsqcup_{\ell = 1}^{4}\ \bigsqcup_{\substack{n = n_1 + \dots + n_4 \\ k + 2 = k_1 + \dots + k_4}} \prod_{i=1}^{4}\ [ M^{(\ell)}_{y_i}(n_i,k_i) ]^{\text{vir}} \,.
\label{242_3}
\end{equation}

Consider the Bryan-Leung K3 surface $\pi_S : S \to \p^1$. Let\footnote{We may
restrict here to the Hilbert scheme of $2$ points, since the evaluation of $F^{\textup{GW}}$ is independent of the number of points.}
\[ L \subset \Hilb^2(B) \]
be a fixed generic line and let $y \in \p^1$ be a point with $2y \in L$. Let
\begin{equation*} M^{\textup{(F)}}_{S,y}(n,k) \end{equation*}
be the moduli space parametrizing $1$-marked genus $0$ stable maps to $\Hilb^2(S)$ in class $n F + k A$,
which map the marked point to $s^{[2]}(2y)$, see \eqref{234_1}.
The subscript $S$ is added to avoid confusion. By Section~\ref{Section_analysis_of_virtual_class},
the moduli space $M^{\textup{(F)}}_{S,y}(n,k)$ carries a natural virtual class.

\begin{lemma} \label{242_lemma} We have
\begin{equation} \int_{ [ M^{(\ell)}_{y_i}(n,k) ]^{\text{vir}} } 1 \ =\  \int_{ [ M^{\textup{(F)}}_{S,y}(n,k) ]^{\text{vir}} } 1 \,. \label{242_lemma_vkmdsv} \end{equation}
\end{lemma}

The Lemma is proven below. We finish the proof of Proposition \ref{Kummer_prop_1}.
By the decomposition \eqref{242_3},
\begin{multline*}
 \quad \quad \quad \sum_{n \geq 0} \sum_{\substack{ \textbf{k} \\ \beta_{n,\textbf{k}} > 0}} \big\langle\, 1 \, \big\rangle_{0, \beta_{n,\textbf{k}}}^{\Km(\AAA)} q^n y^{\sum_{i,j} k_{ij}} \\
= \ 
\sum_{\substack{n \geq 0 \\ k \in \BZ}} \sum_{\ell = 1}^{4}
\sum_{\substack{n = n_1 + \dots + n_4 \\ k + 2 = k_1 + \dots + k_4}}
\prod_{i=1}^{4} q^{n_i} y^{k_i - \frac{1}{2}} \int_{ [ M^{(\ell)}_{y_i}(n_i,k_i) ]^{\text{vir}} } 1 \quad \quad \quad
\end{multline*}
An application of Lemma \ref{242_lemma} then yields
\[
\sum_{\ell = 1}^4 \prod_{i=1}^{4} \Big( \sum_{ \substack{ n_i \geq 0 \\ k_i \in \BZ }} q^{n_i} y^{k_i - \frac{1}{2}} \int_{ [ M^{(\ell)}_{y_i}(n_i,k_i) ]^{\text{vir}} } 1 \Big) 
= \ 4 \cdot ( F^{\textup{GW}}(y,q) )^4 \,.
\]
This completes the proof of Proposition \ref{Kummer_prop_1}.
\end{proof}

\begin{proof}[Proof of Lemma \ref{242_lemma}]
Let $F_y = \pi_S^{-1}(y)$ denote the fiber of $\pi_S$ over $y \in \p^1$. Consider the
deformation of $S$ to the normal cone of $F_y$,
\[ \CS = \Bl_{F_{y} \times 0}(S \times \BA^1) \to \BA^1, \]
and let $\CS^{\circ} \subset \CS$ be the complement of the proper transform of $S \times 0$.
The \emph{relative} Hilbert scheme
\begin{equation} \Hilb^2( \CS^{\circ} / \BA^1 ) \to \BA^1 \label{sofjosfosdf} \end{equation}
parametrizes length $2$ subschemes on the fibers of $\CS^{\circ} \to \BA^1$.
Let 
\[ p : M' \to \BA^1 \]
be the moduli space of $1$-pointed genus $0$ stable maps to $\Hilb^2( \CS^{\circ} / \BA^1 )$
in class $n F + kA$, with the marked point mapping to the proper transform
of $s^{[2]}(2 y) \times \BA^1$. 
The fiber of $p$ over $t \neq 0$ is
\[ p^{-1}(t) = M^{\textup{(F)}}_{S,y}(n,k). \]
The fiber over $t = 0$ parametrizes maps to $\Hilb^2(\BC \times F_y)$.
Since the domain curve has genus $0$, these map to a fixed fiber of the natural map
\[ \Hilb^2( \BC \times F_y ) \xrightarrow{\ \rho\ } \Sym^2 ( \BC \times F_y ) \xrightarrow{\ + \ } F_y \,. \]
We find, that $p^{-1}(0)$ parametrizes $1$-pointed genus $0$ stable maps
into a singular $\mathsf{D}_4$ fiber of a trivial elliptic fibration, with given conditions on the class and the marking.
Comparing with the construction of $\Km(\AAA)$ via \eqref{Kummer_cnstruction_23142} and the definition of $M^{(\ell)}_{y_i}(n_i,k_i)$,
one finds
\[ p^{-1}(0) \cong M^{(\ell)}_{y_i}(n_i,k_i). \]

The moduli space $M'$ carries the perfect obstruction theory obtained
by the construction of section \ref{basic_case} in the relative context.
On the fibers over $t \neq 0$ and $t=0$ the perfect obstruction theory of $M'$ restricts to the perfect obstruction theories of
$M^{\textup{(F)}}_{S,y}(n,k)$ and $M^{(\ell)}_{y_i}(n_i,k_i)$ respectively.
Hence, the associated virtual class $[M']^{\text{vir}}$ restricts on the fibers to the earlier defined virtual classes:
\begin{align*}
\quad \quad \quad t^{!} [M']^{\text{vir}} & = [M^{\textup{(F)}}_{S,y}(n,k)]^{\text{vir}}
\quad \quad (t \neq 0), \\
0^{!} [M']^{\text{vir}} & = [ M^{(\ell)}_{y_i}(n_i,k_i) ]^{\text{vir}}.
\end{align*}
Since $M' \to \BA^1$ is proper, the proof of Lemma \ref{242_lemma} follows now from the principle
of conversation of numbers, see \cite[Section 10.2]{Fulton}.
\end{proof}

\subsubsection{Effective classes}
By Proposition \ref{Kummer_prop_1}, the evaluation of $F^{\textup{GW}}(y,q)$ is reduced to the evaluation of the series
\begin{equation}
\sum_{\substack{ n, \textbf{k} \\ \beta_{n,\textbf{k}} > 0}} \big\langle\, 1 \, \big\rangle_{0, \beta_{n,\textbf{k}}}^{\Km(\AAA)} q^n y^{\sum_{i,j} k_{ij}} .
\label{Kummer_Series} 
\end{equation}
Since $\Km(\AAA)$ is a K3 surface,
the Yau-Zaslow formula \eqref{YZ} applies to the invariants $\langle 1 \rangle^{\Km(\AAA)}_{\beta}$, when $\beta$ is effective\footnote{In fact, the Yau-Zaslow formula applies to all
classes $\beta \in H_2(\Km(\AAA), \BZ)$ which are of type $(1,1)$ and pair positively with an ample class.
}
The remaining difficulty is to identify precisely the set of effective classes of the form $\beta_{n, \textbf{k}}$.

\begin{lemma} Let $n \geq 0$ and $\textbf{k} \in (\BZ/2)^{4 \times 4}$.
If $\beta_{n,\textbf{k}}$ is effective, then there exists a unique $\ell = \ell(n, \textbf{k}) \in \{ 1, \dots, 4 \}$
such that
\[ \beta_{n,\textbf{k}} = B_{\ell} + \sum_{i=1}^4 a_i T_i + \sum_{i,j = 1}^{4} b_{ij} A_{ij} \,. \]
for some integers $a_i \geq 0$ and $b_{ij} \geq 0$.
\end{lemma}
\begin{proof}
If $\beta_{n,\textbf{k}}$ is effective, then by the argument in the proof of Proposition \ref{Kummer_prop_1},
there exist non-negative integers
\[ a_i, \ \ i=1,\dots,4 \quad \text{ and } \quad b_{ij}, \ \ i,j = 1, \dots, 4 \]
such that
\[ \beta_{n,\textbf{k}} = B_{\ell} + \sum_{i=1}^4 a_i T_i + \sum_{i,j = 1}^{4} b_{ij} A_{ij} \]
for \emph{some} $\ell \in \{ 1, \dots, 4 \}$. We need to show, that $\ell$ is unique.
By \eqref{241_3}, we have
\[ \beta_{n,\textbf{k}} = \frac{F'}{2} + \frac{\sum_{i=1}^{4} a_i}{2} F + \sum_{i,j=1}^{4} \Big( b_{ij} - \frac{a_{i}}{2} - \frac{1}{2} \delta_{j \ell} \Big) A_{ij}, \]
hence $k_{ij} = b_{ij} - \frac{a_{i1}}{2} - \frac{1}{2} \delta_{j \ell}$. We find, that
$\ell$ is the unique integer
such that for every $i$ one of the following holds:
\begin{itemize}
 \item $k_{ij} \in \BZ$ for all $j \neq \ell$ and $k_{i \ell} \notin \BZ$,
 \item $k_{ij} \notin \BZ$ for all $j \neq \ell$ and $k_{i \ell} \in \BZ$.
\end{itemize}
In particular, $\ell$ is uniquely determined by $\textbf{k}$.
\end{proof}

By the proof of proposition \ref{Kummer_prop_1}, the contribution
from all classes $\beta_{n, \textbf{k}}$ with a given $\ell$ to the sum \eqref{Kummer_Series} is independent of $\ell$.
Hence, \eqref{Kummer_Series} equals
\begin{equation} 4 \cdot \sum_{ n, \textbf{k}}
\big\langle\, 1 \, \big\rangle_{0, \beta_{n,\textbf{k}}}^{\Km(\AAA)}
q^n y^{\sum_{i,j} k_{ij}}\,, \label{Kummer_Series2} \end{equation}
where the sum runs over all $(n, \textbf{k})$ such that $\beta_{n,\textbf{k}}$ is effective
\emph{and} $\ell(n, \textbf{k}) = 1$.
Hence, we may assume $\ell = 1$ from now on.

It will be useful to rewrite the classes $\beta_{n, \textbf{k}}$ in the basis
\begin{equation} B_1,\ F \quad \text{ and } \quad T_i,\ A_{i2},\ A_{i3},\ A_{i4}, \ \ \ i = 1, \dots, 4. \label{243_1} \end{equation}
Consider the class
\begin{align*}
\beta_{n, \textbf{k}}
& = \frac{1}{2} F' + \frac{n}{2} F + \sum_{i,j = 1}^{4} k_{ij} A_{ij}\ \in H_2( \Km(\AAA) , \BQ) \\
& = B_1 + \tilde{n} F + \sum_{i=1}^{4} \Big( a_i T_i + \sum_{j=2}^{4} b_{ij} A_{ij} \Big) \,,
\end{align*}
where $(n,\textbf{k})$ and $(\tilde{n}, a_i, b_{ij})$ are related by
\begin{equation}
n = 2 \tilde{n} + \textstyle{\sum}_i a_i, \quad
k_{i1} = -\frac{1}{2} (a_i + 1), \quad
k_{ij} = b_{ij} - \frac{a_i}{2} \ \ (j > 2).
\label{eff_trans_eqn}
\end{equation}

\begin{lemma} \label{eff_lemma_1}
If $\beta_{n, \textbf{k}}$ is \emph{effective}, then 
$\tilde{n}, a_i, b_{ij}$ are integers for all $i,j$.
\end{lemma}

\begin{proof}
If $\beta_{n, \textbf{k}}$ is effective with $\ell(n, \textbf{k}) = 1$, there exist
non-negative integers
\[ \tilde{a}_i, \ \ i=1,\dots,4 \quad \text{ and } \quad \tilde{b}_{ij}, \ \ i,j = 1, \dots, 4 \]
such that
\[ \beta_{n,\textbf{k}} = B_{1} + \sum_{i=1}^4 a_i T_i + \sum_{i,j = 1}^{4} b_{ij} A_{ij} . \]
In the basis \eqref{243_1} we obtain
\[ \beta_{n,\textbf{k}}
= B_1 + \Big( \sum_{i=1}^{4} \tilde{b}_{i1} \Big) F
+ \sum_{i=1}^{4} \Big( (\tilde{a}_i - 2 \tilde{b}_{i1} ) T_i + \sum_{j=2}^{4} (\tilde{b}_{ij} - \tilde{b}_{i1}) A_i \Big).
\]
The claim follows.
\end{proof}

\begin{lemma} \label{eff_lemma_2}
If $\tilde{n}, a_i, b_{ij}$ are integers and $\beta_{n, \textbf{k}}^2 \geq -2$, then $\beta_{n, \textbf{k}}$ is effective.
\end{lemma}
\begin{proof}
If $\tilde{n}, a_i, b_{ij}$ are integers, then $\beta_{n, \textbf{k}}$ is the class of a divisor $D$.
By Riemann-Roch we have
\[ \frac{\chi(\CO(D)) + \chi(\CO(-D))}{2} = \frac{D^2}{2} + 2, \]
and by Serre duality we have
\[ h^0(D) + h^0(-D)\ \geq\ \frac{\chi(\CO(D)) + \chi(\CO(-D))}{2} \,. \]
Hence, if $\beta_{n,\textbf{k}}^2 = D^2 \geq -2$, then $h^0(D) + h^0(-D) \geq 1$.
Since $F \cdot \beta_{n, \textbf{k}} = 1$, we have $h^0(-D) = 0$, and therefore $h^0(D) \geq 1$ and $D$ effective.
\end{proof}

We are ready to evaluate the series \eqref{Kummer_Series2}.

By Lemma \ref{eff_lemma_1} we may replace the sum in \eqref{Kummer_Series2}
by a sum over all integers $\tilde{n} \in \BZ$ and all
elements
\[ x_i = a_i T_i + \sum_{j=2}^{4} b_{ij} A_{ij}, \ \ \ \ i=1,\dots 4 \]
such that
\begin{itemize}
 \item[(i)] $a_i, b_{i2}, b_{i3}, b_{i4}$ are integers for $i \in \{1,\dots, 4\}$,
 \item[(ii)] $B_1 + \tilde{n} F + \sum_{i} x_i$ is effective.
\end{itemize}
Hence, using \eqref{eff_trans_eqn} the series \eqref{Kummer_Series2} equals
\begin{equation}
4 \cdot \sum_{\tilde{n}} \sum_{x_1, \dots, x_4} q^{2 \tilde{n} + \sum_i a_i} y^{-2 + \sum_i \langle x_i, T_i \rangle} 
\big\langle\, 1 \, \big\rangle_{0, B_1 + \tilde{n} F + \sum_i x_i}^{\Km(\AAA)}, \label{KumS_3}
\end{equation}
where the sum runs over all $(\tilde{n}, x_1, \dots, x_4)$ satisfying (i) and (ii) above. 

By the Yau-Zaslow formula \eqref{YZ}, we have
\begin{equation}
\big\langle\, 1 \, \big\rangle_{0, B_1 + \tilde{n} F + \sum_i x_i}^{\Km(\AAA)}
=
\left[ \frac{1}{\Delta(\tau)} \right]_{q^{\tilde{n}-1 + \sum_i \langle x_i, x_i \rangle/2}},
\label{eff_YZ_eval}
\end{equation}
whenever $B_1 + \tilde{n} F + \sum_i x_i$ is effective;
here $[\ \cdot\ ]_{q^m}$ denotes the coefficient of $q^m$.
The term \eqref{eff_YZ_eval} vanishes, unless 
\[ \tilde{n}-1 + \frac{1}{2} \sum_i \langle x_i, x_i \rangle = \frac{1}{2} \left(B_1 + \tilde{n} F + \sum_i x_i\right)^2 \geq -1 \,. \]
When evaluating \eqref{KumS_3}, we may therefore restrict to tuples $(\tilde{n}, x_1, \dots, x_4)$, that
also satisfy
\begin{itemize}
 \item[(iii)] $\big(B_1 + \tilde{n} F + \sum_i x_i \big)^2 \geq -2$.
\end{itemize}
By Lemma \ref{eff_lemma_2}, condition (i) and (iii) together imply condition (ii).
In \eqref{KumS_3} we may therefore sum over tuples $(\tilde{n}, x_1, \dots, x_4)$ satisfying (i) and~(iii) alone.
Rewriting (iii) as
\[ \tilde{n} \geq - \sum_i \langle x_i, x_i \rangle/2 \]
and always assuming (i) in the following sums, \eqref{KumS_3} equals
\begin{align*}
& 4 \cdot \sum_{x_1, \dots, x_4}
\sum_{\tilde{n} \geq \sum_i \frac{\langle x_i, x_i \rangle}{-2}}
q^{2 \tilde{n} + \sum_i a_i} y^{-2 + \sum_i \langle x_i, T_i \rangle}
\left[ \frac{1}{\Delta(\tau)} \right]_{q^{\tilde{n}-1 + \sum_i \langle x_i, x_i \rangle/2}} \\
=\ & 4 \cdot \sum_{x_1, \dots, x_4} y^{-2 + \sum_i \langle x_i, T_i \rangle} q^{2 + \sum_i (a_i - \langle x_i, x_i \rangle)} \\
& \quad \quad \quad \quad \quad \quad \quad \ \ \ \times \sum_{\tilde{n} \geq \sum_i \frac{\langle x_i, x_i \rangle}{-2}} q^{2 \tilde{n} - 2 + \sum_i \langle x_i, x_i \rangle}
\left[ \frac{1}{\Delta(\tau)} \right]_{q^{\tilde{n}-1 + \sum_i \frac{\langle x_i, x_i \rangle}{2}}} \\
=\ & \frac{4}{\Delta(2 \tau)} \cdot \sum_{x_1, \dots, x_4} y^{-2 + \sum_i \langle x_i, T_i \rangle} q^{2 + \sum_i (a_i - \langle x_i, x_i \rangle)} \\
=\ & \frac{4}{\Delta(2 \tau)} \cdot \prod_{i=1}^{4} \Big( \sum_{x_i} y^{- \frac{1}{2} + \langle x_i, T_i \rangle} q^{\frac{1}{2} + a_i - \langle x_i, x_i \rangle} \Big)\,.
\end{align*}

Consider the $\mathsf{D}_4$ lattice, defined as $\BZ^4$ together with the bilinear form
\[ \BZ^4 \times \BZ^4 \ni (x,y) \mapsto \langle x,y \rangle := x^T M y, \]
where
\[ M = \left[ \begin{array}{rrrr} 2 & -1 & -1 & -1 \\ -1 & 2 & 0 & 0 \\ -1 & 0 & 2 & 0 \\ -1 & 0 & 0 & 2 \end{array} \right]. \]
Let $(e_1, \dots, e_4)$ denote the standard basis of $\BZ^4$ and let 
\[ \alpha = 2 e_1 + e_2 + e_3 + e_4. \]
Consider the function
\[ \Theta(z,\tau)
= \sum_{x \in \BZ^4} \exp\Big( - 2 \pi i \Big\langle x + \frac{\alpha}{2} , z e_1 + \frac{e_1}{2} \Big\rangle \Big) \cdot
 q^{\left\langle x + \frac{\alpha}{2}, x + \frac{\alpha}{2} \right\rangle}
\]
where $z \in \BC, \tau \in \BH$ and $q = e^{2 \pi i \tau}$.
The function $\Theta(z,\tau)$ is a theta function with characteristics associated to the lattice $\mathsf{D}_4$.
In particular $\Theta(z,\tau)$ is a Jacobi form of index $1/2$ and weight $2$, see \cite[Section 7]{EZ}.\footnote{
The general form of these theta functions is
\[ \Theta_\mathsf{v}\left[ \begin{array}{cc}A\\B\end{array} \right](z, \tau)
= \sum_{x \in \BZ^4} q^{\frac{1}{2} \langle x+A, x+A \rangle} \exp\Big( 2\pi i \cdot \big\langle x+A, z \cdot \mathsf{v} + B \big\rangle \Big) \,. \]
for characteristics $A,B \in \BQ^4$ and a direction vector $\mathsf{v} \in \BC^4$. Here,
\[ \Theta(z,\tau) = \Theta_{(-e_1)} \left[ \begin{array}{cc} \alpha/2 \\ -e_1/2 \end{array} \right](z, 2 \tau). \]
}

\begin{lemma} \label{theta_lemma_xzz} For every $i \in \{ 1, \dots, 4 \}$,
\[
\sum_{x_i} y^{- \frac{1}{2} + \langle x_i, T_i \rangle} q^{\frac{1}{2} + a_i - \langle x_i, x_i \rangle}
\ = \  \Theta(z,\tau)
\]
under $q = e^{2 \pi i \tau}$ and $y = -e^{2 \pi i z}$.
\end{lemma}
\begin{proof}
Let $\mathsf{D}_4(-1)$ denote the lattice $\BZ^4$ with intersection form \[ (x,y) \mapsto {-}x^{T}M y. \]
The $\BZ$-homomorphism defined by
\[ e_1 \mapsto T_i,\ \ e_2 \mapsto A_{i2},\ \ e_3 \mapsto A_{i3}, \ \  e_4 \mapsto A_{i4} \]
is an isomorphism from $\mathsf{D}_4(-1)$ to
\[ \Big( \BZ T_i \oplus \BZ A_{i2} \oplus \BZ A_{i3} \oplus \BZ A_{i4}, \langle \cdot , \cdot \rangle \Big), \]
where $\langle \ , \, \rangle$ denotes the intersection product on $\Km(\AAA)$. Hence,
\[
\sum_{x_i} y^{- \frac{1}{2} + \langle x_i, T_i \rangle} q^{\frac{1}{2} + a_i - \langle x_i, x_i \rangle}
=
\sum_{x \in \BZ^4} y^{- \frac{1}{2} - \langle x, e_1 \rangle} q^{\frac{1}{2} + \langle x, \alpha \rangle + \langle x, x \rangle}
\]
Using the substitution $y = \exp( 2 \pi i z + \pi i )$, we obtain
\[ \sum_{x \in \BZ^4} \exp \Big( - 2 \pi i \cdot \big\langle x + \frac{\alpha}{2}, z e_1 + \frac{e_1}{2} \big\rangle \Big)
 \cdot q^{ \langle x + \frac{\alpha}{2}, x + \frac{\alpha}{2} \rangle} = \Theta(z,\tau) \,. \qedhere \\
\]
\end{proof}

\noindent By Lemma \ref{theta_lemma_xzz}, we conclude
\begin{equation}
\sum_{\substack{ n, \textbf{k} \\ \beta_{n,\textbf{k}} > 0}} \big\langle\, 1 \, \big\rangle_{0, \beta_{n,\textbf{k}}}^{\Km(\AAA)} q^n y^{\sum_{i,j} k_{ij}}
=
\frac{4}{\Delta(2 \tau)} \cdot \Theta(z,\tau)^4
\label{YZ_Kummer_jdhnr}
\end{equation}

\subsubsection{The theta function of the $\mathsf{D}_4$ lattice}
Consider the Dedekind eta function
\begin{equation}
\eta(\tau) = q^{1/24} \prod_{m \geq 1} (1 - q^m)
\end{equation}
and the first Jacobi theta function
\[ \vartheta_1(z,\tau) = -i q^{1/8} (p^{1/2} - p^{-1/2}) \prod_{m \geq 1} (1 - q^m) (1 - p q^m) (1 - p^{-1} q^m), \]
where $q = e^{2 \pi i \tau}$ and $p = e^{2 \pi i z}$.

\begin{prop} \label{YZ_Theta_identity} We have
\begin{equation} \Theta(z,\tau) = \frac{ -\vartheta_1(z,\tau) \cdot \eta(2 \tau)^6 }{ \eta(\tau)^3 } \label{YZ_Theta_identity_eqn} \end{equation}
\end{prop}
The proof of Proposition \ref{YZ_Theta_identity} is given below.
We complete the proof of Theorem \ref{thm_F_evaluation}.
\begin{proof}[Proof of Theorem \ref{thm_F_evaluation}]
By Proposition \ref{Kummer_prop_1}, we have
\[ 4 \cdot F^{\textup{GW}}(y,q)^{4}
= \sum_{\substack{n, \textbf{k} \\ \beta_{b,\textbf{k}} > 0}} \big\langle\, 1 \, \big\rangle_{0, \beta_{n,\textbf{k}}}^{\Km(\AAA)} q^n y^{\sum_{i,j} k_{ij}} \,.
\]
The evaluation \eqref{YZ_Kummer_jdhnr} and Proposition \ref{YZ_Theta_identity} yields
\[ F^{\textup{GW}}(y,q)^{4} = \frac{1}{\Delta(2 \tau)} \left( \frac{ \vartheta_1(z,\tau) \cdot \eta(2 \tau)^{6} }{ \eta(\tau)^3 } \right)^4 \,. \]
Since $\Delta(\tau) = \eta(\tau)^{24}$, we conclude
\[ F^{\textup{GW}}(y,q) = \pm \frac{ \vartheta_1(z,\tau) }{\eta^3(\tau)}. \]
By the definition of $F^{\textup{GW}}$ in Section \ref{concl_basic_case}, the coefficient of $y^{-1/2} q^0$ is $1$. Hence
\[ F^{\textup{GW}}(y,q) = \frac{ \vartheta_1(z,\tau) }{\eta^3(\tau)} = F(z, \tau). \qedhere \]
\end{proof}

\begin{proof}[Proof of Proposition \ref{YZ_Theta_identity}]
Both sides of \eqref{YZ_Theta_identity_eqn} are Jacobi forms of weight $2$ and index $1/2$ for a certain congruence subgroup of the Jacobi group.
The statement would therefore follow by the theory of Jacobi forms \cite{EZ} after comparing enough coefficients of both sides.
For simplicity, we will instead prove the statement directly.

We will work with the variables $q = e^{2 \pi i \tau}$ and $p = e^{2 \pi i z}$. Consider
\begin{equation} F(z,\tau) = \frac{\vartheta_1(z,\tau)}{\eta^3(\tau)} = -i(p^{1/2} - p^{-1/2}) \prod_{m \geq 1} \frac{ (1-pq^m) (1-p^{-1}q^m)}{ (1-q^m)^2 } \label{ffffpqdf}\end{equation}
By direct calculation one finds
\begin{equation}
\begin{aligned}
\label{15343r4g}
F( z+ \lambda \tau + \mu, \tau) & = (-1)^{\lambda + \mu} q^{-\lambda/2} p^{-\lambda} K(z,\tau) \\
\Theta(z + \lambda \tau + \mu, \tau) & = (-1)^{\lambda + \mu} q^{-\lambda/2} p^{-\lambda} \Theta(z,\tau) \,.
\end{aligned}
\end{equation}
We have
\begin{align*}
\Theta(0, \tau)
& = \sum_{x \in \BZ^4} \exp\Big( - 2 \pi i \Big\langle x + \frac{\alpha}{2} , \frac{e_1}{2} \Big\rangle \Big) q^{\left\langle x + \frac{\alpha}{2}, x + \frac{\alpha}{2} \right\rangle} \\
& = \sum_{x' \in \BZ^4 + \frac{\alpha}{2}}  \exp\big( - \pi i \langle x', e_1 \rangle \big) q^{\langle x', x' \rangle}
\end{align*}
Since for every $x' = m + \frac{\alpha}{2}$ with $m \in \BZ^4$ one has
\begin{align*}
\exp\big( - \pi i \langle x', e_1 \rangle \big) + \exp\big( - \pi i \langle - x', e_1 \rangle \big)
& = -i (-1)^{\langle m, e_1\rangle} + i (-1)^{- \langle m, e_1 \rangle} \\ & = 0,
\end{align*}
we find $\Theta(0,\tau) = 0$. By \eqref{ffffpqdf}, we also have $F(0, \tau) = 0$.

Since $\Theta$ and $F$ are Jacobi forms of index $1/2$ (see \cite[Theorem 1.2]{EZ}),
the point $z = 0$ is the only zero of $\Theta$ resp. $F$ in the standard fundamental region.
Therefore, the quotient
\[ \frac{\Theta(z,\tau)}{F(z,\tau)} \]
is a double periodic entire function, and hence a constant in $\tau$.
Using the evaluations
\[ F\left( \frac{1}{2}, \tau \right) = 2 \prod_{m \geq 1} \frac{(1+q^m)^2}{(1-q^m)^2} = 2 \frac{\eta(2 \tau)^2}{\eta(\tau)^4} \]
and 
\[ \Theta\left( \frac{1}{2} , \tau \right)
 = \sum_{x \in \BZ^4} (-1) q^{\langle x + \frac{\alpha}{2}, x + \frac{\alpha}{2} \rangle},
\]
the statement therefore follows directly from Lemma \ref{xmcisforg} below.
\end{proof}

\begin{lemma} \label{xmcisforg} We have
\[ \sum_{x \in \BZ^4} q^{\langle x + \frac{\alpha}{2}, x + \frac{\alpha}{2} \rangle} = 2 \, \frac{ \eta(2 \tau)^8 }{ \eta(\tau)^4 } \,. \]
\end{lemma}
\begin{proof}
As a special case of the Jacobi triple product \cite{Chandra}, we have
\[ 2 \frac{ \eta(2 \tau)^2 }{ \eta(\tau) } = 2 q^{1/8} \prod_{m \geq 1} \frac{(1-q^{2m})^{2}}{(1-q^m)} = \sum_{m \in \BZ} q^{(m + \frac{1}{2})^{2} / 2} \]
For $m = (m_1, \dots, m_4) \in \BZ^4$, let
\[
x_m =
\left( m_1 + \frac{1}{2} \right) \frac{\alpha}{2}
+ \left( m_2 + \frac{1}{2} \right) \frac{e_2}{2}
+ \dots + \left(m_4 + \frac{1}{2}\right) \frac{e_4}{2}
\]
Using that $\alpha, e_2, e_3, e_4$ are orthogonal, we find
\[
16 \, \frac{ \eta(2 \tau)^8 }{ \eta(\tau)^4 }
= \left( \sum_{m \in \BZ} q^{(m + \frac{1}{2})^{2} / 2} \right)^{4}
= \sum_{m \in \BZ^4} q^{\langle x_m, x_m \rangle}
\]
We split the sum over $m = (m_1, \dots, m_4) \in \BZ^4$
depending upon whether $m_1 + m_i$ is odd or even for $i=2,3,4$,
\begin{equation} \sum_{m \in \BZ^4} q^{\langle x_m, x_m \rangle}
 =
\sum_{s_2, s_3, s_4 \in \{0,1\}}
 \sum_{\substack{(m_1, \dots,m_r) \in \BZ^4 \\ m_1 + m_i \equiv s_i\, (2) }} q^{\langle x_m, x_m \rangle}
\label{dmsodvfsdg} 
\end{equation}
For every choice of $s_2, s_3, s_4 \in \{0,1\}$, we have
\begin{equation*} \sum_{\substack{(m_1, \dots,m_r) \in \BZ^4 \\ m_1 + m_i \equiv s_i (2) }} q^{\langle x_m, x_m \rangle}
 = \sum_{x \in \BZ^4} q^{\left\langle x + \frac{\beta}{2} , x + \frac{\beta}{2} \right\rangle},
\end{equation*}
where $\beta \in \BZ^4$ is a root of the $\mathsf{D}_4$-lattice (i.e. $\langle \beta, \beta \rangle = 2$).
Since the isometry group of $\mathsf{D}_4$ acts transitively on roots,
\[ \sum_{x \in \BZ^4} q^{\left\langle x + \frac{\beta}{2} , x + \frac{\beta}{2} \right\rangle} =
  \sum_{x \in \BZ^4} q^{\langle x + \frac{\alpha}{2}, x + \frac{\alpha}{2} \rangle}.
\]
Inserting this into \eqref{dmsodvfsdg} and dividing by $8$, the proof is complete.
\end{proof}

\section{Evaluation of further Gromov-Witten invariants}
\label{Chapter_More_Evaluations} \label{Section_More_Evaluations}

\subsection{Overview}
In Section \ref{Chapter_More_Evaluations} and Section \ref{Section_Hilb2P1xE} we prove Theorem \ref{MThm}

In Section~\ref{BL_reformulation} we reduce
the calculation to a Bryan-Leung K3.
We also state one extra evaluation on the Hilbert scheme of $2$ points of a K3 surface,
which is required in Section \ref{Section_Quantum_Cohomology}.
Next, for each case separately, we analyse the moduli space of maps which are incident to the given conditions.
In each case, the main result is a splitting statement similar to Proposition \ref{W0splitprop}.

As a result, the proof of Theorem \ref{MThm} is reduced to the calculation of certain universal contributions
associated to single elliptic fibers.
These contributions will be determined in Section~\ref{Section_Hilb2P1xE}
using the geometry of $\Hilb^2(\p^1 \times E)$, where $E$ is an elliptic curve.
The strategy is parallel but more difficult to the evaluation considered in Section \ref{basic_case}.

\subsection{Reduction to the Bryan-Leung K3} \label{BL_reformulation}
Let $\pi : S \to \p^1$ be an elliptic K3 surface with a unique section and $24$ nodal fibers.
Let $B$ and $F$ be the section and fiber class respectively, and let 
\[ \beta_h = B + h F \]
for $h \geq 0$. The quantum bracket $\langle \ \ldots \ \rangle_q$ on $\Hilb^d(S), d \geq 1$ is defined by
\[
 \big\langle \gamma_1, \dots, \gamma_m \big\rangle_q^{\Hilb^d(S)}
= \sum_{h \geq 0} \sum_{k \in \BZ} y^k q^{h-1} \langle \gamma_1, \dots, \gamma_m \rangle^{\Hilb^d(S)}_{\beta_h + k A} \,,
\]
where $\gamma_1, \dots, \gamma_m \in H^{\ast}(\Hilb^d(S))$ are cohomology classes.
By arguments parallel to Section \ref{K3statements},
Theorem \ref{MThm} is equivalent to the following Theorem.

\begin{thm} \label{ellthm2} Let $P_1, \dots, P_{2d-2} \in S$ be generic points. For $d \geq 2$,
\begin{align*}
\blangle C(F) \brangle^{\Hilb^d(S)}_q
& = \frac{G(z,\tau)^{d-1}}{\Delta(\tau)} \\[3pt]
\blangle A\, \brangle^{\Hilb^d(S)}_q
& = - \frac{1}{2} \Big( y \frac{d}{dy} G(z, \tau) \Big) \frac{G(z,\tau)^{d-2}}{\Delta(\tau)} \\[3pt]
\blangle I(P_1), \ldots, I(P_{2d-2}) \brangle^{\Hilb^d(S)}_q
& = \frac{1}{d} \binom{2d-2}{d-1} \Big( q \frac{d}{dq} F(z,\tau)\Big)^{2d-2} \frac{1}{\Delta(\tau)}
\end{align*}
under the variable change $q = e^{2 \pi i \tau}$ and $y = - e^{2 \pi i z}$.
\end{thm}

Later we will require one additional evaluation on $\Hilb^2(S)$.
Let $P \in S$ be a generic point and let
\[ \Fp_{-1}(F)^2 1_S \]
be the class of a generic fiber of $\pi^{[2]} : \Hilb^2(S) \to \p^2$.
\begin{thm} \label{extra_eval} Under the variable change $q = e^{2 \pi i \tau}$ and $y = - e^{2 \pi i z}$,
\[ \blangle \Fp_{-1}(F)^2 1_S ,\, I(P) \brangle^{\Hilb^2(S)}_q = \frac{F(z,\tau) \cdot q \frac{d}{dq} F(z,\tau) }{ \Delta(\tau) } \]
\end{thm}

\subsection{Case $\langle C(F) \rangle_q$} \label{CASE1}
We consider the evaluation of $\big\langle C(F) \big\rangle_q^{\Hilb^d(S)}$.
Let $P_1, \dots, P_{d-1} \in S$ be generic points, let $F_0$ be a generic fiber
of the elliptic fibration $\pi : S \to \p^1$, and let
\[ Z = F_0[1] P_1[1] \cdots P_{d-1}[1] \subset \Hilb^d(S) \]
be the induced subscheme of class $[Z] = C(F)$, where we used the notation of Section \ref{special_cycles}~(v).
Consider the evaluation map
\begin{equation} \ev : \Mbar_{0,1}( \Hilb^d(S), \beta_h + kA) \to S, \label{evalmap_lkvkffgk}\end{equation}
the moduli space parametrizing maps incident to the subscheme $Z$
\begin{equation} M_Z = \ev^{-1}(Z) \,, \label{301} \end{equation}
and an element
\[ [f : C \to \Hilb^d(S), p] \in M_Z \,. \]
By Lemma \ref{200}, there does not exist a non-constant genus $0$ stable map to $\Hilb^d(S)$ of class $h' F + k'A$ which is incident to $Z$.
Hence, the marking $p \in C$ must lie on the distinguished irreducible component
\[ C_0 \subset C \]
on which $\pi^{[d]} \circ f$ is non-constant. By Lemma \ref{117}, the restriction $f|_{C_0}$ is therefore an isomorphism
\begin{equation}
\begin{aligned}
f|_{C_0}\colon  C_0 & \to B_0[1] P_1[1] \cdots P_{d-1}[1] \\
& \quad \quad \quad \quad = I(B_0) \cap I(P_1) \cap \ldots \cap I(P_{d-1}) \subset \Hilb^d(S) \,,
\label{v5gkfkf}
\end{aligned}
\end{equation}
where $B_0$ is the section of $S \to \p^1$.
In particular, $f(p) = (F_0 \cap B_0) + \sum_j P_j$.
We identify $C_0$ with its image in $\Hilb^d(S)$.

Let $x_1, \dots, x_{24}$ be the basepoints of the rational nodal fibers of $\pi$ and let $u_i = \pi(P_i)$ for all $i$.
The image line $L = \pi^{[d]} \circ f(C)$ meets the discriminant locus
of $\pi^{[d]}$ in the points
\[ x_i + \sum_{j=1}^{d-1} u_j\ \ (i = 1,\dots, 24) \quad \text{ and } \quad 2 u_i + \sum_{j \neq i} u_j \ \ (i =1,\dots, d-1) \]
By Lemma \ref{200}, the curve $C$ is therefore of the form
\[ C = C_0 \cup A_1 \cup \ldots \cup A_{24} \cup B_1 \cup \ldots \cup B_{d-1} \]
where the components $A_i$ and $B_j$ are attached to the points
\begin{equation} x_i + P_1 + \dots + P_{d-1} \quad \text{ and } \quad u_j + P_1 + \ldots + P_{d-1} \label{fowjroirgfrg} \end{equation}
respectively.
Hence, the moduli space $M_Z$ is set-theoretically a product of spaces parametrizing maps
of the form $f|_{A_i}$ and $f|_{B_j}$ respectively.
We show that the set-theoretic product is scheme-theoretic and the virtual class splits.
The argument is similar to Section \ref{basic_case}.

First, the attachment points \eqref{fowjroirgfrg} do not smooth under infinitesimal deformations:
this follows since the projection
\[ f^{\ast} \CZ_d = \widetilde{C} \to C \]
is \'etale over the points \eqref{fowjroirgfrg}, see the proof of Lemma \ref{doesnotsmooth};
here $\CZ_d \to \Hilb^d(S)$ is the universal family.
Therefore, any infinitesimal deformation of $f$ inside $M_Z$ induces a deformation of the image $f(C_0)$.
This deformation corresponds to moving the points $P_1, \dots, P_{d-1}$ in \eqref{v5gkfkf},
which is impossible since $f$ continues to be incident to $Z$.
Hence, $f(C_0)$ is fixed under infinitesimal deformations.\footnote{
Although $f(C_0)$ is fixed under infinitesimal deformations, the point $u_j$ in the attachment point $f(C_0 \cap B_j) = u_j + P_1 + \dots + P_{d-1}$
may move to first order, compare Section~\ref{Section_CaseI_P}.}

By a construction parallel to Section \ref{splitting}, we have a splitting map
\begin{equation}
\Psi : M_Z \to \bigsqcup_{(\textbf{h}, \textbf{k})} 
\bigg( \prod_{i=1}^{24} M^{\textup{(N)}}_{x_i}(h_{x_i})
\times \prod_{j=1}^{d-1} M^{\textup{(G)}}_{u_j}(h_{y_j}, k_{y_j}) \bigg),
\label{vfivofvg}
\end{equation}
where $M^{\textup{(N)}}_{x_i}(h_{x_i})$ was defined in Section \ref{splitting},
and for an appropriately defined moduli space $M^{\textup{(G)}}_{u_j}(h_{y_j}, k_{y_j})$;
since $f(C_0)$ has class $B$, the disjoint union in \eqref{vfivofvg} runs over all
\begin{equation}
\begin{aligned}
\textbf{h} & = (h_{x_1}, \dots, h_{x_{24}}, h_{u_1}, \dots, h_{u_{d-1}} ) \in (\BN^{\geq 0})^{ \{ x_i, u_j \}} \\
\textbf{k} & = (k_{u_1}, \dots, k_{u_{d-1}}) \in \BZ^{d-1}
\end{aligned}
\end{equation}
such that $\sum_i h_{x_i} + \sum_j h_{u_j} = h$ and $\sum_j k_{u_j} = k$.
Since $f(C_0)$ is fixed under infinitesimal deformations, the map $\Psi$ is an isomorphism.

Let $[ M_Z ]^{\text{vir}}$ be the natural virtual class on $M_Z$. By arguments parallel to Section \ref{anvirclass},
the pushforward $\Psi_{\ast} [M_Z]^{\text{vir}}$ is a product of virtual classes defined on each factor.
Hence, by a calculation identical to Section \ref{concl_basic_case},
$\langle C(F) \rangle_q$ is the product of series corresponding to the points $x_i$ and $u_j$ respectively.

For the points $x_1, \dots, x_{24}$,
the contributing factor agrees with the contribution
from the nodal fibers in the case of Section \ref{Section_higher_dimensional_Yau_Zaslow}. It is the series~\eqref{nodal_fiber_contribution}.
For $u_1, \dots, u_{d-1}$ define the formal series
\begin{equation} G^{\textup{GW}}(y,q) = \sum_{h \geq 0} \sum_{k \in \BZ} y^k q^{h} \int_{[M^{\textup{(G)}}_{u_j}(h, k)]^{\text{vir}} } 1, \label{GGGdef} \end{equation}
where we let $[ M^{\textup{(G)}}_{u_j}(h, k) ]^{\text{vir}}$ denote the induced virtual class on $M^{\textup{(G)}}_{u_j}(h, k)$.
We conclude
\begin{equation} \blangle C(F) \brangle_q^{\Hilb^d(S)} = \frac{G^{\textup{GW}}(y,q)^{d-1}}{\Delta(q)} \,. \label{401} \end{equation}

\subsection{Case $\langle A \rangle_q$} \label{ABCD}
We consider the evaluation of $\langle A\, \rangle_q^{\Hilb^d(S)}$.
Let $P_0, \ldots , P_{d-2} \in S$ be generic points, let 
\[ Z = P_0[2] P_1[1] \cdots P_{d-2}[1] \subset \Hilb^d(S) \]
be the exceptional curve (of class $A$) centered at $2 P_0 + P_1, \dots + P_{d-2}$, and let
\[ M_Z = \ev^{-1}(Z), \]
where $\ev$ is the evaluation map \eqref{evalmap_lkvkffgk}.
We consider an element
\[ [f : C \to \Hilb^d(S), p] \in M_Z. \]
Let $C_0 \subset C$ be the distinguished component of $C$ on which $\pi^{[d]} \circ f$ is non-constant,
and let $C'$ be the union of all irreducible components of $C$ which map into the fiber
\[ (\pi^{[d]})^{-1}( 2 u_0 + u_1 + \ldots + u_{d-2} ), \]
where $u_i = \pi(P_i)$.
Since $f(C_0)$ cannot meet the exceptional curve $Z$,
the component $C'$ contains the marked point $p$,
\[ p \in C'. \]
The restriction $f|_{C'}$ decomposes into the components
\[ f|_{C'} = \phi + P_1 + \dots + P_{d-2}, \]
where $\phi : C' \to \Hilb^2(S)$ maps into the fiber $\pi^{[2] -1}(2 u_0)$
and the $P_i$ denote constant maps.

Consider the Hilbert-Chow morphism
\[ \rho : \Hilb^2(S) \to \Sym^2(S) \]
and the Abel-Jacobi map
\[ \mathsf{aj} : \Sym^2(F_{u_0}) \to F_{u_0} \,. \]
Since $\rho(\phi(p)) = 2 P_0$, the image of $\phi$ lies inside the fiber $V$ of
\[ \rho^{-1}(\Sym^2(F_{u_0}) ) \xrightarrow{\rho} \Sym^2(F_{u_0}) \xrightarrow{\mathsf{aj}} F_{u_0} \]
over the point $\mathsf{aj}(2 P_0)$.
Hence, $f|_{C'}$ maps into the subscheme
\[ \widetilde{V} = V + P_1 + \ldots + P_{d-2} \subset \Hilb^d(S)\,. \]

The intersection of $\widetilde{V}$ with the divisor $D(B_0) \subset \Hilb^d(S)$ is supported in the reduced point 
\begin{equation} s(u_0) + Q + P_1 + \dots + P_{d-2} \in \Hilb^d(S), \label{3ofidofg} \end{equation}
where $s : \p^1 \to S$ is the section and $Q \in F_{u_0}$ is defined by \[ \mathsf{aj}(s(u_0) + Q) = \mathsf{aj}( 2 P_0 ) \,. \]
Since the distinguished component $C_0 \subset C$ must map into $D(B_0)$,
the point $f(C_0 \cap C')$ therefore equals \eqref{3ofidofg}.
Hence, the restriction $f|_{C_0}$ yields an isomorphism
\[ f|_{C_0} : C_0 \xrightarrow{\ \cong\ } B_0[1] Q[1] P_1[1] \cdots P_{d-2}[1]\,, \]
and we will identify $C_0$ with its image.

Following the lines of Section \ref{CASE1}, we find that the domain $C$ is of the form
\[ C = C_0 \cup C' \cup A_1 \cup \ldots \cup A_{24} \cup B_1 \cup \ldots \cup B_{d-2}, \]
where the components $A_i$ and $B_j$ are attached to the points
\[ x_i + Q + P_1 + \dots + P_{d-2}, \quad \quad \quad u_j + Q + P_1 + \dots + P_{d-2} \]
respectively.
Hence, $M_Z$ is set-theoretically a product of spaces corresponding to the points 
\begin{equation} u_0, u_1,\dots, u_{d-2}, x_1, \dots, x_{24}. \label{vkkmfvmf} \end{equation}
By arguments parallel to Section \ref{CASE1}, the moduli scheme $M_Z$ splits scheme-theoretic as a product,
and also the virtual class splits.
Hence, $\langle A \, \rangle_q$ is a product of series corresponding to the points \eqref{vkkmfvmf} respectively.

For $x_1, \dots, x_{24}$
the contributing factor is the same as in Section \ref{concl_basic_case},
and for $u_1, \dots, u_{d-2}$ it is the same as in Section \ref{CASE1}. Let
\begin{equation} \widetilde{G}^{\textup{GW}}(y,q) \in \BQ((y))[[q]] \label{gggt} \end{equation}
denote the contributing factor from the point $u_0$. Then we have
\begin{equation} \blangle A \, \brangle_q^{\Hilb^d(S)} = \frac{ G^{\textup{GW}}(y,q)^{d-2} \widetilde{G}^{\textup{GW}}(y,q) }{\Delta(q) }. \label{402} \end{equation}

\subsection{Case $\langle I(P_1), \dots, I(P_{2d-2}) \rangle_q$}
Let $P_1, \dots, P_{2d-2} \in S$ be generic points. In this section, we consider the evaluation of
\begin{equation} \blangle I(P_1), \dots, I(P_{2d-2}) \brangle_q^{\Hilb^d(S)} \label{IP_eval_xcxczcxcxc} \end{equation}
In Section \ref{Section_intermediate_lemma}, we discuss the geometry of lines in $\Hilb^d(\p^1)$. 
In Section \ref{special_case_3}, we analyse the moduli space
of stable maps incident to $I(P_1), \dots, I(P_{2d-2})$.

\subsubsection{The Grassmannian}\label{Section_intermediate_lemma}
Let $\CZ_d \to \Hilb^d(\p^1)$ be the universal family, and let
\[ L \hookrightarrow \Hilb^d(\p^1) \]
be the inclusion of a line such that $L \nsubseteq I(x)$ for all $x \in \p^1$.
Consider the fiber diagram
\[
\begin{tikzcd}
\widetilde{L} \ar{d} \ar{r} & \CZ_d \ar{d} \ar{r} & \p^1  \\
L \ar{r} & \Hilb^d(\p^1).
\end{tikzcd}
\]
The curve $\widetilde{L} \subset L \times \p^1$ has bidegree $(d,1)$, and is the graph of the morphism
\begin{equation} I_L \colon \p^1 \to L,\ x \mapsto I(x) \cap L \,. \label{MapI_Lasd}\end{equation}
By definition, the subscheme corresponding to a point $y \in L$ is $I_L^{-1}(y)$.
Hence, the ramification index of $I_L$ at a point $x \in \p^1$ is
the length of $I_L(x)$ (considered as a subscheme of $\p^1$) at $x$.
In particular, for $y \in L$, we have $y \in \Delta_{\Hilb^d(\p^1)}$
if and only if $I_L(x) = y$ for a branchpoint $x$ of $I_L$.

Let $R(L) \subset \p^1$
be the ramification divisor of $I_L$.
Since $I_L$ has $2d-2$ branch points counted with multiplicity
(or equivalently, $L$ meets $\Delta_{\Hilb^d(\p^1)}$ with multiplicity $2d-2$),
\[ R(L) \in \Hilb^{2d-2}(\p^1). \]

Let $G = G(2,d+1)$ be the Grassmannian of lines in $\Hilb^d(\p^1)$.
By the construction above relative to $G$,
we obtain a rational map
\begin{equation} \phi : G \dashrightarrow \Hilb^{2d-2}(\p^1),\ L \mapsto R(L) \label{303} \end{equation}
defined on the open subset of lines $L \in G$ with
$L \nsubseteq I(x)$ for all $x \in \p^1$.

The map $\phi$ will be used in the proof of the following result.
For $u \in \p^1$, consider the incidence subscheme
\[ I(2u) = \{ z \in \Hilb^d(\p^1)\ |\ 2u \subset z \} \]
Under the identification $\Hilb^d(\p^1) \equiv \p^d$,
the inclusion $I(2u) \subset \Hilb^d(\p^1)$ is a linear subspace of codimension $2$. 
Let
\begin{equation}
\label{Grassmann_diagram}
 \xymatrix{
\CZ \ar[r]^{q} \ar[d]^{p} & \, \p^d\, \mathrlap{= \Hilb^d(\p^1)} \\
G
}
\end{equation}
be the universal family of $G$, and let
\[ S_u = p(q^{-1}(I(2u))) = \{ L \in G | L \cap I(2u) \neq \varnothing \} \subset G \]
be the divisor of lines incident to $I(2u)$.

\begin{lemma} \label{intermediate_lemma}
Let $u_1, \dots, u_{2d-2} \in \p^1$ be generic points. Then,
\begin{equation} S_{u_1} \cap \ldots \cap S_{u_{2d-2}} \label{Su_intersection} \end{equation}
is a collection of $\frac{1}{d} \binom{2d-2}{d-1}$ reduced points.
\end{lemma}
\begin{proof}
The class of $S_u$ is the Schubert cycle $\sigma_1$.
By Schubert calculus the expected number of intersection points is
\begin{equation*} \int_G \sigma_1^{2d-2} = \frac{1}{d} \binom{2d-2}{d-1}. \end{equation*}
It remains to prove that the intersection \eqref{Su_intersection} is transverse.

Given a line $L \subset I(x) \subset \Hilb^d(\p^1)$ for some $x \in \p^1$,
there exist at most $2d-1$ different points $v \in \p^1$ with $2 v \subset z$ for some $z \in L$.
Hence, for every $L$ in \eqref{Su_intersection} we have $L \nsubseteq I(x)$ for all $x \in \p^1$.
Therefore, $S_{u_1} \cap \ldots \cap S_{u_{2d-2}}$ lies in the domain of $\phi$.
Then, by construction of $\phi$, the intersection \eqref{Su_intersection}
is the fiber of $\phi$ over the point
\[ u_1 + \dots + u_{2d-2} \in \Hilb^{2d-2}(\p^1). \]
We will show that $\phi$ is generically finite.
Since $u_1, \dots, u_{2d-2}$ are generic,
the fiber over $u_1 + \dots + u_{2d-2}$ is then a set of finitely many reduced points.

We determine an explicit expression for the map $\phi$.
Let $L \in G$ be a line with $L \nsubseteq I(x)$ for all $x \in \p^1$,
let $f,g \in L$ be two distinct points and let $x_0, x_1$ be coordinates on $\p^1$.
We write
\begin{align*}
 f & = a_n x_0^n + a_{n-1} x_0^{n-1} x_1 + \dots + a_0 x_1^n \\
 g & = b_n x_0^n + b_{n-1} x_0^{n-1} x_1 + \dots + b_0 x_1^n
\end{align*}
for coefficients $a_i, b_i \in \BC$.
The condition $L \nsubseteq I(x)$ for all $x$ is equivalent to
$f$ and $g$ having no common zeros. Consider the rational function
\[ h(x) = h(x_0/x_1) = f/g = \frac{ a_n x^n + \dots + a_0 }{ b_n x^n + \dots + b_0 }, \]
where $x = x_0/x_1$. The ramification divisor $R(L)$ is generically
the zero locus of the nominator of $h' = (f/g)' = (f' g - f g')/g^2$; in coordinates we have
\[ f' g - f g' = \sum_{m = 0}^{2d-2} \Big( \sum_{ \substack{ i + j = m+1 \\ i < j }} (i - j) (a_i b_j - a_j b_i) \Big) x^m. \]
Let $M_{ij} = a_i b_j - a_j b_i$ be the Pl\"ucker coordinates on $G$. Then we conclude
\[ 
\phi(L) = \sum_{m = 0}^{2d-2} \Big( \sum_{ \substack{ i + j = m+1 \\ i < j }} (i - j) M_{ij} \Big) x^m \ \in \Hilb^{2d-2}(\p^1) \,. 
\]
By a direct verification, the differential of $\phi$ at the point with coordinates
\[ (a_0, \dots, a_n) = (1, 0, \dots, 0 , 1), \quad \quad(b_0, \dots, b_n) = (0,1,0,\dots,0,1) \]
is an isomorphism. Hence, $\phi$ is generically finite.
\end{proof}

Let $u_1, \dots, u_{2d-2} \in \p^1$ be generic points. Consider a line
\[ L\ \in\  S_{u_1} \cap \ldots \cap S_{u_{2d-2}} = \phi^{-1}(u_1 + \dots + u_{2d-2}) \]
and let $U_L$ be the formal neighborhood of $L$ in $G$.
By the proof of Lemma~\ref{intermediate_lemma}, the map
\[ \phi : G \dashrightarrow \Hilb^{2d-2}(\p^1) \]
is \'etale near $L$.
Hence, $\phi$ induces an isomorphism from $U_L$ to
\begin{equation} \Spec\big( \widehat{\CO}_{\Hilb^d(\p^1),u_1+\dots+u_{2d-2}} \big) \equiv \prod_{i=1}^{2d-2} \Spec ( \widehat{\CO}_{\p^1,u_i} ) \,, \label{sos34mkgmrg}\end{equation}
the formal neighborhood of $\Hilb^d(\p^1)$ at $u_1+\dots+u_{2d-2}$.
Composing $\phi$ with the projection to the $i$-th factor of \eqref{sos34mkgmrg}, we obtain maps
\begin{equation} \kappa_{i} : U_L \xrightarrow{\ \phi\ }\Spec\big( \widehat{\CO}_{\Hilb^d(\p^1),u_1+\dots+u_{2d-2}} \big) \to \Spec ( \widehat{\CO}_{\p^1,u_i} ) \subset \p^1 \,,
 \label{kappa_map_def}
\end{equation}
which parametrize the deformation of the branch points of $I_L$ (defined in~\eqref{MapI_Lasd}).

In the notation of the diagram \eqref{Grassmann_diagram}, consider the map 
\begin{equation}
q^{-1}(\Delta_{\Hilb^d(\p^1)})\ \to\ G \label{408} \end{equation}
whose fiber over a point $L' \in G$
are the intersection points of $L'$
with the diagonal $\Delta_{\Hilb^d(\p^1)}$.
Since $L$ is in the fiber of a generically finite map over a generic point,
we have
\[ L \cap \Delta_{\Hilb^d(\p^1)} = \{ \xi_1, \dots, \xi_{2d-2} \} \]
for pairwise disjoint subschemes $\xi_i \in \Hilb^d(\p^1)$ of type $(21^{d-2})$ with $2 u_i \subset \xi_i$.
The restriction of \eqref{408} to $U_L$ is a $(2d-2)$-sheeted trivial fibration, and hence admits sections
\begin{equation} v_1, \dots, v_{2d-2} : U_L \to q^{-1}(\Delta_{\Hilb^d(\p^1)})|_{U_L}, \label{ifimsoidf} \end{equation}
such that for every $i$ the composition $q \circ v_i$ restricts to $\xi_i$ over the closed point.
Moreover, since $q \circ v_i$ is incident to the diagonal and must contain
twice the branchpoint $\kappa_i$ defined in \eqref{kappa_map_def},
we have the decomposition
\begin{equation} q \circ v_i = 2 \kappa_i + h_1 + \dots + h_{d-2} \label{fatt_kfmvodmfo} \end{equation}
for maps $h_1, \dots, h_{d-2} : U_L \to \p^1$.

\subsubsection{The moduli space} \label{special_case_3}
Let $P_1, \dots, P_{2d-2} \in S$ be generic points and let $u_i = \pi(P_i)$ for all $i$. Let
\[ \ev : \Mbar_{0,2d-2}(\Hilb^d(S), \beta_h + kA) \to \big(\Hilb^d(S)\big)^{2d-2} \]
be the evaluation map and let
\begin{equation*} M_Z = \ev^{-1}\bigl(I(P_1) \times \dots \times I(P_{2d-2})\bigr)\end{equation*}
be the moduli space of stables maps incident to $I(P_1), \dots, I(P_{2d-2})$.
We consider an element
\[ [f : C \to \Hilb^d(S), p_1, \dots, p_{2d-2}] \in M_Z \,. \]

Since $P_i \in f(p_i)$ and $P_i$ is generic, the line
$L = \pi(f(C)) \subset \Hilb^d(\p^1)$ is incident to $I(2u_i)$ for all $i$,
and therefore lies in the finite set
\begin{equation}
 S_{u_1} \cap \dots \cap S_{u_{2d-2}} \subset G(2,d+1)
\label{Su_intersection_new}
\end{equation}
defined in Section \ref{Section_intermediate_lemma};
here $G(2,d+1)$ is the Grassmannian of lines in $\p^d$.

Because the points $u_1, \dots, u_{2d-2}$ are generic,
by the proof of Lemma~\ref{intermediate_lemma}
also $L$ is generic.
By arguments identical to the case of Section \ref{basic_case_settheoretic_splitting},
the map $f|_{C_0} : C_0 \to L$ is an isomorphism.
We identify $C_0$ with the image $L$.

For $x \in \p^1$, let $\widetilde{x} = I(x) \cap L$ be the unique point on $L$ incident to~$x$.
The points 
\begin{equation} \label{L_with_W_int_points_2} \widetilde{x}_1, \dots, \widetilde{x}_{24}, \widetilde{u}_1, \dots, \widetilde{u}_{2d-2} \end{equation}
are the intersection points of $L$ with
the discriminant of $\pi^{[d]}$ defined in \eqref{501}.
Hence, by Lemma \ref{200}, the curve $C$ admits the decomposition
\begin{equation*} C = C_0 \cup A_1 \cup \dots \cup A_{24} \cup B_1 \cup \dots \cup B_{2d-2}, \label{ogorgdfgfdg2} \end{equation*}
where $A_i$ and $B_j$ are the components of $C$ attached to the points
$\widetilde{x}_i$ and $\widetilde{u}_j$ respectively; see also Section \ref{basic_case_settheoretic_splitting}.

By Lemma \ref{doesnotsmooth}, the node points $C_0 \cap A_i$ and $C_0 \cap B_j$ do not smooth
under deformations of $f$ inside $M_Z$.
Hence, by the construction of Section \ref{splitting},
we have a splitting morphism
\begin{equation}
\Psi: M_{Z}
\ra
\bigsqcup_{L} \bigsqcup_{\textbf{h}, \textbf{k}} \bigg( \prod_{i=1}^{24} M^{\textup{(N)}}_{x_i}(h_{x_i})
\times \prod_{j=1}^{2d-2} M^{\textup{(H)}}_{u_j}(h_{u_j}, k_{u_j}) \bigg),
\label{405} \end{equation}
where $\textbf{h}, \textbf{k}$ runs over the set \eqref{spl1_indexing} (with $y_j$ replaced by $u_j$) satisfying \eqref{spl2_indexing},
and $L$ runs over the set of lines \eqref{Su_intersection_new},
and where $M^{\textup{(H)}}_{u_j}(h', k')$ is the moduli space defined as follows:

Consider the evaluation map
\[ \ev : \Mbar_{0,2}( \Hilb^2(S) , h' F + k' A ) \ra (\Hilb^2(S))^2 \]
and let
\begin{equation} \ev^{-1}\big( I(P_j) \times \Hilb^2(B_0) \big) \label{erewrewg} \end{equation}
be the subscheme of maps incident to $I(P_j)$ and $\Hilb^d(B_0)$ at the marked points.
We define $M^{\textup{(H)}}_{u_j}(h', k')$ to be the open and closed component of \eqref{erewrewg}
whose $\BC$-points parametrize maps into the fiber $\pi^{[2] -1}(2 u_j)$.
Using this definition, the map $\Psi$ is well-defined
(for example, the intersection point $C_0 \cap B_j$ maps to the second marked point in \eqref{erewrewg}).

In the case considered in Section \ref{basic_case},
the image line $L = f(C_0)$ was fixed under infinitesimal deformations.
Here, this does not seem to be the case; the line $L$ may move infinitesimal.
Nonetheless, the following Proposition shows that these deformations are all captured by the image of $\Psi$.

\begin{proposition} \label{352_splitting_lemma}
The splitting map \eqref{405} is an isomorphism.
\end{proposition}

We will require the following Lemma, which will be proven later.
\begin{lemma} \label{410_lemma} Let $\phi : C \to \Hilb^2(S)$ be a family in $M^{\textup{(H)}}_{u_j}(h', k')$ over a connected scheme $Y$,
\begin{equation} \label{fam_diag}
\begin{tikzcd}
C \ar{r}{\phi} \ar{d} & \Hilb^2(S) \,. \\ Y
\end{tikzcd}
\end{equation}
Then $\pi^{[2]} \circ \phi$ maps to $\Hilb^2(\p^1) \cap I(u_j)$.
\end{lemma}

\begin{proof}[Proof of Proposition \ref{352_splitting_lemma}]
We define an inverse to $\Psi$. Let
\begin{equation} \Big((\phi_{i}' : A_i \to S,q_{x_i})_{i = 1, \dots,24} , (\phi_{j} : B_j \to \Hilb^2(S), p_j, q_j)_{j=1,\dots, 2d-2}\Big) \label{410} \end{equation}
be a family of maps in the right hand side of \eqref{405} over a connected scheme~$Y$.
By Lemma \ref{410_lemma}, $\pi^{[2]} \circ \phi_{j} : B_j \to \Hilb^2(\p^1)$
maps into $I(u_j) \cap \Delta_{\Hilb^2(\p^1)}$.
Since the intersection of the line $I(u_j)$
and the diagonal $\Delta_{\Hilb^2(S)}$ is infinitesimal,
we have the inclusion 
\[ I(u_j) \cap \Delta_{\Hilb^2(\p^1)} \hookrightarrow \Spec( \widehat{\CO}_{\Delta_{\Hilb^2(\p^1)},2 u_j} ) = \Spec ( \widehat{\CO}_{\p^1, u_j} ), \]
and therefore the induced map $\iota_j \colon Y \to \Spec ( \widehat{\CO}_{\p^1, u_i} )$
making the diagram
\[
\begin{tikzcd}
B_j \ar{d} \ar{dr}{\pi^{[2]} \circ \phi_j} &  \\
Y \ar{r}{\iota_j} & 
\Spec( \widehat{\CO}_{\Delta_{\Hilb^2(\p^1)},2 u_j} )\mathrlap{\, = \Spec ( \widehat{\CO}_{\p^1, u_j} )}
\end{tikzcd}
\]
commutative.
Let $\ell = (\iota_j)_{j} \colon Y \to U_L$, where
\[ U_L = \textstyle{\prod}_{j = 1}^{2d-2} \Spec( \widehat{\CO}_{\p^1, u_j} ) \equiv \Spec( \widehat{\CO}_{\Hilb^{2d-2}(\p^1), \sum_{i} u_i} ). \]

Under the generically finite rational map
\[ G(2, d+1) \dashrightarrow \Hilb^{2d-2}(\p^1), \]
defined in \eqref{303},
the formal scheme $U_L$ is isomorphic
to the formal neighborhood of $G(2,d+1)$ at the point $[L]$.
We identify these neighborhoods under this isomorphism.

Let $\CZ_L \to U_L$ be the restriction of the universal family $\CZ \to G(2,d+1)$ to~$U_L$.
By pullback via $\ell$, we obtain a family of lines in $\p^d$ over the scheme $Y$,
\begin{equation} \ell^{\ast} \CZ_L \to Y, \label{dwofmoofdg} \end{equation}
together with an induced map
\[ \psi : \ell^{\ast} \CZ_L \xrightarrow{\ell} \CZ_L \to \p^d \equiv \Hilb^d(B_0) \xrightarrow{s^{[d]}} \Hilb^d(S). \]

We will require sections of $\ell^{\ast} \CZ_L \to Y$,
which allow us to glue the domains of the maps $\phi_{i}'$ and $\phi_{j}$ to $\ell^{\ast} \CZ_L$.
Consider the sections
\[ v_1, \dots, v_{2d-2} : Y \to \ell^{\ast} \CZ_L \]
which are the pullback under $\ell$ of the sections $v_i : U_L \to \CZ_L$ defined in \eqref{ifimsoidf}.
By construction, the section $v_i : Y \to \ell^{\ast} \CZ$
parametrizes the points of $\ell^{\ast} \CZ_L$ which map to the diagonal $\Delta_{\Hilb^d(S)}$
under $\psi$
(in particular, over closed points of $Y$ they map to $I(u_j) \cap L$).

For $j=1,\dots,2d-2$, consider the family of maps $\phi_j : B_j \to \Hilb^2(S)$,
\begin{equation}
 \xymatrix{
B_j \ar@<+6pt>[d]^{\pi_j} \ar[r]^{\phi_{j}} & \mathrlap{\Hilb^2(S)} \\
Y \ar@<+6pt>[u]^{p_j, q_j} \ar@<+0pt>[u],
}
\label{411} \end{equation}
where $p_j$ is the marked point mapping to $I(P_j)$,
and $q_j$ is the marked point mapping to $\Hilb^2(B_0)$.
Let $C'$ be the curve over $Y$
which is obtained by glueing the component $B_j$ to the line $\ell^{\ast} \BZ_L$ along
the points $q_j, v_j$ for all $j$:
\[ C' = \Big( \ell^{\ast} \BZ_L \sqcup B_1 \sqcup \ldots \sqcup B_{2d-2} \Big)\Big/ q_1 \sim v_1, \dots, q_{2d-2} \sim v_{2d-2} \,. \]
We will define a map $f' : C' \to \Hilb^d(S)$.

For all $j$,
let $\kappa_j : U_L \to \Spec ( \widehat{\CO}_{\p^1,u_j} ) \subset \p^1$ be the map defined in \eqref{kappa_map_def}.
By construction, we have $\kappa_{j} \circ \ell = \iota_j$.
Hence, by \eqref{fatt_kfmvodmfo}, there exist maps $h_1, \dots, h_{d-2} : Y \to S$ with
\begin{equation}
\psi \circ v_j = \phi_{j} \circ q_j + h_1 + \dots + h_{d-2} \,.
\label{fomvodmvfd}
\end{equation}
Let $\pi_j : B_j \to Y$ be the map of the family $B_j/Y$, and define
\[ \widetilde{\phi}_{u_j} = \Big( \phi_{j} + \sum_{i=1}^{n-2} h_i \circ \pi_j\Big) : B_j \ra \Hilb^d(S). \]
Define the map
\[ f' : C' \to \Hilb^d(S) \]
by $f'|_{C_0} = \psi$ and by $f'|_{B_j} = \widetilde{\phi}_{j}$ for every $j$.
By \eqref{fomvodmvfd}, the map $\widetilde{\phi}_{u_j}$ restricted to $q_j$
agrees with $\psi : C' \to \Hilb^d(S)$ restricted to $v_j$.
Hence $f'$ is well-defined.

By a parallel construction, we obtain a canonical glueing of the components $A_i$ to $C'$
together with a glueing of the maps $f'$ and $\phi_i' : A_i \to S$.
We obtain a family of maps
\[ f : C \to \Hilb^d(S) \]
over $Y$, which lies in $M_Z$ and such that $\Psi(f)$ equals \eqref{410}.
By a direct verification, the induced morphism on the moduli spaces is the desired inverse to~$\Psi$.
Hence, $\Psi$ is an isomorphism.
\end{proof}

The remaining steps in the evaluation of \eqref{IP_eval_xcxczcxcxc} are similar to Section~\ref{basic_case}.
Using the identification 
\[ H^0(C_0, f^{\ast} T_{\Hilb^d(S)}) = H^0(C_0, T_{C_0}) \oplus \bigoplus_{j=1}^{2d-2} T_{\Delta_{\Hilb^2(\p^1)}, \phi_j(q_j)}, \]
where $q_j = C_0 \cap B_j$ are the nodes and $\phi_j$ is as in the proof of Proposition~\ref{352_splitting_lemma},
one verifies that the virtual class splits according to the product \eqref{405}.
Hence, the invariant \eqref{IP_eval_xcxczcxcxc} is a product of series associated to the points $x_i$ and $u_j$ respectively.
Let
\begin{equation}
H^{\textup{GW}}(y,q) = \sum_{h \geq 0} \sum_{k \in \BZ} y^{k - \frac{1}{2}} q^{h} \int_{[M^{\textup{(H)}}_{u_j}(h, k)]^{\text{vir}}} 1 \in \BQ((y^{1/2}))[[q]],
\label{hhh} \end{equation}
be the contribution from the point $u_j$.
By Lemma \ref{intermediate_lemma}, there are $\frac{1}{d} \binom{2d-2}{d-1}$ lines in the set \eqref{Su_intersection_new}. Hence,
\begin{equation}
\blangle I(P_1), \dots, I(P_{2d-2}) \brangle_q^{\Hilb^d(S)}
=
\frac{1}{d} \binom{2d-2}{d-1}
\frac{ H^{\textup{GW}}(y,q)^{2d-2} }{\Delta(q)}.
\label{403}
\end{equation}

\begin{proof}[Proof of Lemma \ref{410_lemma}]
Since $\phi$ is incident to $I(P_j)$, the composition
$\pi^{[2]} \circ \phi$ maps to $I(u_j)$.
Therefore, we only need to show that $\pi^{[2]} \circ \phi$ maps to $\Delta_{\Hilb^2(\p^1)}$.

It is enough to consider the case $Y = \Spec( \BC[\epsilon]/\epsilon^2 )$. Let
$f_0 : C_0 \to \Hilb^2(S)$ be the restriction of $f$ over the closed point of $Y$,
and consider the diagram
\[
\begin{tikzcd}
C \ar{d}{\pi_C} \ar{r}{\phi} & \Hilb^2(S) \ar{d}{\pi^{[2]}} \\
\Spec( \BC[\epsilon]/\epsilon^2 ) \ar{r}{a} & \Hilb^2(\p^1) \,.
\end{tikzcd}
\]
where $\pi_C$ is the given map of the family \eqref{fam_diag} and $a$ is the induced map.
Let~$s$ be the section of $\CO(\Delta_{\p^1})$ with zero locus $\Delta_{\p^1}$,
and assume the pullback $\phi^{\ast} s$ is non-zero.

Let $\Omega_{\pi_2}$ be the sheaf of relative differentials of $\pi_2 := \pi^{[2]}$. The composition
\begin{equation} \phi^{\ast} \pi_2^{\ast} \pi_{2 \ast} \Omega_{\pi_2} \to \phi^{\ast} \Omega_{\pi_2} \overset{d}{\to} \Omega_{\pi_C} \label{304} \end{equation}
factors as
\begin{equation} \phi^{\ast} \pi_2^{\ast} \pi_{2 \ast} \Omega_{\pi_2} \to \pi_C^{\ast} \pi_{C \ast} \Omega_{\pi_C} \to \Omega_{\pi_C}. \label{304_11} \end{equation}
Since the second term in \eqref{304_11} is zero, the map \eqref{304} is zero. Hence, $d$ factors as
\begin{equation} \phi^{\ast} \Omega_{\pi_2} \to \phi^{\ast}( \Omega_{\pi_2} / \pi_2^{\ast} \pi_{2 \ast} \Omega_{\pi_2} ) \to \Omega_{\pi_C}. \label{214} \end{equation}
By Lemma \ref{213} below, $\Omega_{\pi_2} / \pi_2^{\ast} \pi_{2 \ast} \Omega_{\pi_2}$ is the pushforward of a sheaf supported on $\pi_2^{-1}(\Delta_{\p^1})$.
After trivializing $\CO(\Delta_{\p^1})$ near $2 u_j$, write $\phi^{\ast} s = \lambda \epsilon$ for some $\lambda \in \BC \setminus \{ 0\}$. Then, by \eqref{214},
\[ 0 = d( s \cdot \Omega_{\pi_2} ) = \lambda \epsilon \cdot d( \Omega_{\pi_2} ) \subset \Omega_{\pi_{C}}. \]
In particular, $d b = 0$ for every $b \in \Omega_{\pi_2}$,
which does not vanish on $\pi_2^{-1}(\Delta_{\p^1})$.
Since $\phi|_{C}$ is non-zero, this is a contradiction.
\end{proof}

\begin{lemma} \label{213}
Let $x \in \p^1$ be the basepoint of a smooth fiber of $\pi : S \to \p^1$.
Then, there exists a Zariski-open $2x \in U \subset \Hilb^2(\p^1)$
and a map
\begin{equation} u: \CO_U^{\oplus 2} \to \pi_{2 \ast} \Omega_{\pi_2}|_{U} \label{215} \end{equation}
with cokernel equal to $j_{\ast} \CF$ for a sheaf $\CF$ on $\Delta_{\p^1} \cap U$.
\end{lemma}
\begin{proof}
Let $U$ be an open subset of $x \in \p^1$
such that $\pi_{\ast} \Omega_{\pi}|_U$ is trivialized by a section
\[ \alpha \in \pi_{\ast} \Omega_{\pi}(U) = \Omega_{\pi}( S_U ), \]
where $S_U = \pi^{-1}(U)$.
Consider the open neighborhood $\widetilde{U} = \Hilb^2(U)$ of the point $2x \in \Hilb^2(\p^1)$.

Let $D_U \subset S_U \times S_U$ be the diagonal and consider the $\BZ_2$ quotient
\[ \Bl_{D_U} ( S_U \times S_U ) \overset{ /\BZ_2 }{\ra} \Hilb^2(S_U) = \pi_2^{-1}( \Hilb^2(U) ). \]
For $i \in \{ 1,2 \}$, let 
\[ q_i : \Bl_{D_U} ( S_U \times S_U ) \to S_U \]
be the composition of the blowdown map with the $i$-th projection.
Let $t$ be a coordinate on $U$ and let
\[ t_i = q_{i}^{\ast} t, \quad \alpha_i = q_{i}^{\ast} \alpha \]
for $i = 1,2$ be the induced global functions resp. 1-forms on $\Bl_{D_U} ( S_U \times S_U )$.
The two 1-forms
\[ \alpha_1 + \alpha_2 \quad \text{ and } \quad (t_1 - t_2) (\alpha_1 - \alpha_2) \]
are $\BZ_2$ invariant and descend to global sections of $\pi_{2 \ast} \Omega_{\pi_2}|U$.
Consider the induced map
\begin{equation*} u: \CO_U^{\oplus 2} \ra \pi_{2 \ast} \Omega_{\pi_2}|_{U} \end{equation*}

The map $u$ is an isomorphism away from the diagonal
\begin{equation} \Delta_{\pi} \cap \widetilde{U} = V( (t_1 - t_2)^2 ) \subset \widetilde{U} \,. \label{216} \end{equation}
Hence, it is left to check the statement of the lemma in an infinitesimal neighborhood of \eqref{216}.
Let $U'$ be a small analytic neighborhood of $v \in U$ such that the restriction $\pi_{U'}: S_{U'} \to U'$ is analytically isomorphic to the quotient
\[ (U' \times \BC) \mathclose{}/\mathopen{} \sim \ra U', \]
where $\sim$ is the equivalence relation
\[ (t, z) \sim (t', z') \quad \Longleftrightarrow \quad t = t' \text{ and } z - z' \in \Lambda_{t}, \]
with an analytically varying lattice $\Lambda_t : \BZ^2 \to \BC$.
Now, a direct and explicit verification yields the statement of the lemma.
\end{proof}

\subsection{Case $\langle \Fp_{-1}(F)^2 1_S, I(P) \rangle_q$} \label{Section_CaseI_P}
Let $F^{\text{GW}}(y,q)$ and $H^{\text{GW}}(y,q)$ be the power series defined in \eqref{FGW} and \eqref{hhh} respectively,
let $P \in S$ be a point and let $F$ be the class of a fiber of $\pi :S \to \p^1$.

\begin{lemma} \label{Lemma_in_special_case_4} We have
\[ \big\langle \, \Fp_{-1}(F)^2 1_S\, ,\, I(P)\, \big\rangle_q^{\Hilb^2(S)} = \frac{F^{\text{GW}}(y,q) \cdot H^{\text{GW}}(y,q)}{\Delta(q)} \]
\end{lemma}

\begin{proof}
Let $F_1, F_2$ be fibers of $\pi \colon S \to \p^1$ over generic points $x_1, x_2 \in \p^1$ respectively,
and let $P \in S$ be a generic point. Define the subschemes
\[ Z_1 = F_1[1] F_2[1] \quad \text{ and } \quad Z_2 = I(P) \,. \]
Consider the evaluation map
\[ \ev : \Mbar_{0,1}(\Hilb^2(S), \beta_h + kA ) \to \Hilb^2(S) \]
from the moduli space of stable maps with \emph{one} marked point,
let
\[ M_{Z_2} = \ev^{-1}(Z_2), \]
and let
\[ M_Z \subset M_{Z_2} \]
be the closed substack of $M_{Z_2}$ of maps which are incident to both $Z_1$ and $Z_2$.

Let $[f : C \to \Hilb^2(S), p_1] \in M_Z$
be an element, let $C_0$ be the distinguished component of $C$ on which $\pi^{[2]} \circ f$ is non-zero,
and let $L = \pi^{[2]}(f(C_0))$ be the image line.
Since $P \in S$ is generic, we have $2v \in L$ where $v = \pi(P)$. Hence,
$L$ is the line through $2v$ and $u_1 + u_2$, and has the diagonal points
\begin{equation} L \cap \Delta_{\Hilb^2(\p^1)} = \{ 2u, 2v \} \label{dmvofdvfv} \end{equation}
for some fixed $u \in \p^1 \setminus \{ v \}$.
By Lemma \ref{117}, the restriction $f|_{C_0}$ is therefore an isomorphism onto the embedded line $L \subset \Hilb^2(B_0)$.
Using arguments parallel to Section~\ref{Section_settheoretic_splitting},
the moduli space $M_Z$ is \emph{set-theoretically}
a product of the moduli space of maps to the nodal fibers,
the moduli space $M^{\textup{(F)}}_{u}(h',k')$ parametrizing maps over $2u$, 
and the moduli space $M^{\textup{(H)}}_{v}(h'',k'')$ parametrizing maps over $2v$. 

Under infinitesimal deformations of $[f : C \to \Hilb^2(S)]$ inside $M_Z$,
the line~$L$ remains incident to $x_1 + x_2$,
but may move to first order at the point $2 v$ (see Section \ref{special_case_3});
hence, it may move also at $2 u$ to first order.
In particular, the moduli space is scheme-theoretically \emph{not} a product of the above moduli spaces.
Nevertheless, by degeneration,
we will reduce to the case of a scheme-theoretic product.
For simplicity, we work on the component of $M_Z$
which parametrizes maps with no component mapping to the nodal fibers of $\pi$;
the general case follows by completely analog arguments with an extra $1/\Delta(q)$ factor appearing as contribution from the nodal fibers.

Let $N \subset M_{Z_2}$ be the \emph{open} locus of maps
$f : C \to \Hilb^2(S)$ in $M_{Z_2}$
with
\[ \pi^{[2]}(f(C)) \cap \Delta_{\Hilb^2(\p^1)} = 
\{ 2 t, 2 v \} \]
for some point $t \in \p \setminus \{ x_1, \dots, x_{24}, v \}$.
Under deformations of an element $[f] \in N$,
the intersection point $2t$ may move freely and independently of $v$.
Hence, we have a splitting \emph{isomorphism}
\begin{equation}
 \Psi : N \ra \bigsqcup_{\substack{h = h_1 + h_2 \\ k=k_1 + k_2 - 1}} M^{\textup{(F)}}(h_1, k_1) \times M^{\textup{(H)}}_{v}(h_2, k_2), \label{412}
\end{equation}
where
\begin{itemize}
 \item $M^{\textup{(F)}}(h, k)$ is the moduli space of $1$-pointed stable maps to $\Hilb^2(S)$
of genus $0$ and class $hF + kA$ such that the marked point is mapped to $s^{[2]}(2t)$ for some $t \in \p \setminus \{ x_1, \dots, x_{24}, v \}$,
 \item $M^{\textup{(H)}}_{v}(h,k)$ is the moduli space defined in Section~\ref{special_case_3}.
\end{itemize}
For every decomposition $h=h_1 + h_2$ and $k = k_1 + k_2 - 1$ separately,
let
\[ M^{\textup{(F)}}(h_1, k_1) \times M^{\textup{(H)}}_{v}(h_2, k_2) \ra \Delta_{\Hilb^2(\p^1)} \times \Spec\bigl( \widehat{\CO}_{\Delta_{\Hilb^2(\p^1)}, 2 v} \bigr) \]
be the product of the compositions of the first evaluation map with $\pi^{[2]}$ on each factor, let 
\begin{equation}
\iota \colon V \hookrightarrow
\Delta_{\Hilb^2(\p^1)} \times \Spec\big( \widehat{\CO}_{\Delta_{\Hilb^2(\p^1)}, 2 v} \big)
\label{REGMMM}
\end{equation}
be the subscheme parametrizing the intersection points
$L \cap \Delta_{\Hilb^2(\p^1)}$ of lines~$L$ which are incident to $x_1 + x_2$,
and consider the fiber product
\begin{equation} \label{diagram_fkmvkfmv}
\begin{tikzcd}
M_{Z,(h_1,h_2,k_1,k_2)} \ar{r} \ar{d} & M^{\textup{(F)}}(h_1, k_1) \times M^{\textup{(H)}}_{v}(h_2, k_2) \ar{d} \\
V \ar{r}{\iota} & \Delta_{\Hilb^2(\p^1)} \times \Spec\big( \widehat{\CO}_{\Delta_{\Hilb^2(\p^1)}, 2 v} \big)
\end{tikzcd}
\end{equation}
Then, by definition, the splitting isomorphism \eqref{412} restricts to an isomorphism
\[ \Psi : M_Z \to \bigsqcup_{\substack{h = h_1 + h_2 \\ k=k_1 + k_2 - 1}}M_{Z,(h_1,h_2,k_1,k_2)}. \]

Restricting the natural virtual class on $M_{Z_2}$ to the open locus,
we obtain a virtual class $[N]^{\text{vir}}$ of dimension $1$.
By the arguments of Section \ref{Section_analysis_of_virtual_class},
\begin{equation} \Psi_{\ast}[N]^{\text{vir}}
= \sum_{\substack{h = h_1 + h_2 \\ k=k_1 + k_2 - 1}} [ M^{\textup{(F)}}(h_1, k_1) ]^{\text{vir}} \times [M_v^{\textup{(H)}}(h_2,k_2)]^{\text{vir}}, \label{ABBAC} \end{equation}
where $[ M^{\textup{(F)}}(h_1, k_1) ]^{\text{vir}}$ is a $\Delta_{\Hilb^2(\p^1)}$-relative version
of the virtual class considered in Section~\ref{Section_analysis_of_virtual_class},
and $[M^{\textup{(H)}}_{v}(h_2, k_2)]^{\text{vir}}$ is the virtual class constructed in Section~\ref{special_case_3}.
The composition of $\iota$ with the projection to the second factor is an isomorphism.
Hence $\iota$ is a regular embedding and we obtain
\begin{multline} \label{qqqq}
\blangle \Fp_{-1}(F)^2, I(P) \brangle^{\Hilb^2(S)}_{\beta_h+kA}
= \deg( \Psi_{\ast} [M_Z]^{\text{vir}} ) \\
= \sum_{\substack{h = h_1 + h_2 \\ k=k_1 + k_2 - 1}} \deg\, \iota^{!} \Big([ M^{\textup{(F)}}(h_1, k_1) ]^{\text{vir}} \times [M_v^{\textup{(H)}}(h_2,k_2)]^{\text{vir}} \Big) .
\end{multline}

We proceed by degenerating the first factor in the product
\[ M^{\textup{(F)}}(h_1, k_1) \times M^{\textup{(H)}}_{v}(h_2, k_2), \]
while keeping the second factor fixed. Let
\[ \CS = \Bl_{F_{u} \times 0}(S \times \BA^1) \to \BA^1, \]
be a deformation of $S$ to the normal cone of $F_u$, where $u$ was defined in \eqref{dmvofdvfv}.
Let $\CS^{\circ} \subset \CS$ be the complement of the proper transform of $S \times 0$
and consider the \emph{relative} Hilbert scheme
$\Hilb^2( \CS^{\circ} / \BA^1 ) \to \BA^1$,
which ap\-pea\-red already in~\eqref{sofjosfosdf}.
Let
\begin{equation} p \colon \widetilde{M}^{\textup{(F)}}(h_1, k_1) \to \BA^1 \label{eigreigrg} \end{equation}
be the moduli space of $1$-pointed stable maps to $\Hilb^2(\CS^{\circ}/\BA^1)$ of genus $0$ and class $h_1 F + k_1 A$,
which map the marked point to the closure of 
\[ (\Delta_{\Hilb^2(B_0)} \setminus \{ x_1, \dots, x_{24}, v\}) \times (\BA^1 \setminus \{ 0 \}). \]
Over $t \neq 0$, \eqref{eigreigrg} restricts to $M^{\textup{(F)}}(h_1, k_1)$,
while the fiber over $0$, denoted
\[ M_0^{\textup{(F)}}(h_1, k_1) = p^{-1}(0), \]
parame\-tri\-zes maps into the trivial elliptic fibration $\Hilb^2(\BC \times E)$
incident to the diagonal $\Delta_{\Hilb^2(\BC \times e)}$ for a fixed $e \in E$.
Since addition by $\BC$ acts on $M_0^{\textup{(F)}}(h_1, k_1)$ we have the product decomposition
\begin{equation} M_0^{\textup{(F)}}(h_1, k_1) = M^{\textup{(F)}}_{0,\text{fix}}(h_1, k_1) \times \Delta_{\Hilb^2(\BC \times e)}, \label{dmvfmvvbcbv} \end{equation}
where $M^{\textup{(F)}}_{0,\text{fix}}(h_1, k_1)$ is a fixed fiber of
\[ M_0^{\textup{(F)}}(h_1, k_1) \to \Delta_{\Hilb^2(\BC \times e)}. \]

Consider a deformation of the diagram \eqref{diagram_fkmvkfmv} to $0 \in \BA^1$,
\[
\begin{tikzcd}
M'_{Z,(h_1,h_2,k_1,k_2)} \ar{r} \ar{d} & M_0^{\textup{(F)}}(h_1, k_1) \times M^{\textup{(H)}}_{v}(h_2, k_2) \ar{d} \\
V' \ar{r}{\iota'} & \Delta_{\Hilb^2(\BC \times E)} \times \Spec\bigl( \widehat{\CO}_{\Delta_{\Hilb^2(\p^1)}, 2 v} \bigr),
\end{tikzcd}
\]
where $(V',\iota')$ is the fiber over $0$ of a deformation of $(V,\iota)$ 
such that the composition with the projection to
$\Spec( \widehat{\CO}_{\Delta_{\Hilb^2(\p^1)}, 2 v} )$ remains an isomorphism.
By construction, the total space of the deformation
\[ M_{Z,(h_1,h_2,k_1,k_2)} \rightsquigarrow M'_{Z,(h_1,h_2,k_1,k_2)} \]
is proper over $\BA^1$. Using the product decomposition \eqref{dmvfmvvbcbv}, we find
\[ M'_{Z,(h_1,h_2,k_1,k_2)} \cong M^{\textup{(F)}}_{0,\text{fix}}(h_1, k_1) \times M^{\textup{(H)}}_{v}(h_2, k_2) \]
Hence, after degeneration, we are reduced to a scheme-theoretic product. It remains
to consider the virtual class.

By the relative construction of Section~\ref{Section_analysis_of_virtual_class}
the moduli space $\widetilde{M}^{\textup{(F)}}(h_1, k_1) $ carries a virtual class
\begin{equation} [ \widetilde{M}^{\textup{(F)}}(h_1, k_1) ]^{\text{vir}} \label{SOSOSO123} \end{equation}
which restricts to
$[ M^{\textup{(F)}}(h_1, k_1) ]^{\text{vir}}$ over $t \neq 0$,
while over $t=0$ we have
\begin{equation}
0^{!}[ \widetilde{M}^{\textup{(F)}}(h_1, k_1) ]^{\text{vir}} = {\rm pr}_1^{\ast} \left( [ M^{\text{fix}}_1(h_1,k_1) ]^{\text{vir}} \right)\,.
 \label{kgfjreufr}
\end{equation}
where ${\rm pr}_1$ is the projection to the first factor in \eqref{dmvfmvvbcbv}
and $[ M^{\text{fix}}_1(h_1,k_1) ]^{\text{vir}}$ is the virtual class obtained by the construction of Section~\ref{Section_analysis_of_virtual_class}.
We conclude,
\begin{multline} \label{rrrr}
\deg\, \iota^{!} \Big([ M^{\textup{(F)}}(h_1, k_1) ]^{\text{vir}} \times [M_v^{\textup{(H)}}(h_2,k_2)]^{\text{vir}} \Big)  \\
=
\deg \, (\iota')^{!} \Big({\rm pr}_1^{\ast} \big( [ M^{\text{fix}}_1(h_1,k_1) ]^{\text{vir}} \big) \times [M_v^{\textup{(H)}}(h_2,k_2)]^{\text{vir}} \Big) \\
=
\deg \big( [ M^{\text{fix}}_1(h_1,k_1) ]^{\text{vir}} \big) \cdot \deg \big( [M_v^{\textup{(H)}}(h_2,k_2)]^{\text{vir}} \big).
\end{multline}

By definition (see \eqref{hhh}),
\[ \deg [M_v^{\textup{(H)}}(h_2,k_2)]^{\text{vir}} = \left[ H^{\textup{GW}}(y,q) \right]_{q^{h_2} y^{k_2-1/2}}, \]
where $[\ \cdot\ ]_{q^{a} y^b}$ denotes the $q^a y^b$ coefficient.
The moduli space $M^{\text{fix}}_1(h_1,k_1)$
is isomorphic to the space $M^{(\ell)}_{y_i}(h_1,k_1)$ defined in \eqref{moduli_space_Kummer_case}.
Since the construction of the virtual class on both sides agree,
the virtual class is the same under this isomorphism. Hence, by Lemma~\ref{242_lemma},
\[ \deg [ M^{\text{fix}}_1(h_1,k_1) ]^{\text{vir}} = \left[ F^{\textup{GW}}(y,q) \right]_{q^{h_1} y^{k_1-1/2}}. \]

Inserting into \eqref{rrrr} yields
\begin{multline*}
\deg\, \iota^{!} \Big([ M^{\textup{(F)}}(h_1, k_1) ]^{\text{vir}} \times [M_v^{\textup{(H)}}(h_2,k_2)]^{\text{vir}} \Big) \\
= \left[ H^{\textup{GW}}(y,q) \right]_{q^{h_2} y^{k_2-1/2}} \cdot \left[ F^{\textup{GW}}(y,q) \right]_{q^{h_1} y^{k_1-1/2}},
\end{multline*}
which completes the proof by equation \eqref{qqqq}.
\end{proof}

\section{The Hilbert scheme of $2$ points of $\p^1 \times E$} \label{Section_Hilb2P1xE}
\subsection{Overview}
In previous sections we expressed genus $0$ Gromov-Witten invariants
of the Hilbert scheme of points of an elliptic K3 surface $S$ in terms of
universal series which depend only on specific fibers of the fibration $S \to \p^1$.
The contributions from nodal fibers have been determined
before by Bryan and Leung in their proof \cite{BL} of the Yau-Zaslow formula \eqref{YZ}.
The yet undetermined contributions from smooth fibers, denoted
\begin{equation} \label{undetermined_series} F^{\textup{GW}}(y,q),\ G^{\textup{GW}}(y,q),\ \widetilde{G}^{\textup{GW}}(y,q),\ H^{\textup{GW}}(y,q) \end{equation}
in equations \eqref{FGW}, \eqref{GGGdef}, \eqref{gggt}, \eqref{hhh} respectively,
depend only on infinitesimal data near the smooth fibers,
and not on the global geometry of the K3 surface.
Hence, one may hope to find similar contributions
in the Gromov-Witten theory of the Hilbert scheme of points
of other elliptic fibrations. 

Let $E$ be an elliptic curve with origin $0_E \in E$, and let 
\[ X = \p^1 \times E \]
be the trivial elliptic fibration.
Here, we study the genus $0$ Gromov-Witten theory of the Hilbert scheme
\[ \Hilb^2(X) \,. \]
and use our results to determine the series \eqref{undetermined_series}.

From the view of Gromov-Witten theory, the variety $\Hilb^2(X)$ has two advantages
over the Hilbert scheme of $2$ points of an elliptic K3 surface.
First, $\Hilb^2(X)$ is not holomorphic symplectic. Therefore,
we may use ordinary Gromov-Witten invariants
and in particular the main computation method which exists in genus $0$ Gromov-Witten theory -- the WDVV equation.
Second, we have an additional map
\[ \Hilb^2(X) \to \Hilb^2(E) \]
induced by the projection of $X$ to the second factor which is useful in calculations.

Our study of the Gromov-Witten theory of $\Hilb^2(X)$ will proceed in two independent directions.
First,
we directly analyse the moduli space of stable maps to $\Hilb^2(X)$ which are incident to certain geometric cycles.
Similar to the K3 case, this leads to an explicit expression of generating series of Gromov-Witten invariants of $\Hilb^2(X)$
in terms of the series \eqref{undetermined_series}.
This is parallel to the study of
the Gromov-Witten theory of the Hilbert scheme of points of a K3 surface
in Sections \ref{Section_higher_dimensional_Yau_Zaslow} and \ref{Section_More_Evaluations}.

In a second independent step, we will calculate the Gromov-Witten invariants of $\Hilb^2(X)$
using the WDVV equations and a few explicit calculations of initial data.
Then, combining both directions, we are able to solve for the functions \eqref{undetermined_series}.
This leads to the following result.

Let $F(z,\tau)$ be the Jacobi theta function \eqref{FFFdef} and, with $y = -e^{2\pi i z}$, let
\[ G(z,\tau) = F(z,\tau)^2 \left( y \frac{d}{dy} \right)^2 \log(F(z,\tau)) \]
be the function which appeared already in Section \ref{Section_More_evaluations_Introduction}. 

\begin{thm} \label{Hilb2P1e_complete_evaluation_theorem} Under the variable change $y=-e^{2\pi iz}$ and $q = e^{2 \pi i \tau}$,
\begin{align*}
F^{\textup{GW}}(y,q) & = F(z,\tau) \\
G^{\textup{GW}}(y,q) & = G(z,\tau) \\
\widetilde{G}^{\textup{GW}}(y,q) & = - \frac{1}{2} \left( y \frac{d}{dq} \right) G(z,\tau) \\
H^{\textup{GW}}(y,q) & = \left( q \frac{d}{dq} \right) F(y,q)
\end{align*}
\end{thm}

The proof of Theorem \ref{Hilb2P1e_complete_evaluation_theorem} via the geometry of $\Hilb^2(X)$ is independent
from the Kummer K3 geometry studied in Section~\ref{section_Kummer_evaluation}.
In particular, our approach here yields a second proof of Theorem \ref{thm_F_evaluation}.

\subsection{The fiber of $\Hilb^2(\p^1 \times E) \to E$}
\subsubsection{Definition}
The projections of $X = \p^1 \times E$ to the first and second factor
induce the maps
\begin{equation} \pi : \Hilb^2(X) \to \Hilb^2(\p^1) = \p^2 \quad \text{ and } \quad \tau : \Hilb^2(X) \to \Hilb^2(E) \label{induced_morphisms} \end{equation}
respectively. Consider the composition
\[ \sigma : \Hilb^2(X) \xrightarrow{\, \tau \, } \Hilb^2(E) \xrightarrow{\, + \, } E \]
of $\tau$ with the addition map $+ \colon \Hilb^2(E) \to E$.
Since $\sigma$ is equivariant with respect to the natural action of $E$ on $\Hilb^2(X)$ by translation,
it is an isotrivial fibration with smooth fibers. We let
\[ Y = \sigma^{-1}(0_E) \]
be the fiber of $\sigma$ over the origin $0_E \in E$.

Let $\gamma \in H_2(\Hilb^2(X))$ be an effective curve class and let
\[ \Mbar_{0,m}(\Hilb^2(X), \gamma) \]
be the moduli space of $m$-pointed stable maps to $\Hilb^2(X)$ of genus $0$ and class~$\gamma$.
The map $\sigma$ induces an isotrivial fibration
\[ \sigma : \Mbar_{0,m}(\Hilb^2(X), \gamma) \to E \]
with fiber over $0_E$ equal to
\[ \bigsqcup_{\gamma'}\, \Mbar_{0,m}(Y, \gamma'), \]
where the disjoint union runs over all effective curve classes
$\gamma' \in H_2(Y;\BZ)$ with $\iota_{\ast} \gamma' = \gamma$; here $\iota : Y \to \Hilb^2(X)$ is the inclusion.

For cohomology classes $\gamma_1, \dots, \gamma_m \in H^{\ast}(\Hilb^2(X))$, we have
\begin{multline*}
 \int_{[ \Mbar_{0,m}(\Hilb^2(X), \gamma) ]^{\text{vir}}} \ev_1^{\ast} ( \gamma_1 \cup [Y] )\, \cdots \, \ev_{m}^{\ast}( \gamma_m ) \\
= \sum_{\substack{ \gamma' \in H_2(Y) \\ \iota_{\ast} \gamma' = \gamma} }
\int_{ [ \Mbar_{0,m}(Y, \gamma') ]^{\text{vir}}} (\iota \circ \ev_1)^{\ast}( \gamma_1 ) \cdots (\iota \circ \ev_{m})^{\ast}( \gamma_m ),
\end{multline*}
where we let $[\ \cdot \ ]^{\text{vir}}$ denote the virtual class defined by ordinary Gromov-Witten theory.
Hence, for calculations related to the Gromov-Witten theory of $\Hilb^2(X)$
we may restrict to the threefold $Y$.

\subsubsection{Cohomology} \label{cohomology_of_Y} \label{Section_cohomology_of_Y}
Let $D_{X} \subset X \times X$ be the diagonal and let
\begin{equation} \Bl_{D_{X}}(X \times X) \to \Hilb^2(X) \label{fmdefg} \end{equation}
be the $\BZ_2$-quotient map which interchanges the factors. Let
\[ W = \p^1 \times \p^1 \times E\ \hookrightarrow\ X \times X, \quad (x_1, x_2,e) \mapsto (x_1, e, x_2,-e) \]
be the fiber of $0_E$ under $X \times X \to E \times E \overset{+}{\to} E$ and consider the blowup
\begin{equation} \rho : \widetilde{W} = \Bl_{D_{X} \cap W} W \to W. \label{blowup_map} \end{equation}
Then, the restriction of \eqref{fmdefg} to $\widetilde{W}$ yields the $\BZ_2$-quotient map
\begin{equation} g : \widetilde{W} \to \widetilde{W} / \BZ_2 = Y. \label{gquotientmap} \end{equation}

Let $D_{X,1}, \dots, D_{X,4}$ be the components of the intersection
\[ D_{X} \cap W = \{ (x_1, x_2, f) \in \p^1 \times \p^1 \times E\ | \ x_1 = x_2 \text{ and } f = -f \} \]
corresponding to the four $2$-torsion points of $E$, and let
\[ E_1, \dots, E_4 \]
be the corresponding exceptional divisors of the blowup $\rho : \widetilde{W} \to W$.
For every $i$, the restriction of $g$ to $E_i$ is an isomorphism onto its image. Define
the cohomology classes
\[ \Delta_i = g_{\ast} [ E_i ], \quad \quad A_i = g_{\ast} [ \rho^{-1}(y_i) ] \]
for some $y_i \in D_{X,i}$.
We also set 
\[ \Delta = \Delta_1 + \dots + \Delta_4 \,, \quad \quad A = \frac{1}{4} ( A_1 + \dots + A_4 ). \]
Let $x_1, x_2 \in \p^1$ and $f \in E$ be points, and define
\[ B_1 = g_{\ast} \big[ \rho^{-1}( \p^1 \times x_2 \times f ) \big],
\quad \quad 
 B_2 = \frac{1}{2} \cdot g_{\ast} \big[ \rho^{-1}(x_1 \times x_2 \times E) \big] \,.
\]
Identify the fiber of $\Hilb^2(E) \to E$ over $0_E$ with $\p^1$, and consider the diagram
\begin{equation} \label{some_mmm_diagram}
\begin{tikzcd}
Y \ar{r}{\tau} \ar{d}{\pi} & \p^1 \\ \p^2
\end{tikzcd}
\end{equation}
induced by the morphisms \eqref{induced_morphisms}.
Let $h \in H^2(\p^2)$ be the class of a line and let $x \in \p^1$ be a point. Define the divisor classes
\[ D_1 = [ \tau^{-1}(x) ], \quad \quad D_2 = \pi^{\ast} h \,. \]

\begin{lemma} The cohomology classes
\begin{equation} D_1, D_2, \Delta_1, \dots, \Delta_4 \quad \quad (\text{resp. } B_1, B_2, A_1, \dots, A_4 \big) \label{305} \end{equation}
form a basis of $H^2(Y;\BQ)$\, (resp. of $H^4(Y;\BQ))$.
\end{lemma}
\begin{proof} Since the map $g$ is the quotient map by the finite group $\BZ_2$, we have the isomorphism
\[ g^{\ast} \colon H^{\ast}(Y;\BQ) \to H^{\ast}(\widetilde{W};\BQ)^{\BZ_2}, \]
where the right hand side denotes the $\BZ_2$ invariant part of the cohomology of~$\widetilde{W}$.
The Lemma now follows from a direct verification.
\end{proof}
By straight-forward calculation, we find the following intersections between the basis elements \eqref{305}.
\begin{center}
\begin{tabular}{lcl}
\begin{tabular}{c | c c c}
\cline{1-4}
$\cdot$ & $B_1$ & $B_2$ & $A_i$ \\
\cline{1-4}
$D_1$           & $0$ & $1$ & $0$ \\
$D_2$           & $1$ & $0$ & $0$ \\
$\Delta_j$      & $0$ & $0$ & $-2 \delta_{ij}$ \\
\end{tabular}
& $\quad \quad$ &
\begin{tabular}{c | c c c }
\cline{1-4}
$\cdot$     & $D_1$ & $D_2$ & $\Delta_i$ \\
\cline{1-4}
$D_1$           & $0$     & $2 B_1$ & $0$ \\
$D_2$           & $2 B_1$ & $2 B_2$ & $2 A_i$ \\
$\Delta_j$      & $0$     & $2 A_j$ & $4 (A_i - B_1) \delta_{ij}$ \\[5pt]
\end{tabular}
\end{tabular}
\end{center}
Finally, using intersection against test curves, the canonical class of $Y$ is
\[ K_Y = -2 D_2 \,. \]

\subsubsection{Gromov-Witten invariants}
Let $r,d \geq 0$ be integers and let $\textbf{k} = (k_1, \dots, k_4)$ be a tuple of half-integers $k_i \in \frac{1}{2} \BZ$. Define the class
\[ \beta_{r,d,\textbf{k}} = r B_1 + d B_2 + k_1 A_1 + k_2 A_2 + k_3 A_3 + k_4 A_4. \]
Every algebraic curve in $Y$ has a class of this form.

For cohomology classes $\gamma_1, \dots, \gamma_m \in H^{\ast}(Y;\BQ)$ define
the genus $0$ potential
\begin{equation}
\big\langle \gamma_1, \dots, \gamma_l \big\rangle^Y = \sum_{r,d \geq 0} \sum_{\textbf{k} \in (\frac{1}{2} \BZ)^4 } \zeta^r q^d y^{\sum_i k_i}
\int_{[ \Mbar_{0,m}(Y, \beta_{r,d,\textbf{k}}) ]^{\text{vir}} } \ev_{1}^{\ast}(\gamma_1) \cdots \ev_m^{\ast}(\gamma_m),
\label{306} \end{equation}
where $\zeta, y, q$ are formal variables and the integral on the right hand side
is defined to be $0$ whenever $\beta_{r,d, \textbf{k}}$ is not effective.

The virtual class of $\Mbar_{0,m}(Y,\beta_{r,d,\textbf{k}})$
has dimension $2r+m$.
Hence, for homogeneous classes $\gamma_1, \dots, \gamma_m$
of complex degree $d_1, \dots, d_m$ respectively satisfying $\sum_i d_i = 2r + m$,
only terms with $\zeta^r$ contribute to the sum \eqref{306}. 
In this case, we often set $\zeta = 1$. 

\subsubsection{WDVV equations} 
Let $\iota \colon Y \to \Hilb^2(X)$ denote the inclusion and consider the subspace
\begin{equation} i^{\ast} H^{\ast}(\Hilb^2(X) ;\BQ) \subset H^{\ast}(Y;\BQ). \label{subspace_abc} \end{equation}
of classes pulled back from $\Hilb^2(X)$. The tuple of classes
\[ b = (T_i)_{i=1}^{8} = (e_Y, D_1, D_2, \Delta, B_1, B_2, A, \pt_Y), \]
forms
a basis of \eqref{subspace_abc};
here $e_Y = [Y]$ is the fundamental class and $\pt_Y$ is the class of point of $Y$.
Let $(g_{ef})_{e,f}$ with 
\[ g_{ef} = \langle T_e, T_f \rangle = \int_Y T_e \cup T_f \]
be the intersection matrix of $b$,
and let $( g^{ef} )_{e,f}$ be its inverse.

\begin{lemma} Let $\gamma_1, \dots, \gamma_4 \in i^{\ast} H^{\ast}(\Hilb^2(X) ;\BQ)$
be homogeneous classes of complex degree $d_1, \dots, d_4$ respectively such that $\sum_i d_i = 5$.
Then,
\begin{equation} \label{313}
 \sum_{e,f=1}^{8} \big\langle \gamma_1, \gamma_2, T_e \big\rangle^Y g^{ef} \big\langle \gamma_3, \gamma_4, T_f \big\rangle^Y
=
\sum_{e,f=1}^{8} \big\langle \gamma_1, \gamma_4, T_e \big\rangle^Y g^{ef} \big\langle \gamma_2, \gamma_3, T_f \big\rangle^Y.
\end{equation}
\end{lemma}
\begin{proof}
The claim follows directly from the classical WDVV equation \cite{FP} and direct formal manipulations.
\end{proof}

We reformulate equation \eqref{313} into the form we will use.
Let 
\[ \gamma \in i^{\ast} H^2(\Hilb^2(X);\BQ) \]
be a divisor class and let
\[ Q(\zeta, y, q) = \sum_{i,d,k} a_{ikd} \zeta^i y^k q^d \]
be a formal power series.
Define the differential operator $\partial_{\gamma}$ by
\[ \partial_{\gamma} Q(\zeta, y, q) = \sum_{i,d,k} \Big( \int_{i B_1 + d B_2 + k A} \gamma \Big) a_{ikd} \zeta^i y^k q^d. \]
Explicitly, we have
\[ \partial_{D_1} = q \frac{d}{d q}, \quad \quad \partial_{D_2} = \zeta \frac{d}{d \zeta}, \quad \quad \partial_{\Delta} = -2 y \frac{d}{d y}. \]
Then, for homogeneous classes $\gamma_1, \dots, \gamma_4 \in i^{\ast} H^{\ast}(\Hilb^2(X) ;\BQ)$ of complex degree $2,1,1,1$ respectively,
the left hand side of \eqref{313} equals
\begin{multline}
 \partial_{\gamma_2} \big\langle \gamma_1, \gamma_3 \cup \gamma_4 \big\rangle^Y
+ \partial_{\gamma_4} \partial_{\gamma_3} \big\langle \gamma_1 \cup \gamma_2 \big\rangle^Y \\
+
\sum_{\substack{ T_e \in \{ B_1, B_2, A \} \\ T_f \in \{ D_1, D_2, \Delta \} }}
\partial_{\gamma_2}\left( \big\langle \gamma_1, T_e \big\rangle^Y \right) g^{ef}
\partial_{\gamma_3} \partial_{\gamma_4} \partial_{T_f} \big\langle 1 \big\rangle^Y, \label{315}
\end{multline}
where we let $\big\langle 1 \big\rangle$ denote the Gromov-Witten potential with no insertions.
The expression for the right hand side of \eqref{313} is similar.

\subsubsection{Relation to the Gromov-Witten theory of $\Hilb^2(K3)$}
Recall the power series \eqref{undetermined_series},
\[ F^{\textup{GW}}(y,q),\ G^{\textup{GW}}(y,q),\ \widetilde{G}^{\textup{GW}}(y,q),\ H^{\textup{GW}}(y,q). \]
\begin{prop} \label{comparision_proposition} There exist a power series 
\[ \widetilde{H}^{\textup{GW}}(y,q) \in \BQ((y^{1/2}))[[q]] \]
such that
\begin{align}
\big\langle B_2, B_2 \big\rangle^Y & = (F^{\textup{GW}})^2 \tag{i} \\
\big\langle \pt_{Y} \big\rangle^Y & = 2 G^{\textup{GW}} \tag{ii} \\
\big\langle B_1, B_2 \big\rangle^Y & = 2 F^{\textup{GW}} \cdot H^{\textup{GW}} + G^{\textup{GW}} \tag{iii} \\
\big\langle A, B_1 \big\rangle^Y & = \widetilde{G}^{\textup{GW}} + \widetilde{H}^{\textup{GW}} \cdot H^{\textup{GW}} \tag{iv} \\
\big\langle A, B_2 \big\rangle^Y & = \frac{1}{2} \widetilde{H}^{\textup{GW}} \cdot F^{\textup{GW}} \,. \tag{v}
\end{align}
\end{prop}

\begin{proof}
Let $d \geq 0$ be an integer, let $\textbf{k} = (k_1, \dots, k_4) \in (\frac{1}{2} \BZ)^4$ be a tuple of half-integers, and let
\[ \beta_{d, \textbf{k}} = B_1 + d B_2 + k_1 A_1 + \dots + k_4 A_4. \]
Consider a stable map $f : C \to Y$ of genus $0$ and class $\beta_{d, \textbf{k}}$.
The composition
$\pi \circ f : C \to \p^2$ has degree $1$ with image a line $L$.
Let $C_0$ be the component of $C$ on which $\pi \circ f$ is non-constant.

Let $g : \widetilde{W} \to Y$ be the quotient map \eqref{gquotientmap}, and consider the fiber diagram
\[
\begin{tikzcd}
\widetilde{C} \ar{r}{\widetilde{f}} \ar{d}    & \widetilde{W} \ar{d}{g} \ar{r}{p} & \p^1 \times E \\
C \ar{r}{f} & Y,
\end{tikzcd}
\]
where $p = \pr_{23} \circ \rho$ is the composition of the blowdown
map with the projection to the $(2,3)$-factor of $\p^1 \times \p^1 \times E$.
Then, parallel to the case of elliptic K3 surfaces,
the image of $\widetilde{C}$ under $p \circ \widetilde{f}$
is a comb curve
\[ B_e + \pr_1^{-1}(z), \]
where $B_e$ is the fiber of the projection $X \to E$ over some point $e \in E$,
the map $\pr_1 : \p^1 \times E \to \p^1$ is the projection to the first factor,
and $z \subset \p^1$ is a zero-dimensional subscheme of length $d$.

Let $G_0 \subset \widetilde{C}$
be the irreducible component which maps with degree $1$ to~$B_e$
under $p \circ \widetilde{f}$.
The projection $\widetilde{C} \to C$ induces a flat map
\begin{equation} G_0 \to C_0 \,. \label{flll_map} \end{equation}

If \eqref{flll_map} has degree $2$, then similar to the arguments of Lemma \ref{117},
the restriction $f|_{C_0}$ is an isomorphism onto an embedded line
\begin{equation*} L \subset \Hilb^2(S_e) \subset Y, \end{equation*}
where $e = -e \in E$ is a $2$-torsion point of $E$.
Since $f|_{C_0}$ is irreducible, we have $L \nsubseteq I(x)$ for all $x \in \p^1$.
The tangent line to $\Delta_{\Hilb^2(\p^1)}$ at $2x$ is $I(x)$ for every $x \in \p^1$.
Hence, $L$ intersects the diagonal $\Delta_{\Hilb^2(\p^1)}$ in two distinct points.

If \eqref{flll_map} has degree $1$,
the map $f|_{C_0}$ is the sum
of two maps $C_0 \to X$.
The first of these must map $C_0$ to a section of $X \to \p^1$, the second must be constant
since there are no non-constant maps to the fiber of $X \to \p^1$.
Hence, the restriction $f|_{C_0}$ is an isomorphism onto the embedded line\footnote{
If $e$ is a $2$-torsion point of $E$,
we take the proper transform instead of $\rho^{-1}$ in \eqref{308}.
This case will not appear below.}
\begin{equation} B_e + (x',-e) = g\left( \rho^{-1}\left( x' \times B_e \times -e \right) \right), \label{308} \end{equation}
for some $x' \in \p^1$ and $e \in E$; here we used the notation \eqref{116}.

Every irreducible component of $C$ other then $C_0$ maps into the fiber
of
\[ \pi : Y \to \Hilb^2(\p^1) = \p^2 \]
over a diagonal point $2x \in \Delta_{\Hilb^2(\p^1)}$.

Summarizing, the map $f : C \to Y$ therefore satisfies one of the following.
\begin{enumerate}
 \item[(A)]
The restriction $f|_{C_0}$
is an isomorphism onto a line 
\begin{equation} L \subset \Hilb^2(B_e) \subset Y  \label{312} \end{equation}
where $e \in E$ is a $2$-torsion point.
The line $L$ intersects the diagonal in the distinct points $2 x_1$ and $2 x_2$.
The curve $C$ has a decomposition
\begin{equation} C = C_0 \cup C_1 \cup C_2 \label{Csplitting_eqn2} \end{equation}
such that for $i=1,2$ the restriction $f|_{C_i}$ maps in the fiber $\pi^{-1}(2 x_i)$.
 \item[(B)] The restriction $f|_{C_0}$ is an isomorphism onto the line \eqref{308} for some $x' \in \p^1$ and $e \in E$.
The image $f(C_0)$ meets the fiber $\pi^{-1}(\Delta_{\Hilb^2(\p^1)})$
only in the point $(x',e) + (x',-e)$.
Hence, the curve $C$ admits the decomposition
\begin{equation} C = C_0 \cup C_1 \label{Csplitting_eqn} \end{equation}
where $f|_{C_1}$ maps to the fiber $\pi^{-1}(2 x')$.
\end{enumerate}
According to the above cases, we say that $f : C \to Y$ is of type (A) or (B).

We consider the different cases of Proposition \ref{comparision_proposition}.

\vspace{7pt}
\noindent \textbf{Case (i).} Let $Z_1, Z_2$ be generic fibers of the natural map
\[ \pi : Y \to \Hilb^2(\p^1). \]
The fibers $Z_1, Z_2$ have class $2 B_2$.
Let $f : C \to Y$ be a stable map of class~$\beta_{d, \textbf{k}}$ incident to $Z_1$ and $Z_2$.
Then $f$ must be of type (A) above, with the line $L$ in \eqref{312} uniquely determined by $Z_1, Z_2$
up the choice of the 2-torsion point $e \in E$. 
After specifying a $2$-torsion point,
we are in a case completely parallel to Section \ref{Section_higher_dimensional_Yau_Zaslow},
except for the existence of the nodal fibers in the K3 case.
Following the argument there, we find the contribution from each fixed $2$-torsion point to be $(F^{\textup{GW}})^2$.
Hence,
\[ \big\langle 2 B_2, 2 B_2 \big\rangle^Y = \left| \left\{ e \in E\ |\ 2e = 0 \right\} \right| \cdot  (F^{\textup{GW}})^2 = 4 \cdot (F^{\textup{GW}})^2. \]

\vspace{7pt}
\noindent \textbf{Case (ii).} Let $x_1, x_2 \in \p^1$ and $e \in E$ be generic, and consider the point
\[ y = (x_1,e) + (x_2,-e) \in Y\,. \]
A stable map $f : C \to Y$ of class $\beta_{d,\textbf{k}}$ incident to $y$ must be of type (B) above,
with $x' = x_1$ or $x_2$ in \eqref{308}.
In each case, the calculation proceeds completely analogous to Section \ref{CASE1} and yields the contribution $G^{\textup{GW}}$.
Summing up both cases, we therefore find $\langle y \rangle^Y = 2 G^{\textup{GW}}$.

\vspace{7pt}
\noindent \textbf{Case (iii).} Let $x', x_1, x_2 \in \p^1$ and $e \in E$ be generic points. Let
\begin{equation} \label{owrogr}
 Z_1 = g(\rho^{-1}( \p^1 \times x' \times e )) = (\p^1 \times e) + (x', -e)
\end{equation}
and let $Z_2$ be the fiber of $\pi$ over the point $x_1 + x_2$.
The cycles $Z_1, Z_2$ have the cohomology classes $[Z_1] = B_1$ and $[Z_2] = 2 B_2$ respectively.
Let
\[ f \colon C \to Y \]
be a $2$-marked stable map of genus $0$ and class $\beta_{d, \textbf{k}}$ with markings $p_1, p_2 \in C$
incident to $Z_1, Z_2$ respectively.
Since $f(p_1) \in Z_1$, we have
\[ f(p_1) = (x'', e) + (x',-e) \]
for some $x'' \in \p^1$. Since also $f(p_2) \in Z_2$ and $e$ is generic, $x'' \in \{ x', x_1, x_2 \}$.

Assume $x'' = x_1$. Then, $f$ is of type (B) and the restriction $f|_{C_0}$ is an isomorphism onto the line $\ell = B_e + (x_1,-e)$.
The line $\ell$ meets the cycle $Z_2$ in the point $(x_2,e) + (x_1,-e)$
and no marked point of $C$ lies on the component~$C_1$ in the splitting \eqref{Csplitting_eqn}.
Parallel to (ii), the contribution of this case is $G^{\textup{GW}}$. The case $x'' = x_2$ is identical.

Assume $x'' = x'$. Then, $\pi(f(p_1)) = 2x'$.
Since $\pi(f(p_2)) = x_1 + x_2$, we have $\pi(f(p_1)) \cap \pi(f(p_2)) = \varnothing$. Hence, $f$ is of type (A) and we have
the decomposition
\[ C = C_0 \cup C_1 \cup C_2, \]
where $f|_{C_0}$ maps to a line $L \subset \Hilb^2(B_{e'})$ for a $2$-torsion point $e' \in E$,
the restriction $f|_{C_1}$ maps to $\pi^{-1}(2 x')$,
and $f|_{C_2}$ maps to the fiber of $\pi$ over the diagonal point of $L$ which is not $2x'$.
We have $p_1 \in C_1$ with $f(p_1) \in Z_1$, and $p_2 \in C_0$ with $f(p_2) = (x_1,e') + (x_2, -e')$.
The contribution from maps to the fiber over $2x'$ matches the contribution $H^{\textup{GW}}$ considered in Section \ref{Section_CaseI_P}.
Since there is no marking on $C_2$, the contribution from maps $f|_{C_2}$ is $F^{\textup{GW}}$.
For each fixed $2$-torsion point $e' \in E$, 
we therefore find the contribution $F^{\textup{GW}} \cdot H^{\textup{GW}}$.

In total, we obtain
\[ \big\langle B_1, 2 B_2 \big\rangle = 2 \cdot G^{\textup{GW}} + 4 \cdot F^{\textup{GW}} \cdot H^{\textup{GW}} \,. \]

\vspace{7pt}
\noindent \textbf{Case (iv).}
Let $x, x' \in \p^1$ and $e' \in E$ be generic points,
and let $e \in E$ be the $i$-th $2$-torsion point.
Consider the exceptional curve at $(x,e)$,
\[ Z_1 = g(\rho^{-1} (x,x,e) ) \]
and the cycle which appeared in \eqref{owrogr} above,
\begin{equation*}
Z_2 = g(\rho^{-1}( \p^1 \times x' \times e' )) = (\p^1 \times e') + (x', -e').
\end{equation*}
We have $[Z_1] = A_i$ and $[Z_2] = B_1$.
Consider a $2$-marked stable map $f : C \to Y$ of class $\beta_{d, \textbf{k}}$ with markings $p_1, p_2 \in C$
incident to $Z_1, Z_2$ respectively.

If $f$ is of type (B), we must have $\pi(f(p_1)) \cap \pi(f(p_2)) \neq \varnothing$.
Hence, $f(p_2) = (x, e') + (x', -e')$ and the restriction $f|_{C_0}$ is an isomorphism onto
\[ \ell = ( \rho^{-1} ( x \times \p^1 \times e') ) = B_{(-e')} + (x,e') \]
In the splitting \eqref{Csplitting_eqn}, the component $C_1$ is attached to the component $C_0 \equiv \ell$ at $(x, -e') + (x, e')$.
Then, the contribution here matches precisely the contribution of the point $u_0$ in the K3 case of Section \ref{ABCD};
it is $\widetilde{G}^{\textup{GW}}$.

Assume $f$ is of type (A).
The line $L$ in \eqref{312} lies inside $\Hilb^2(B_{e''})$ for some $2$-torsion point $e'' \in E$.
Since $e'$ is generic, $\pi(L)$ is the line through $2x$ and $2x'$.
Consider the splitting \eqref{Csplitting_eqn2} with $C_1$ and $C_2$ mapping to the fibers of $\pi$
over $2x$ and $2x'$ respectively.
The contribution from maps $f|_{C_2}$ over $2x'$ is parallel to Section \ref{special_case_3}; it is $H^{\textup{GW}}$.
Let $\widetilde{H}_0$ (resp. $\widetilde{H}_1$) be the contribution from maps $f|_{C_1}$ over $2x$ if $e'' = e$ (resp. if $e'' \neq e)$.
Then, summing up over all $2$-torsion points, the total contribution is
$\widetilde{H}^{\textup{GW}} \cdot H^{\textup{GW}}$, where $\widetilde{H}^{\textup{GW}} = \widetilde{H}_0 + 3 \widetilde{H}_1$.

Adding up both cases, we obtain $\langle A_i, B_1 \rangle^Y = \widetilde{G}^{\textup{GW}} + \widetilde{H}^{\textup{GW}} \cdot H^{\textup{GW}}$.

\vspace{7pt}
\noindent \textbf{Case (v).} This is identical to the second case of (iv) above, with the difference that the second marked point does lie on $C_0$, not $C_2$.
\end{proof}

\subsection{Calculations}
\subsubsection{Initial Conditions} \label{Section_Initial_conditions}
Define the formal power series
\begin{align*}
 H & = \sum_{d \geq 0} \sum_{k \in \BZ} H_{d,k} y^k q^d = \big\langle B_2, B_2 \big\rangle^Y \\
 I & = \sum_{d \geq 0} \sum_{k \in \BZ} I_{d,k} y^k q^d = \big\langle \pt_Y \big\rangle^Y \\
 T & = \sum_{d \geq 0} \sum_{k \in \BZ} T_{d,k} y^k q^d = \big\langle 1 \big\rangle^Y,
\end{align*}
where $\big\langle 1 \big\rangle^Y$ is the Gromov-Witten potential \eqref{306} with no insertion,
and we have set $\zeta = 1$ in \eqref{306}. We have the following
initial conditions.
\begin{prop} \label{320} We have
\begin{enumerate}
\item[(i)] {\makebox[2.5cm]{$T_{0,k} = 8/k^3$ \hfill}} for all $k \geq 1$
\item[(ii)] {\makebox[2.5cm]{$T_{d,-2d} = 2/d^3$\hfill}} for all $d \geq 1$
\item[(iii)] {\makebox[2.5cm]{$H_{-1,0} = 1$\hfill}}
\item[(iv)] {\makebox[2.5cm]{$H_{d,k} = 0$\hfill}} if $(d = 0, k \leq -2)$ or $( d > 0, k < -2d )$
\item[(v)] {\makebox[2.5cm]{$T_{d,k} = 0$\hfill}} if $k < -2d$
\item[(vi)] {\makebox[2.5cm]{$I_{d,k} = 0$\hfill}} if $k < -2d$. 
\end{enumerate}
\end{prop}

\begin{proof}
\noindent \textbf{Case (i).} The moduli space $\Mbar_{0}(Y,\sum_i k_i A_i)$ is non-empty
only if there exists a $j \in \{1, \dots, 4\}$ with $k_i = \delta_{ij} k$ for all $i$.
Hence,
\[
T_{0,k}
= \sum_{k_1 + \dots + k_4 = k} \int_{[\Mbar_{0}(Y,\sum_{i} k_i A_i)]^{\text{vir}}} 1
\ =\ \sum_{i = 1}^{4} \int_{[\Mbar_{0}(Y,k A_i)]^{\text{vir}}} 1 \,.
\]
Since the term in the last sum is independent of $i$,
\begin{equation} \label{ssssss} T_{0,k} = 4 \int_{[\Mbar_{0}(Y,k A_1)]^{\text{vir}}} 1 \,. \end{equation}

Let $e \in E$ be the first $2$-torsion point,
let
\[ D_{X,1} = \{\, (x,x,e)\, |\, x \in \p^1\, \} \subset \p^1 \times \p^1 \times E \]
and consider the subscheme
\[ \Delta_1 = g( \rho^{-1}( D_{X,1} ) ) \]
which already appeared in Section \ref{Section_cohomology_of_Y}.
The divisor $\Delta_1$ is isomorphic to the exceptional divisor $E_1$ of the blowup $\rho : \widetilde{W} \to W$, see \eqref{blowup_map}.
Hence $\Delta_1 = \p(V)$, where
\[ V = \CO_{\p^1}(2) \oplus \CO_{\p^1} \to \p^1. \]
Under the isomorphism $\Delta_1 = \p(V)$, the map
\begin{equation} \pi : \Delta_1 \to \Delta_{\Hilb^2(\p^1)} \equiv \p^1. \label{ffibrt_map} \end{equation}
is identified  with the natural $\p(V) \to \p^1$.

The normal bundle of the exceptional divisor $E_1 \subset \widetilde{W}$ is $\CO_{\p(V)}(-1)$.
Hence, taking the $\BZ_2$ quotient \eqref{gquotientmap} of $\widetilde{W}$, the normal bundle of $\Delta_1 \subset Y$ is
\[ N = N_{\Delta_1/Y} = \CO_{\p(V)}(-2). \]

For $k \geq 1$, the moduli space
\[ M = \Mbar_0(Y, kA_1) \]
parametrizes maps to the fibers of the fibration \eqref{ffibrt_map}.
Since the normal bundle $N$ of $\Delta_{1}$ has degree $-2$ on each fiber,
there is no infinitesimal deformations of maps out of $\Delta_1$. Hence,
$M$ is isomorphic to $\Mbar_0( \p(V), d \Ff)$, where
$\Ff$ is class of a fiber of $\p(V)$.
In particular, $M$ is smooth of dimension $2k-1$.

By smoothness of $M$ and convexity of $\p(V)$ in class $k \Ff$, the virtual class of~$M$ is the Euler class of the obstruction bundle ${\rm Ob}$
with fiber
\[ {\rm Ob}_f = H^1(C, f^{\ast} T_Y) \]
over the moduli point $[f : C \to Y] \in M$.
The restriction of the tangent bundle $T_Y$ to a fixed fiber $A_0$ of \eqref{ffibrt_map} is
\[ T_{Y}|_{A_0} \cong \CO_{A_0}(2) \oplus \CO_{A_0} \oplus \CO_{A_0}(-2), \]
Hence, 
\[ {\rm Ob}_f = H^1(C, f^{\ast} T_Y) = H^1(C, f^{\ast} N). \]

Consider the relative Euler sequence of $p : \p(V) \to \p^1$,
\begin{equation} 0 \to \Omega_{p} \to p^{\ast} V \otimes \CO_{\p(V)}(-1) \to \CO_{\p(V)} \to 0. \label{euler} \end{equation}
By direct calculation, $\Omega_{p} = \CO_{\p(V)}(-2) \otimes p^{\ast} \CO_{\p^1}(-2)$. Hence, twisting \eqref{euler} by $p^{\ast} \CO_{\p^1}(2)$, we obtain the sequence
\begin{equation} 0 \to N \to p^{\ast} V(2) \otimes \CO_{\p(V)}(-1) \to p^{\ast} \CO_{\p^1}(2) \to 0. \label{euler2} \end{equation}
Let $q : \CC \to M$ be the universal curve and let $f : \CC \to \Delta_1 \subset Y$ be the universal map.
Pulling back \eqref{euler2} by $f$, pushing forward by $q$ and taking cohomology we obtain the exact sequence
\begin{equation}
 0 \to R^0q_{\ast} f^{\ast} p^{\ast} \CO_{\p^1}(2) \to R^1q_{\ast} f^{\ast} N \to R^1q_{\ast} f^{\ast} p^{\ast} V(2) \otimes \CO_{\p(V)}(-1) \to 0.
\end{equation}
The bundle $R^1q_{\ast} f^{\ast} N$ is the obstruction bundle ${\rm Ob}$, and
\[ R^0q_{\ast} f^{\ast} p^{\ast} \CO_{\p^1}(2) = q_{\ast} q^{\ast} p^{\prime \ast} \CO_{\p^1}(2) = p^{\prime \ast} \CO_{\p^1}(2), \]
where $p' : M \to \p^1$ is the map induced by $p : \p(V) \to \p^1$.
We find 
\[ c_1(p^{\prime \ast} \CO_{\p^1}(2)) = 2 p^{\prime \ast} \pt_{\p^1}, \]
where $\pt_{\p^1}$ is the class of a point on $\p^1$. Taking everything together, we have
\begin{align}
\int_{[\Mbar_{0}(Y,k A_1)]^{\text{vir}}} 1
& = \int_M e(R^1q_\ast f^{\ast} N) \notag \\
& = \int_M c_1(p^{\prime \ast} \CO_{\p^1}(2)) c_{2k-2}( R^1q_{\ast} f^{\ast} p^{\ast} V(2) \otimes \CO_{\p(V)}(-1) ) \notag \\
 & = 2 \int_{M_x} c_{2k-2}( R^1q_{\ast} f^{\ast} p^{\ast} V(2) \otimes \CO_{\p(V)}(-1) )|_{M_x}, \label{rree}
\end{align}
where $M_x = \Mbar_{0}(\p^1,k)$ is the fiber of $p' : M \to \p^1$ over some $x \in \p^1$. Since
\[ p^{\ast} V(2) \otimes \CO_{\p(V)}(-1)|_{p^{-1}(x)} = \CO_{\p(V)_x}(-1) \oplus \CO_{\p(V)_x}(-1), \]
the term \eqref{rree} equals $2 \int_{\Mbar_{0,0}(\p^1, k)} c_{2k-2}(\CE)$,
where $\CE$ is the bundle with fiber
\[ H^1(C, f^{\ast} \CO_{\p^1}(-1)) \oplus H^1(C, f^{\ast} \CO_{\p^1}(-1)). \]
over a moduli point $[f \colon C \to \p^1] \in M_x$.
Hence, using the Aspinwall-Morrison formula \cite[Section 27.5]{MirSym} the term \eqref{rree} is
\[ \int_{[\Mbar_{0}(Y,k A_1)]^{\text{vir}}} 1 = 2 \cdot \int_{\Mbar_{0,0}(\p^1, k)} c_{2k-2}(\CE) = \frac{2}{k^3}. \]
Combining with \eqref{ssssss}, the proof of case (i) is complete.

\vspace{8pt}
\noindent \textbf{Case (ii) and (v).}
Let $f \colon C \to Y$ be a stable map of genus $0$ and class $d B_2 + \sum k_i A_j$.
Then, $f$ maps into the fiber of
\[ \pi : Y \to \Hilb^2(\p^1) \]
over some diagonal point $2x \in \Delta_{\Hilb^2(\p^1)}$.
The reduced locus of such a fiber is the union
\begin{equation} \Sigma_{x} \cup A_{x,e_1} \cup \ldots \cup A_{x,e_4} \label{310} \end{equation}
where $e_1, \dots, e_4 \in E$ are the $2$-torsion points of $E$,
\[ A_{x,e} = g(\rho^{-1}( x \times x \times e )) \]
is the exceptional curve of $\Hilb^2(X)$ at $(x,e) \in X$,
and $\Sigma_{x}$ is the fiber of the addition map $\Hilb^2(F_x) \to F_x = E$
over the origin $0_E$.
Hence,
\[ f_{\ast} [ C ] = a [\Sigma_x] + \sum_i b_i [ A_{x,e_i} ] \]
for some $a, b_1, \dots, b_4 \geq 0$.
Since $[ A_{x,e_i} ] = A_i$ and
\[ [ \Sigma_x ] = B_2 - \frac{1}{2} ( A_1 + A_2 + A_3 + A_4 ) \]
we must have $d = a$ and therefore
\[ f_{\ast} [ C ] = d B_2 + \sum_i (b_i - d/2) A_i. \]
Since $b_i \geq 0$ for all $i$, we find $\sum_i k_i \geq -2d$ with equality
if and only if $k_i = -d/2$ for all $i$. This proves (v) and shows
\begin{equation} T_{d, -2d} = \int_{[\Mbar_{0}(Y, d B_2 - \sum_{i} (d/2) A_i )]^{\text{vir}}} 1 \,. \label{xxvvb}\end{equation}
Moreover, if $f : C \to Y$ has class $d B_2 - \sum_i (d/2) A_i$,
it is a degree $d$ cover of the curve $\Sigma_x$ for some $x$.

We evaluate the integral \eqref{xxvvb}.
Let $Z'$ be the proper transform of
\[ \p^1 \times E \hookrightarrow W, \ (x,e) \mapsto (x,x,e) \]
under the blowup map $\rho \colon \widetilde{W} \to W$, and let 
\[ Z = g(Z') = Z' / \BZ_2 \subset Y \]
be its image under $g : \widetilde{W} \to Y$.
The projection map $\pr_{1,3} \circ \rho : Z' \to \p^1 \times E$ descends by $\BZ_2$ quotient to the isomorphism
\begin{equation} (\tau_{|Z}, \pi_{|Z}) : Z \to \p^1 \times \p^1, \label{gfrmgfg} \end{equation}
where $\tau : Y \to \p^1$ is the morphism defined in \eqref{some_mmm_diagram}.
Under the isomorphism \eqref{gfrmgfg},
the curve $\Sigma_x$ equals $\p^1 \times x$.
Since moreover the normal bundle of $Z \subset Y$ has degree $-2$ on $\Sigma_x$, we find
\[ \Mbar_{0}(Y, d B_2 - 2 d A) \cong \Mbar_0(\p^1, d) \times \p^1. \]
The normal bundle $Z \subset Y$ is the direct sum
\[ N = N_{Z/Y} = \pr_{1}^{\ast} \CO_{\p^1}(a) \otimes \pr_{2}^{\ast} \CO_{\p^1}(-2). \]
for some $a$. We determine $a$. 
Under the isomorphism \eqref{gfrmgfg}, the curve 
\[ R = x \times \p^1 \subset \p^1 \times \p^1 \]
corresponds
to the diagonal in a generic fiber of $\tau : Y \to \p^1$.
The generic fiber of $\tau$ is isomorphic to $\p^1 \times \p^1$, hence
$c_1(N) \cdot R = 2$ and $a=2$.
The result now follows by an argument parallel to (i).

\vspace{8pt}
\noindent \textbf{Case (iii).}
This follows directly from the proof of Proposition \ref{comparision_proposition} Case (i)
since the line in \eqref{312} has class $B_1 - A_i$ for some $i$.

\vspace{8pt}
\noindent \textbf{Case (iv).} Let $f : C \to Y$ be a stable map of genus $0$ and class $\beta_{d, \textbf{k}}$
incident to the cycles $Z_1, Z_2$ of the proof of Proposition \ref{comparision_proposition} Case (i).
Then, there exists an irreducible component $C_0 \subset C$ which maps isomorphically to the line $L$ considered in \eqref{312}.
We have $[L] = B_1 - A_i$ for some $i$

Since all irreducible components of $C$ except for $C_0$ gets mapped under $f$ to curves of the form $\Sigma_x$ or $A_{x,i}$,
we have
\begin{align*}
f_{\ast} [ C ] = \beta_{d, \textbf{k}}
& = [L] + d [\Sigma_x ] + \sum_j b_j A_j \\
& = B_1 + d B_2 + \sum_j (-d/2 - \delta_{ij} + b_j) 
\end{align*}
for some $b_1, \dots, b_4 \geq 0$.
If $d = 0$ we find $k = \sum_{i} k_i \geq -1$.
If $d > 0$, then $f$ maps to at least one curve of the form $\Sigma_x$ with non-zero degree.
Since $L$ and $\Sigma_x$ are disjoint, we must have $b_j > 0$ for some $j$.
This shows
$k = \sum_j k_j \geq  -2d$.

\vspace{8pt}
\noindent \textbf{Case (vi).} This case follows by an argument parallel to (iv).
\end{proof}

\subsubsection{The system of equations} \label{system_of_equations}
Let $\frac{d}{d q}$ and $\frac{d}{d y}$ be the formal differentiation operators with respect to $q$ and $y$ respectively.
We will use the notation
\[ \partial_{\tau} = q \frac{d}{d q} \quad \text{ and } \quad \partial_z = y \frac{d}{d y}. \]

The WDVV equation \eqref{313},
applied to the cohomology insertions
\[ \xi = (\gamma_1, \dots, \gamma_4) \]
specified below yield the following relations:
\begin{equation} \label{relations_first_batch}
\begin{alignedat}{2}
 & \xi = (B_2, D_2, D_2, \Delta)     && : \quad \quad \big\langle B_2, A \big\rangle = - \frac{1}{2} \partial_z(H) \\
 & \xi = (B_2, D_2, D_2, D_1) && : \quad \quad \big\langle B_1, B_2 \big\rangle = \partial_{\tau} H + \frac{1}{2} I \\
 & \xi = (A, D_2, D_2, \Delta) && : \quad \quad \big\langle A, A \big\rangle = \frac{1}{4} \partial_z^2 H - \frac{1}{4} I \\
 & \xi = (A, D_2, D_2, D_1) && : \quad \quad \big\langle B_1 , A \big\rangle = - \frac{1}{2} \partial_z \partial_{\tau} H \\
 & \xi = (B_1, D_2, D_2, \Delta) && : \quad \quad \big\langle B_1 , A \big\rangle - \frac{1}{4} \partial_z I = - \frac{1}{2} \partial_z \big\langle B_1, B_2 \big\rangle \\
 & \xi = (B_1, D_2, D_2, D_1) && : \quad \quad 2 \big\langle B_1, B_1 \big\rangle + \partial_{\tau} I = 2 \partial_{\tau} \big\langle B_1, B_2 \big\rangle \\
 &                                    && \quad \quad \quad \Leftrightarrow \big\langle B_1, B_1 \big\rangle = \partial_{\tau}^2 H
\end{alignedat}
\end{equation}
Using \eqref{relations_first_batch} and the WDVV equations \eqref{313} with insertions $\xi$ further yields:
\vspace{3pt}

\noindent \textbf{W1.} $\xi = (B_2, D_1, D_1, D_2)$:
\begin{equation*} 0 = 2 \partial_{\tau}^2 H + 2 \partial_{\tau} I - H \cdot \partial_{\tau}^3 T + \frac{1}{2} \partial_{z} H \cdot \partial_{z} \partial_{\tau}^2 T \end{equation*}

\noindent \textbf{W2.} $\xi = (B_2, \Delta, \Delta, D_2)$:
\begin{equation*} 0 = 2 \partial_{z}^2 H + 4 \partial_{\tau} H + 2 I - H \cdot \partial_{z}^2 \partial_{\tau} T + \frac{1}{2} \partial_{z} H \cdot ( 4 + \partial_{z}^3 T ) \end{equation*}

\noindent \textbf{W3.} $\xi = (B_2, \Delta, \Delta, D_1)$:
\begin{multline*}
0 = 4 \partial_{\tau}^2 H + 2 \partial_{\tau} I - \partial_{z}^2 I + \frac{1}{2} \partial_{z} \partial_{\tau} H \cdot (4 + \partial_{z}^3 T) - \partial_{\tau} H \cdot \partial_{z}^2 \partial_{\tau} T \notag \\
  - \frac{1}{2} \partial_{z}^2 H \cdot \partial_{z}^2 \partial_{\tau} T + \partial_{z} H \cdot \partial_{z} \partial_{\tau}^2 T \end{multline*}

\noindent \textbf{W4.} $\xi = (A, \Delta, \Delta, D_2)$:
\begin{equation*}
 0 = -8 \partial_{z} \partial_{\tau} H - 4 \partial_{z}^3 H + 8 \partial_{z} I + 2 \partial_{z} H \cdot \partial_{z}^2 \partial_{\tau} T - \partial_{z}^2 H \cdot (4 + \partial_{z}^3 T) + I \cdot (4 + \partial_{z}^3 T) 
\end{equation*}

\noindent \textbf{W5.} $\xi = (A, \Delta, D_1, D_1)$:
\begin{multline*}
 0 = -2 \partial_{\tau}^2 I + \frac{1}{2} \partial_{z}^2 \partial_{\tau} H \cdot \partial_{z}^2 \partial_{\tau} T - \partial_{z} \partial_{\tau} H \cdot \partial_{z} \partial_{\tau}^2 T - \frac{1}{2} \partial_{z}^3 H \cdot \partial_{z} \partial_{\tau}^2 T \notag \\
 + \partial_{z}^2 H \cdot \partial_{\tau}^3 T - \frac{1}{2} \partial_{\tau} I \cdot \partial_{z}^2 \partial_{\tau} T + \frac{1}{2} \partial_{z} I \cdot \partial_{z} \partial_{\tau}^2 T 
\end{multline*}

\noindent \textbf{W6.} $\xi = (B_1, D_1, D_1, D_2)$:
\begin{equation*}
 0 = 2 \partial_{\tau}^3 H - \partial_{\tau}^2 I - \partial_{\tau} H \cdot \partial_{\tau}^3 T - \frac{1}{2} I \cdot \partial_{\tau}^3 T + \frac{1}{2} \partial_{z} \partial_{\tau} H \cdot \partial_{z} \partial_{\tau}^2 T 
\end{equation*}

\subsubsection{Non-degeneracy of the equations}

\begin{prop} \label{321}
The initial conditions of
Proposition \ref{320}
and the equations \textbf{W1}~-~\textbf{W6}
together determine $H_{d,k}, I_{d,k}, T_{d,k}$ for all $d$ and $k$.
\end{prop}

\begin{proof}[Proof of Proposition \ref{321}]
For all $d,k$, taking the coefficient of $q^d y^k$ in equations \textbf{W1} - \textbf{W6} yields
\begin{multline}
2 d^2 H_{d,k} + 2 d I_{d,k} = \sum_{j,l} (d-l)^2 \Big( (d-l) - \frac{1}{2} j (k-j) \Big) H_{l,j} T_{d-l,k-j} \tag{W1}
\end{multline}
\begin{multline}
(2 k (k+1) + 4 d) H_{d,k} + 2 I_{d,k} = \sum_{j,l} (k-j)^2 \Big( (d-l) - \frac{1}{2} j (k-j) \Big) H_{l,j} T_{d-l,k-j} \tag{W2}
\end{multline}
\begin{multline}
2d (2d + k) H_{d,k} + (2d - k^2) I_{d,k} = \tag{W3} \\
- \sum_{j,l} (k-j) \Big( j (d-l) - l (k-j) \Big) \Big( (d-l) - \frac{1}{2} j (k-j) \Big) H_{l,j} T_{d-l,k-j} \notag
\end{multline}
\begin{multline}
(2 k + 1) I_{d,k} - k( k^2 + k + 2d )H_{d,k} = \tag{W4} \\ 
- \frac{1}{2} \sum_{j,l} (k-j)^2 \Big( (j (d-l) - \frac{1}{2} (k-j)) H_{l,j} + \frac{1}{2} (k-j) I_{l,j} \Big) T_{d-l,k-j} \notag
\end{multline}
\begin{multline}
 2 d^2 I_{d,k} =
\sum_{j,l} (d-l) \Big( j (d-l) - l (k-j) \Big) \cdot \tag{W5} \\
\Big( j (d-l) H_{l,j} - \frac{1}{2} j^2 (k-j) H_{l,j} + \frac{1}{2} (k-j) I_{l,j} \Big) T_{d-l,k-j}
\end{multline}
\begin{multline}
 2 d^3 H_{d,k} - d^2 I_{d,k} =
\sum_{j,l} (d-l)^2 \cdot \tag{W6} \\
\Big( (d-l) ( l H_{l,j} + \frac{1}{2} I_{l,j} ) - \frac{1}{2} j l (k-j) H_{l,j} \Big) T_{d-l,k-j}.
\end{multline}

\noindent \emph{Claim 1.} The initial conditions and \textbf{W1} - \textbf{W6} determine
$H_{0,k}, I_{0,k}, T_{0,k}$ for all $k$, except for $H_{0,0}$

\vspace{1pt}
\noindent \emph{Proof of Claim 1.}
The values $T_{0,k}$ are determined by the initial conditions.
Consider the equation \textbf{W2} for $(d,k) = (0,0)$. Plugging in $(d,k) = (0,0)$ and using $H_{0,-1} = 1, T_{0,1} = 8$, we find $I_{0,0} = 2$.

Let $d = 0$ and $k > 0$, and assume we know the values $H_{0,j}, I_{0,j}$ for all $j < k$ except for $H_{0,0}$. Then, equations \textbf{W3} and \textbf{W4} read
\begin{align*} 
-4 k^2 I_{0,k} + \text{ (known terms) } & = 0  \\
b -4 k^2 (k + 1) H_{0,k} + \text{ (known terms) } & = 0.
\end{align*}
Hence, also $I_{0,k}$ and $H_{0,k}$ are uniquely determined. By induction, the proof of Claim 1 is complete.\qed

Let $d >0$. We argue by induction.
Calculating the first values of $H_{0,k}, I_{0,k}$ and $T_{0,k}$, and plugging them into equations \textbf{W1} - \textbf{W6}
for $(d,k) = (1,-2)$ and $(d,k) = (1,-1)$, we find by direct calculation that the values
\[ H_{0,0},\ H_{1,-2},\ H_{1,-1},\ I_{1,-2},\ I_{1,-1},\ T_{1,-1},\ T_{1,0} \]
are determined.

Let now $(d = 1, k \geq 0)$ or $(d > 1, k \geq -2d)$, and assume we know the values $H_{l,j}, I_{l,j}, T_{l,j}$ for all $l < d, j \leq k + 2 (d-l)$ and for all $l = d, j < k$.
Also assume, that we know $T_{d,k}$.
The proof of Proposition \ref{321} follows now from the following claim.

\vspace{4pt}
\noindent \emph{Claim 2:} The values $H_{d,k}, I_{d,k}, T_{d,k + 1}$ are determined.
\vspace{4pt}

\noindent \emph{Proof of Claim 2.}
Solving for the terms $H_{d,k}, I_{d,k}, T_{d,k+1}$ in the equations \textbf{W1}, \textbf{W6} and \textbf{W5}, we obtain:
\begin{align}
 2 d^2 H_{d,k} + 2d I_{d,k} - d^2 \left(d + \frac{1}{2} (k+1)\right) T_{d,k+1} & = \text{ (known terms) } \tag{W1} \\
 2 d^3 H_{d,k} - d^2 I_{d,k} & = \text{ (known terms) } \tag{W6} \\
 -2 I_{d,k} + \left(d + \frac{1}{2} (k + 1)\right) T_{d,k+1} & = \text{ (known terms) }, \tag{W5}
\end{align}
where in the last line we divided by $d^2$. These equations in matrix form read
\[
\begin{pmatrix}
2d & 2 & -d (d + \frac{1}{2} (k+1)) \\ 2d & -1 & 0 \\ 0 & -2 & d + \frac{1}{2} (k+1) 
\end{pmatrix} \cdot
\begin{pmatrix}
 H_{d,k} \\ I_{d,k} \\ T_{d,k+1}
\end{pmatrix} = \text{ (known terms) }
\]
The matrix on the left hand side has determinant $(2d - 3) (k + 2d + 1) d$.
It vanishes if $d = \frac{3}{2}$ or $k = -2d - 1$ or $d = 0$.
By assumption, each of these cases were excluded.
Hence the values $H_{d,k}, I_{d,k}, T_{d,k + 1}$ are uniquely determined.
\end{proof}

\noindent \textbf{Remark.} We have selected very particular WDVV equations for $Y$ above. Using
additional equations, one may show that the values
\[ H_{0,-1} = 1, \quad T_{0,0} = 0, \quad T_{0,1} = 8, \quad T_{1,-2} = 2 \]
together with the vanishings of Proposition~\ref{320} (iv) - (vi) suffice to determine the series $H, I, T$.

\subsubsection{Solution of the equations}
Let $z \in \BC$ and $\tau \in \BH$ and consider the actual variables
\begin{equation} y = -e^{2 \pi i z} \quad \text{ and } \quad q = e^{2 \pi i \tau} \,. \label{var_change_1000} \end{equation}
Let $F(z,\tau)$ and $G(z,\tau)$ be the functions \eqref{FFFdef} and \eqref{G_Function_def} respectively.
\begin{thm} \label{fIT} \label{HIT} We have
\begin{align*}
 H & = F(z,\tau)^2 \\
 I & = 2\, G(z,\tau) \\
 T & = 8 \sum_{k \geq 1} \frac{1}{k^3} y^k + 12 \sum_{k,n \geq 1} \frac{1}{k^3} q^{kn} \\
 & \quad + 8 \sum_{k,n \geq 1} \frac{1}{k^3} (y^k + y^{-k}) q^{kn}  + 2 \sum_{k,n \geq 1} \frac{1}{k^3} (y^{2k} + y^{-2k}) q^{(2n-1) k}.
\end{align*}
under the variable change $y = -e^{2 \pi i z}$ and $q=e^{2 \pi i \tau}$.
\end{thm}

\begin{proof}
By Proposition \ref{321}, it suffices to show that the functions defined in the
statement of Theorem \ref{fIT} satisfy the initial conditions of Proposition \ref{320}
and the WDVV equations \textbf{W1} - \textbf{W6}.
By a direct check, the initial conditions are satisfied. We consider the WDVV equations.

For the scope of this proof,
define $H = F(z,\tau)^2$ and $I = 2\, G(z,\tau)$ and
\begin{multline*} T = 8 \sum_{k \geq 1} \frac{1}{k^3} y^k + 12 \sum_{k,n \geq 1} \frac{1}{k^3} q^{kn} \\
+ 8 \sum_{k,n \geq 1} \frac{1}{k^3} (y^k + y^{-k}) q^{kn}  + 2 \sum_{k,n \geq 1} \frac{1}{k^3} (y^{2k} + y^{-2k}) q^{(2n-1) k}.
\end{multline*}
considered as a function in $z$ and $\tau$ under the variable change \eqref{var_change_1000}.
We show these functions satisfy the equations \textbf{W1} - \textbf{W6}.

For a function $A(z,\tau)$, we write
\begin{equation*}
A^{\bullet} = \partial_z A := \frac{1}{2 \pi i} \frac{\partial A}{\partial z} = y \frac{d}{d y} A, \quad \quad
A' = \partial_\tau A := \frac{1}{2 \pi i} \frac{\partial A}{\partial \tau} = q \frac{d}{d q} A
\end{equation*}
for the differential of $A$ with respect to $z$ and $\tau$ respectively.

For $n \geq 1$, define the deformed Eisenstein series \cite{O}
\begin{align*}
 J_{2,n}(z,\tau) & = \delta_{n,1} \frac{y}{y - 1} + B_n - n \sum_{k, r \geq 1} r^{n - 1} (y^k + (-1)^n y^{-k}) q^{k r} \\
 J_{3,n}(z, \tau) & = -B_n \Big( 1 - \frac{1}{2^{n - 1}} \Big) - n \sum_{k , r \geq 1} \left( r - 1/2 \right)^{n - 1} (y^k + (-1)^n y^{-k}) q^{k (r - \frac{1}{2})},
\end{align*}
where $B_{n}$ are the Bernoulli numbers (with $B_1 = -\frac{1}{2}$) and we used the variable change \eqref{var_change_1000}. We also let
\begin{align*} G_n(z, \tau) & = J_{4,n}(2z, 2 \tau) \\
& = -B_n \Big( 1 - \frac{1}{2^{n - 1}} \Big) - n \sum_{k , r \geq 1} (r - 1/2)^{n - 1} (y^{2k} + (-1)^n y^{-2k}) q^{k (2 r - 1)}. \end{align*}
Then we have
\begin{equation}
\begin{aligned}
 \partial_z^3 T & = -4 - 8 J_{2,1} - 16 G_1 \\
 \partial_z^2 \partial_\tau T & = -4 J_{2,2} - 8 G_2 \\
 \partial_z \partial_\tau^2 T & = - \frac{8}{3} J_{2,3} - \frac{16}{3} G_3 \\
 \partial_\tau^3 T & = -2 J_{2,4} - 4 G_4 + \frac{1}{20} E_4.
\end{aligned} \label{Trelations}
\end{equation}
Since $T(z,\tau)$ appears only as a third derivative in the equations \textbf{W1} - \textbf{W6},
we may trade it for deformed Eisenstein series using equations \eqref{Trelations}.

The first theta function $\vartheta_1(z,\tau)$ satisfies the heat equation
\[ \partial_z^2 \vartheta_1 = 2 \partial_{\tau} \vartheta_1, \]
which implies that $F = F(z,\tau) = \vartheta_1(z,\tau) / \eta^3(\tau)$ satisfies
\begin{equation}
 \partial_{\tau}F = \frac{1}{2} \partial_z^2 F - \frac{1}{8} E_2(\tau) F, \label{yyyy}
\end{equation}
where $E_2(\tau) = 1 - 24 \sum_{d \geq 1} \sum_{k|d} k q^d$ is the second Eisenstein series.
With a small calculation, we obtain the relation
\begin{equation}
 I = 4 \partial_{\tau}(H) - \partial_z^2(H) + E_2 \cdot H. \label{ItoH}
\end{equation}

Hence, using equations \eqref{Trelations} and \eqref{ItoH},
we may replace in the equations \textbf{W1} - \textbf{W6}
the function $T$ with deformed Eisenstein Series and $I$ with terms involving only $H$ and $E_2$.
Hence, we are left with a system of partial differential equations
between the square of the Jacobi theta function $F$, deformed Eisenstein series and classical modular forms.

These new equations may now be checked directly by methods of complex analysis as follows.
Divide each equation by $H$; derive how the quotients
\[ \frac{H^{k \bullet}}{H} \quad \text{ and } \quad \frac{H^{k \prime}}{H} \]
(with $H^{k \bullet}$ and $H^{k \prime}$ the $k$-th derivative of $H$ with respect to $z$ resp. $\tau$ respectively)
transform under the variable change
\[ (z,\tau) \mapsto (z + \lambda \tau + \mu, \tau) \quad \quad (\lambda, \mu \in \BZ) \,; \]
using the periodicity properties of the deformed Eisenstein series proven in \cite{O},
show that each equation is is double periodic in $z$;
calculate all appearing poles using the expansions of the deformed Eisenstein series in \cite{O};
prove all appearing poles cancel;
finally prove that the constant term is $0$ by evaluating at $z = 1/2$.
Using this procedure, the proof reduces to a long, but standard calculation.
\end{proof}

\subsubsection{Proof of Theorem \ref{Hilb2P1e_complete_evaluation_theorem}}
We will identify functions in $(z,\tau)$ with their expansion in $y,q$ under the variable change \eqref{var_change_1000}.
By Proposition \ref{comparision_proposition},
the definition of $H$ in \eqref{Section_Initial_conditions}, and Theorem~\ref{HIT}, we have
\[ (F^{\textup{GW}})^2 = \langle B_2, B_2 \rangle^Y = H = F(z,\tau)^2 \]
which implies
\begin{equation} F^{\textup{GW}}(y,q) = \pm F(z,\tau) \,. \label{12355} \end{equation}
By definition \eqref{FGW}, the $y^{-1/2} q^0$-coefficient of $F^{\textup{GW}}(y,q)$ is~$1$.
Hence, there is a positive sign in \eqref{12355}, and we have equality.
This proves the first equation of Theorem \ref{Hilb2P1e_complete_evaluation_theorem}.
The case $G^{\textup{GW}} = G$ is parallel.

Finally, the two remaining cases follow  directly from
Proposition \ref{comparision_proposition},
the relations \eqref{relations_first_batch}
and Theorem~\ref{HIT}.
This completes the proof of Theorem \ref{Hilb2P1e_complete_evaluation_theorem}.

\section{Quantum Cohomology} \label{Section_Quantum_Cohomology}
\subsection{Overview}
Let $S$ be a K3 surface.
In section \ref{FockSpace_section} we recall basic facts about the Fock space
\[ \CF(S) = \bigoplus_{d \geq 0} H^{\ast}( \Hilb^d(S) ; \BQ ) \,. \]
In Section \ref{Section_Main_Conjecture_WDVV} we define a $2$-point quantum operator $\CE^{\Hilb}$,
which encodes the quantum multiplication with a divisor class.
In section \ref{Main_conjecture_section} we introduce natural operators $\CE^{(r)}$ acting on $\CF(S)$.
In Section \ref{Main_conjecture_section2}, we state a series of conjectures
which link $\CE^{(r)}$ to the operator $\CE^{\Hilb}$.
In section \ref{Main_conjecture_examples_section} we present several example calculations and 
prove our conjectures in the case of $\Hilb^2(S)$.
Here, we also discuss the relationship of the K3 surface case
to the case of $\CA_1$-resolution studied by Maulik and Oblomkov in \cite{MO2}.

\subsection{The Fock space} \label{FockSpace_section}
The Fock space of the K3 surface $S$,
\begin{equation} \CF(S) = \bigoplus_{d \geq 0} \CF_d(S) = \bigoplus_{d \geq 0} H^{\ast}(\Hilb^d(S),\BQ) \label{P1}, \end{equation}
is naturally bigraded with the $(d,k)$-th summand given by
\[ \CF_{d}^k(S) = H^{2 (k + d)}(\Hilb^d(S),\BQ) \]
For a bihomogeneous element $\mu \in \CF_{d}^k(S)$, we let
\[ | \mu | = d, \quad \quad k(\mu) = k. \]
The Fock space
$\CF(S)$ carries a natural scalar product $\big\langle \cdot\, \big|\, \cdot \big\rangle$
defined by declaring the direct sum \eqref{P1} orthogonal and setting 
\[ \big\langle \mu\, \big|\, \nu \big\rangle = \int_{\Hilb^d(S)} \mu \cup \nu \]
for $\mu, \nu \in H^{\ast}(\Hilb^d(S),\BQ)$.
If $\alpha, \alpha' \in H^{\ast}(S,\BQ)$ we also write
\[ \langle \alpha, \alpha' \rangle = \int_S \alpha \cup \alpha'. \]
If $\mu, \nu$ are bihomogeneous, then $\langle \mu | \nu \rangle$ is nonvanishing only if $|\mu| = |\nu|$ and $k(\mu) + k(\nu) = 0$.
For all $\alpha \in H^{\ast}(S,\BQ)$ and $m \neq 0$, the Nakajima operators $\Fp_m(\alpha)$ act on $\CF(S)$
bihomogeneously of bidegree $(-m, k(\alpha))$,
\[ \Fp_{m}(\alpha) : \CF_d^k \to \CF_{d-m}^{k+ k(\alpha)} \,. \]
The commutation relations
\begin{equation} [ \Fp_{m}(\alpha), \Fp_{n}(\beta) ] = -m \delta_{m + n,0} \langle \alpha, \beta \rangle\, {\rm id}_{\CF(S)}, \label{N1} \end{equation}
are satisfied for all $\alpha, \beta \in H^{\ast}(S)$ and all $m, n \in \BZ \setminus \{ 0 \}$.

The inclusion of the diagonal $S \subset S^m$ induces a map
\[ \tau_{\ast m} : H^{\ast}(S,\BQ) \to H^{\ast}(S^m,\BQ) \stackrel{\sim}{=} H^{\ast}(S,\BQ)^{\otimes m} \, .\]
For $\tau_{\ast} = \tau_{\ast 2}$, we have 
\[ \tau_{\ast}(\alpha) = \sum_{i,j} g^{ij} \, (\alpha \cup \gamma_i) \otimes \gamma_j, \]
where $\{ \gamma_i \}_i$ is a basis of $H^{\ast}(S)$
and $g^{ij}$ is the inverse of the intersection matrix $g_{ij} = \langle \gamma_i, \gamma_j \rangle$.

For $\gamma \in H^{\ast}(S,\BQ)$ and $n \in \BZ$ define the Virasoro operator
\[ L_n(\gamma) = - \frac{1}{2} \sum_{k \in \BZ} : \Fp_k \Fp_{n-k} : \tau_{\ast}(\gamma), \]
where $: -- :$ is the normal ordered product \cite{Lehn}
and we used
\[ \Fp_{k}\Fp_{l} \cdot \alpha \otimes \beta = \Fp_{k}(\alpha) \Fp_{l}(\beta). \]
We are particularly interested in the degree $0$ Virasoro operator
\begin{align*}
L_0(\gamma)
& = - \frac{1}{2} \sum_{k \in \BZ \setminus 0} : \Fp_k \Fp_{-k} : \tau_{\ast}(\gamma) \\
& = - \sum_{k \geq 1} \sum_{i,j} g^{ij} \Fp_{-k}(\gamma_i \cup \gamma) \Fp_k(\gamma_j) \,,
\end{align*}
The operator $L_0(\gamma)$ is characterized by the commutator relations
\[ \big[ \Fp_{k}(\alpha), L_0(\gamma) \big] = k \, \Fp_{k}(\alpha \cup \gamma). \]
Let $e \in H^{\ast}(S)$ denote the unit. The restriction of $L_0(\gamma)$ to $\CF_d(S)$,
\[ L_0(\gamma)\big|_{\CF_d(S)} : H^{\ast}(\Hilb^d(S),\BQ) \to H^{\ast}(\Hilb^d(S),\BQ) \]
is the cup product by the class
\begin{equation} D(\gamma) = \frac{1}{(d-1)!} \Fp_{-1}(\gamma) \Fp_{-1}(e)^{d-1} \in H^{\ast}(\Hilb^d(S),\BQ) \label{divisor_class_hilbd} \end{equation}
of subschemes incident to $\gamma$, see \cite{Lehn2}.
In the special case $\gamma = e$, the operator $L_0 = L_0(e)$ is the \emph{energy operator},
\begin{equation} L_0\big|_{\CF_d(S)} = d \cdot {\rm id}_{\CF_d(S)} \,. \label{energy_operator} \end{equation}

Finally, define Lehn's diagonal operator \cite{Lehn2}
\[ \partial = - \frac{1}{2} \sum_{i,j \geq 1} ( \Fp_{-i} \Fp_{-j} \Fp_{i+j} + \Fp_i \Fp_j \Fp_{-(i+j)} ) \tau_{3 \ast}( [S] ) \,. \]
For $d \geq 2$, the operator $\partial$ acts on $\CF_d(S)$ by cup product with $-\frac{1}{2} \Delta_{\Hilb^d(S)}$, where
\[ \Delta_{\Hilb^d(S)} = \frac{1}{(d-2)!} \Fp_{-2}(e) \Fp_{-1}(e)^{d-2} \]
is the class of the diagonal in $\Hilb^d(S)$.

\subsection{The WDVV equation} \label{Section_Main_Conjecture_WDVV}
Let $S$ be an elliptic K3 surface with a section. Let $B$ and $F$ be the section and fiber class respectively,
and let
\[ \beta_h = B + hF \,. \]
For $d \geq 1$ and cohomology classes $\gamma_1, \dots, \gamma_m \in H^{\ast}(\Hilb^d(S);\BQ)$ define the quantum bracket
\begin{equation*}
 \big\langle \gamma_1, \dots, \gamma_m \big\rangle_q^{\Hilb^d(S)}
= \sum_{h \geq 0} \sum_{k \in \BZ} y^k q^{h-1} \blangle \gamma_1, \dots, \gamma_m \brangle^{\Hilb^d(S)}_{\beta_h + k A} \,,
\end{equation*}
where the bracket on the right hand side was defined in \eqref{bbbm}.

Define the 2-point quantum operator
\[ \CE^{\Hilb} : \CF(S) \otimes \BQ((y))((q)) \ra \CF(S) \otimes \BQ((y))((q)) \]
by the following two conditions.
\begin{itemize}
 \item for all homogeneous $a,b \in \CF(S)$,
\[ \big\langle a\ |\ \CE^{\Hilb} b \big\rangle =
\begin{cases}
\big\langle a , b \big\rangle_q & \text{ if } |a| = |b| \\
0 & \text{ else, }
\end{cases}
\]
\item $\CE^{\Hilb}$ is linear over $\BQ((y))((q))$.
\end{itemize}
Since $\Mbar_{0,2}(\Hilb^d(S),\alpha)$ has reduced
virtual dimension $2d$,
the operator $\CE^{\Hilb}$ is self-adjoint of bidegree $(0,0)$.

For $d \geq 0$, consider a divisor class
\[ D \in H^2(\Hilb^d(S)), \]
and the operator of primitive quantum multiplication\footnote{defined in \eqref{vsdfdsjf}} with $D$,
\[ \mathsf{M}_D \colon a \mapsto D \ast a \]
for all $a \in \CF_d(S) \otimes \BQ((y))((q)) \otimes \BQ[\qpar]/\qpar^2$.
If
\begin{equation*} D = D(\gamma) \text{ for some } \gamma \in H^2(S)\quad \text{ or } \quad D = -\frac{1}{2} \Delta_{\Hilb^d(S)}, \label{fomofddf} \end{equation*}
by the divisor axiom we have
\begin{align*}
\mathsf{M}_{D(\gamma)} \big|_{\CF_d(S)}
& =  \Big( L_0(\gamma) + \qpar\, \Fp_{0}(\gamma) \CE^{\Hilb} \Big)\Big|_{\CF_d(S)} \\
-\frac{1}{2} \mathsf{M}_{\Delta_{\Hilb^d(S)}}\big|_{\CF_d(S)}
& = \Big( \partial    + \qpar\, y \frac{d}{d y} \CE^{\Hilb} \Big)\Big|_{\CF_d(S)} \,,
\end{align*}
where $\frac{d}{dy}$ is formal differentiation with respect to the variable $y$,
and $\Fp_0(\gamma)$ for $\gamma \in H^{\ast}(S)$ is the degree $0$ Nakajima operator defined by the following
conditions:\footnote{
The definition precisely matches the action of the extended Heisenberg algebra
$\big\langle\, \Fp_k(\gamma),\ k \in \BZ\, \big\rangle$ on the full Fock space
$\CF(S) \otimes \BQ[ H^{\ast}(S,\BQ) ]$
under the embedding $q^{h-1} \mapsto q^{B + hF}$, see \cite[section 6.1]{KY}.}
\begin{equation} \label{dowofgefg}
\begin{aligned}
\bullet\ \,  & [ \Fp_0(\gamma), \Fp_{m}(\gamma') ] = 0 && \text{ for all } \gamma' \in H^{\ast}(S),\ m \in \BZ, \\
\bullet\ \,  & \Fp_0(\gamma)\, q^{h-1} y^k \, 1_S = \big\langle \gamma, \beta_h \big\rangle q^{h-1} y^k\, 1_S && \text{ for all } h,k.
\end{aligned}
\end{equation}
Since the classes $D(\gamma)$ and $\Delta_{\Hilb^d(S)}$ span $H^2(\Hilb^d(S)$, the operator $\CE^{\Hilb}$ therefore
determines quantum multiplication $\mathsf{M}_D$ for every divisor class $D$.

Let $D_1, D_2 \in H^2(\Hilb^d(S),\BQ)$ be divisor classes. 
By associativity and commutativity of quantum multiplication, we have
\begin{equation} D_1 \ast ( D_2 \ast a ) = D_2 \ast ( D_1 \ast a ) \label{12333} \end{equation}
for all $a \in \CF_d(S)$.
After specializing $D_1$ and $D_2$,
we obtain the main commutator relations for the operator $\CE^{\Hilb}$:

For all $\gamma, \gamma' \in H^2(S,\BQ)$, after restriction to $\CF(S)$, we have
\begin{equation} \label{P4}
\begin{aligned}
\Fp_0(\gamma)\, \big[ \CE^{\Hilb}, L_0(\gamma') \big]  & = \Fp_0(\gamma')\, \big[ \CE^{\Hilb}, L_0(\gamma) \big] \\
\Fp_0(\gamma)\, \big[ \CE^{\Hilb}, \partial \big]  & = y \frac{d}{dy}\, \big[ \CE^{\Hilb}, L_0(\gamma) \big]  \,.
\end{aligned}
\end{equation}

The equalities \eqref{P4} hold
only after restricting to $\CF(S)$.
In both cases, the extension of these equations to $\CF(S) \otimes \BQ((y))((q))$ does \emph{not} hold,
since $\Fp_0(\gamma)$ is not $q$-linear, and $y \frac{d}{dy}$ is not $y$-linear.

Equations \eqref{P4} show that the commutator of $\CE^{\Hilb}$ with a divisor intersection operator 
is essentially independent of the divisor.

\subsection{The operators $\CE^{(r)}$} \label{Main_conjecture_section}
For all $(m,\ell) \in \BZ^2 \setminus \{ 0 \}$ consider fixed formal power series
\begin{equation} \varphi_{m,\ell}(y,q)\, \in \BC((y^{1/2}))[[q]] \label{006} \end{equation}
which satisfy the symmetries
\begin{equation}
\begin{aligned} \label{symmetries_phi}
 \varphi_{m,\ell} & = - \varphi_{-m, -\ell} \\
 \ell \varphi_{m,\ell} & = m \varphi_{\ell, m} \,.
\end{aligned}
\end{equation}
Let $\Delta(q) = q \prod_{m \geq 1} (1-q^m)^{24}$ be the modular discriminant
and let
\[ F(y,q) =  (y^{1/2} + y^{-1/2}) \prod_{m \geq 1} \frac{ (1 + yq^m) (1 + y^{-1}q^m)}{ (1-q^m)^2 } \]
be the Jacobi theta function which appeared in Section \eqref{YZ_section_statement_of_results},
considered as formal power series in $q$ and $y$ in the region $|q| < 1$.

Depending on the functions \eqref{006},
define for all $r \in \BZ$ operators
\begin{equation} \CE^{(r)} : \CF(S) \otimes \BC((y^{1/2}))((q)) \ra \CF(S) \otimes \BC((y^{1/2}))((q)) \label{Def_CEr_operators}\end{equation}
by the following recursion relations:

\vspace{8pt}
\noindent{\bf Relation 1.}
For all $r \geq 0$,
\begin{equation*}
\CE^{(r)} \Big|_{\CF_0(S) \otimes \BC((y^{1/2}))((q))} = \frac{\delta_{0r}}{F(y,q)^2 \Delta(q)} \cdot {\rm id}_{\CF_0(S) \otimes \BC((y^{1/2}))((q))},
\end{equation*}

\vspace{8pt}
\noindent{\bf Relation 2.}
For all $m \neq 0$, $r \in \BZ$ and homogeneous $\gamma \in H^{\ast}(S)$,
\[ [ \Fp_{m}(\gamma), \CE^{(r)} ] = \sum_{\ell \in \BZ} \frac{\ell^{k(\gamma)}}{m^{k(\gamma)}} : \Fp_\ell(\gamma) \CE^{(r+m-\ell)} : \, \varphi_{m,\ell}(y,q), \]
where $k(\gamma)$ denotes the shifted complex cohomological degree of $\gamma$,
\[ \gamma \in H^{2(k(\gamma) + 1)}(S;\BQ) \,, \]
and $: -- :$ is a variant of the normal ordered product defined by
\[ : \Fp_\ell(\gamma) \CE^{(k)} : = \begin{cases}
                             \Fp_{\ell}(\gamma) \CE^{(k)} & \text{ if } \ell \leq 0 \\
                             \CE^{(k)} \Fp_{\ell}(\gamma) & \text{ if } \ell > 0 \,.
                            \end{cases}
\]
\vspace{6pt}

By definition, the operator $\CE^{(r)}$ is homogeneous of bidegree $(-r,0)$; it is $y$-linear, but \emph{not} $q$ linear.

\begin{lemma} The operators $\CE^{(r)}, r \in \BZ$ are well-defined. \end{lemma}
\begin{proof}
By induction, Relation 1 and 2 uniquely determine the operators $\CE^{(r)}$.
It remains to show that the Nakajima commutator relations \eqref{N1} are preserved by $\CE^{(r)}$.
Hence, we need to show
\[ \Big[ \big[ \Fp_m(\alpha), \Fp_n(\beta) \big] , \CE^{(r)} \Big]
 = \big[ -m \delta_{m + n,0} \langle \alpha, \beta \rangle\, {\rm id}_{\CF(S)}, \CE^{(r)} \big] = 0
\]
for all homogeneous $\alpha, \beta \in H^{\ast}(S)$ and all $m, n \in \BZ \setminus \{ 0 \}$. We have
\begin{equation} \Big[ \big[ \Fp_m(\alpha), \Fp_n(\beta) \big] , \CE^{(r)} \Big]
 = \Big[ \Fp_m(\alpha), \big[ \Fp_n(\beta), \CE^{(r)} \big] \Big]
- \Big[ \Fp_n(\beta) , \big[ \Fp_m(\alpha) , \CE^{(r)} \big] \Big].
\label{Muasfdagd} \end{equation}
Using Relation~2, we obtain
\begin{align}
\notag & \Big[ \Fp_m(\alpha), \big[ \Fp_n(\beta), \CE^{(r)} \big] \Big] \\
\notag =\ &
\Big[ \Fp_m(\alpha),\ \sum_{\ell \in \BZ} \frac{\ell^{k(\beta)}}{n^{k(\beta)}} : \Fp_\ell(\beta) \CE^{(r+m-\ell)} : \, \varphi_{m,\ell}(y,q) \Big] \\
\label{fovmdfd} =\ & \frac{(-m)^{k(\beta) + 1}}{n^{k(\beta)}} \langle \alpha, \beta \rangle \CE^{(r+n+m)} \varphi_{n,-m} \\
\notag & \quad + \sum_{\ell, \ell' \in \BZ} \frac{ \ell^{k(\beta)} (\ell')^{k(\alpha)} }{ n^{k(\beta)} m^{k(\alpha)} }
: \Fp_\ell(\beta) \bigl( : \Fp_{\ell'}(\alpha) \CE^{(r+n+m-\ell- \ell')} : \bigr)\! :\varphi_{m,\ell'} \varphi_{n,\ell}.
\end{align}
Similarly, we have
\begin{multline} \label{equation2_epsr_check}
 \Big[ \Fp_n(\beta) , \big[ \Fp_m(\alpha) , \CE^{(r)} \big] \Big]
= \frac{(-n)^{k(\alpha) + 1}}{m^{k(\alpha)}} \langle \alpha, \beta \rangle \CE^{(r+n+m)} \varphi_{m,-n} \\
+ \sum_{\ell, \ell' \in \BZ} \frac{ \ell^{k(\beta)} (\ell')^{k(\alpha)} }{ n^{k(\beta)} m^{k(\alpha)} }
: \Fp_{\ell'}(\alpha) \bigl( : \Fp_{\ell}(\beta) \CE^{(r+n+m-\ell- \ell')} : \bigr)\! :\varphi_{m,\ell'} \varphi_{n,\ell}.
\end{multline}
Since for all $\ell, \ell' \in \BZ$ we have
\[
: \Fp_\ell(\beta) \bigl( : \Fp_{\ell'}(\alpha) \CE^{(r+n+m-\ell- \ell')} : \bigr)\! :
\ \, =\  \,
: \Fp_{\ell'}(\alpha) \bigl( : \Fp_{\ell}(\beta) \CE^{(r+n+m-\ell- \ell')} : \bigr)\! :
\]
the second terms
in \eqref{fovmdfd} and \eqref{equation2_epsr_check}
agree. Hence, \eqref{Muasfdagd} equals
\begin{equation}
\langle \alpha, \beta \rangle \CE^{(r+m+n)} 
\bigg\{ \frac{(-m)^{k(\beta) + 1}}{n^{k(\beta)}} \varphi_{n,-m}
- \frac{(-n)^{k(\alpha) + 1}}{m^{k(\alpha)}} \varphi_{m,-n}  \bigg\}
\label{hovsmdovsv}
\end{equation}
If $\langle \alpha, \beta \rangle = 0$ we are done, hence we may assume otherwise.
Then, for degree reasons, $k(\alpha) = - k(\beta)$.
Using the symmetries \eqref{symmetries_phi}, we find
\[ \varphi_{m,-n} = - \frac{m}{n} \varphi_{-n,m} = \frac{m}{n} \varphi_{n,-m} \]
Inserting both equations into \eqref{hovsmdovsv}, this yields
\[
\langle \alpha, \beta \rangle \CE^{(r+m+n)} \varphi_{n,-m} \bigg\{
- \frac{m^{-k(\alpha) + 1}}{n^{-k(\alpha)}}
\ +\  \frac{n^{k(\alpha) + 1}}{m^{k(\alpha)}} \cdot \frac{m}{n}  
\bigg\} = 0. \qedhere
\]
\end{proof}


\subsection{Conjectures} \label{Main_conjecture_section2}
Let $G(y,q)$ be the formal expansion in the variables $y,q$ of the function $G(z,\tau)$
which already appeared in Section \ref{Section_More_evaluations_Introduction},
\begin{align*}
G(y,q) & = F(y,q)^2 \left( y \frac{d}{dy} \right)^2 \log( F(y,q) ) \\
& = F(y,q)^2 \cdot \bigg\{ \frac{y}{(1+y)^2} - \sum_{d \geq 1} \sum_{m | d} m \big( (-y)^{-m} + (-y)^m \big) q^d \bigg\}.
\end{align*}

\vspace{8pt}
\noindent{\bf Conjecture A.}
{\em There exist unique series $\varphi_{m,\ell}$ for $(m,\ell) \in \BZ^2 \setminus \{ 0\}$ such that
the following hold:
\begin{enumerate}
\item[(i)] the symmetries \eqref{symmetries_phi} are satisfied,
\item[(ii)] the initial conditions
\[ \varphi_{1,1} = G(y,q) - 1, \quad \varphi_{1,0} = -i \cdot F(y,q), \quad \varphi_{1,-1} = - \frac{1}{2} \, q \frac{d}{dq}\big( F(y,q)^2 \big) \,, \]
hold, where $i=\sqrt{-1}$ is the imaginary number,
\item[(iii)] Let $\CE^{(r)}, r \in \BZ$ be the operators \eqref{Def_CEr_operators} defined by the functions $\varphi_{m,\ell}$.
Then, $\CE^{(0)}$ satisfies after restriction to $\CF(S)$ the WDVV equations
\begin{equation} \label{WDVV_eqn_for_CEr}
\begin{aligned}
\Fp_0(\gamma)\, [ \CE^{(0)}, L_0(\gamma') ] & = \Fp_0(\gamma')\, [ \CE^{(0)}, L_0(\gamma) ] \\
\Fp_0(\gamma)\, [ \CE^{(0)}, \partial ]     & = y \frac{d}{dy}\, [ \CE^{(0)}, L_0(\gamma) ]
\end{aligned}
\end{equation}
for all $\gamma, \gamma' \in H^2(S,\BQ)$.
\end{enumerate}
}

\vspace{8pt}
Conjecture A is a purely algebraic, non-degeneracy statement for the WDVV equations \eqref{WDVV_eqn_for_CEr}.
It has been checked numerically on $\CF_d(S)$ for all $d \leq 5$.
The first values of the series $\varphi_{m,\ell}$ are 
given in Appendix \ref{Appendix_Numerical_Values}.
For the remainder of Section \ref{Section_Quantum_Cohomology}, we \emph{assume} conjecture A to be true,
and we let $\CE^{(r)}$ denote the operators defined by the (hence unique) functions $\varphi_{m,\ell}$ satisfying (i)-(iii) above.
Since Conjecture A has been shown to be true for $\CF_d(S)$ for all $d \leq 5$,
the restriction of $\CE^{(0)}$ to the subspace $\oplus_{d \leq 5} \CF_d(S)$ is well-defined unconditionally.

The following conjecture relates $\CE^{(0)}$ to the quantum operator $\CE^{\Hilb}$.
Let $L_0$ be the energy operator \eqref{energy_operator}.
Define the operator
\[ G(y,q)^{L_0} : \CF(S) \otimes \BQ((y))((q)) \ra \CF(S) \otimes \BQ((y))((q)) \]
by the assignment
\[ G(y,q)^{L_0}( \mu ) = G(y,q)^{|\mu|} \cdot \mu \]
for any homogeneous $\mu \in \CF(S)$.

\vspace{8pt}
\noindent{\bf Conjecture B.} After restriction to $\CF(S)$,
\begin{equation} \CE^{\Hilb} \ = \ \CE^{(0)} - \frac{1}{F(y,q)^2 \Delta(q)} G(y,q)^{L_0} \, . \label{MMMDCDVD} \end{equation}
\vspace{4pt}

Combining Conjectures A and B we obtain an algorithmic procedure
to determine the 2-point quantum bracket $\langle \cdot , \cdot \rangle_q$.
The equality of Conjecture B is conjectured to hold only after restriction to $\CF(S)$.
The extension of \eqref{MMMDCDVD} to $\CF(S) \otimes \BQ((y))((q))$ is clearly false:
The operators $\CE^{\Hilb}$ and $G^{L_0}/(F^2 \Delta)$ are $q$-linear by definition, but $\CE^{(0)}$ is not.

Let $\QJac$ be the ring of holomorphic quasi-Jacoi forms defined in Appendix~\ref{Appendix_Quasi_Jacobi_Forms},
and let
\[ \QJac = \bigoplus_{m \geq 0} \bigoplus_{k \geq -2m} \QJac_{k,m} \]
be the natural bigrading of $\QJac$ by index $m$ and weight $k$,
where $m$ runs over all non-negative half-integers $\frac{1}{2} \BZ^{\geq 0}$.

\vspace{8pt}
\noindent{\bf Conjecture C.}
{\em For every $(m,\ell) \in \BZ^2 \setminus \{ 0 \}$, the series
\[ \varphi_{m,\ell} + {\rm sgn}(m) \delta_{m \ell} \]
is a quasi-Jacobi form of index $\frac{1}{2} (|m|+|\ell| )$ and weight $-\delta_{0\ell}$.}
\vspace{8pt}

Define a new degree functions $\underline{\deg}(\gamma)$ on $H^{\ast}(S)$ by the assignment
\begin{itemize}
 \item $\gamma \in \BQ F \mapsto \underline{\deg}(\gamma) = -1$
 \item $\gamma \in \BQ (B+F) \mapsto \underline{\deg}(\gamma) = 1$
 \item $\gamma \in \{ F,\, B+F \}^{\perp} \mapsto \underline{\deg}(\gamma) = 0$,
\end{itemize}
where the orthogonal complement $\{ F,\, B+F \}^{\perp}$ is defined with respect to the inner product $\langle \cdot , \cdot \rangle$.

\begin{lemmastar} \label{oiejgojeogerg}
Assume Conjectures A and C hold.
Let $\gamma_i, \widetilde{\gamma}_i \in H^{\ast}(S)$ be $\underline{{\rm deg}}$-homogeneous classes,
and let
\begin{equation} \label{munu_coh_classes}
\mu = \prod_{i} \Fp_{-m_i}(\gamma_i) 1_S, \quad \quad \nu = \prod_{j} \Fp_{-n_j}(\widetilde{\gamma}_j) 1_S
\end{equation}
be cohomology classes of $\Hilb^m(S)$ and $\Hilb^n(S)$ respectively. Then
\[ \Big\langle \mu\ \Big|\ \CE^{(n-m)} \nu \Big\rangle = \frac{\Phi}{F(y,q)^2 \Delta(q)} \]
for a quasi-Jacobi form $\Phi \in \QJac$ of index $\frac{1}{2} (|m|+|n|)$ and weight
\[ \sum_{i} \underline{\deg}(\gamma_i) + \sum_j \underline{\deg}(\gamma'_j). \]
\end{lemmastar}

\vspace{5pt}
\begin{proof}[Proof of Lemma \ref{oiejgojeogerg}] By a straight-forward induction on $|\mu| + |\nu|$. \end{proof}

Let $\mu, \nu \in H^{\ast}(\Hilb^d(S))$. By Lemma \eqref{oiejgojeogerg} and Conjecture B, we have
\begin{equation} \langle \mu, \nu \rangle_q = \frac{\varphi(y,q)}{F(y,q)^2 \Delta(q)} \label{xxxzzz} \end{equation}
for a quasi-Jacobi form $\varphi(y,q)$. Since $F(y,q)$ has a simple zero at $z=0$,
we expect the function \eqref{xxxzzz} to have a pole of order $2$ at $z=0$.
Numerical experiments (Conjecture~J) or deformation invariance\footnote{See \cite{thesis} for a discussion
of the monodromy action by deformations of $\Hilb^d(S)$ in the moduli space of irreducible holomorphic-symplectic varieties.}
suggest that the series $\langle \mu, \nu \rangle_q$ is nonetheless holomorphic at $z=0$.
Combining everything, we obtain the following prediction.

\begin{lemmastar} \label{index_weight_lemma_jacforms} Assume Conjectures A,\, B,\, C,\, J \, hold.
Let $\mu, \nu \in H^{\ast}(\Hilb^d(S))$ be cohomology classes of the form \eqref{munu_coh_classes}.
Then,
\[ \big\langle \mu , \nu \big\rangle^{\Hilb^d(S)}_q = \frac{\Phi(y,q)}{\Delta(q)} \]
for a quasi-Jacobi form $\Phi(y,q)$ of index $d-1$ and weight
\[ 2 + \sum_{i} \underline{\deg}(\gamma_i) + \sum_j \underline{\deg}(\gamma'_j). \]
\end{lemmastar}

\subsection{Examples} \label{Main_conjecture_examples_section}
\subsubsection{The higher-dimensional Yau-Zaslow formula} \label{Examples_higher_dim_YZ}
\noindent \textbf{(i)} Let $F$ be the fiber of the elliptic fibration $\pi : S \to \p^1$. Then
{\allowdisplaybreaks
\begin{align*}
& \blangle \Fp_{-1}(F)^d 1_S\ \Big|\ \Big( \CE^{(0)} - \frac{1}{F^2 \Delta} G^{L_0} \Big) \Fp_{-1}(F)^d 1_S \brangle \\[4pt]
= & \blangle \Fp_{-1}(F)^d 1_S\ \Big|\ \CE^{(0)} \Fp_{-1}(F)^d 1_S \brangle \\
=\ &  (-1)^d \blangle 1_S\ \Big|\ \Fp_{1}(F)^d \CE^{(0)} \Fp_{-1}(F)^d 1_S \brangle \\
=\ &  (-1)^d \blangle 1_S\ \Big|\ \Fp_0(F)^d \CE^{(d)} \varphi_{1,0}^d \Fp_{-1}(F)^d 1_S \brangle \\
=\ &  (-1)^d \blangle 1_S\ \Big|\ \Fp_0(F)^{2d} \CE^{(0)} (-1)^d \varphi_{1,0}^{d} \varphi_{-1,0}^d 1_S \brangle \\
=\ &  \frac{\varphi_{1,0}^{d} \varphi_{-1,0}^d}{F(y,q)^2 \Delta(q)} \\
=\ &  \frac{F(y,q)^{2d-2}}{\Delta(q)}
\end{align*}
shows Conjecture B to be in agreement with Theorem \ref{MThm0}; here we have used $\Fp_0(F) = 1$ above.
}

\vspace{8pt}
\noindent \textbf{(ii)} Let $B$ be the class of the section of $\pi : S \to \p^1$ and consider the class
\[ W = B + F. \]
We have $\langle W, W \rangle = 0$ and $\langle W, \beta_h \rangle = h-1$.
Hence, $\Fp_0(W)$ acts as $q \frac{d}{dq}$ on functions in $q$.
We have
{\allowdisplaybreaks
\begin{align*}
& \blangle \Fp_{-1}(W)^d 1_S\ \Big|\ \Bigl( \CE^{(0)} - \frac{1}{F^2 \Delta} G^{L_0} \Bigr) \Fp_{-1}(W)^d 1_S \brangle \\[4pt]
=\ & \blangle \Fp_{-1}(W)^d 1_S\ \Big|\ \CE^{(0)} \Fp_{-1}(W)^d 1_S \brangle \\
=\ &  (-1)^d \blangle 1_S\ \Big|\ \Fp_0(W)^d \CE^{(d)} \varphi_{1,0}^d \Fp_{-1}(W)^d 1_S \brangle \\
=\ &  \blangle 1_S\ \Big|\ \Fp_0(W)^{2d} \CE^{(0)} \varphi_{1,0}^{d} \varphi_{-1,0}^d 1_S \brangle \\
=\ &  \left( q \frac{d}{dq} \right)^{2d} \left( \frac{\varphi_{1,0}^{d} \varphi_{-1,0}^d}{F(y,q)^2 \Delta(q)} \right)\\
=\ &  \left( q \frac{d}{dq} \right)^{2d} \left( \frac{F(y,q)^{2d-2}}{\Delta(q)} \right).
\end{align*}
}
\subsubsection{Further Gromov-Witten invariants} \label{Section_Examples_More_evaluations}
\vspace{8pt}
\noindent \textbf{(i)} Let $\pt \in H^4(S;\BZ)$ be the class of a point. For $d \geq 1$, let
\[ C(F) = \Fp_{-1}(F) \Fp_{-1}(\pt)^{d-1} 1_S \in H_2(\Hilb^2(S),\BZ) \]
and
\[ D(F) = \Fp_{-1}(F) \Fp_{-1}(e)^{d-1} 1_S \in H^2(\Hilb^2(S),\BZ) \,. \]
Then,
{\allowdisplaybreaks
\begin{align*}
& \blangle C(F) \ \Big|\ \bigl( \CE^{(0)} - \frac{1}{F^2 \Delta} G^{L_0} \bigr) D(F) \brangle \\[4pt]
& = \frac{1}{(d-1)!} \blangle \Fp_{-1}(F) \Fp_{-1}(\pt)^{d-1} 1_S \ \Big|\ \CE^{(0)} \Fp_{-1}(F) \Fp_{-1}(e)^{d-1} 1_S \brangle \\
& = \frac{1}{(d-1)!}  \blangle \Fp_{-1}(\pt)^{d-1} 1_S \ \Big|\ \CE^{(0)} \varphi_{1,0} \varphi_{-1,0} \Fp_{-1}(e)^{d-1} 1_S \brangle \\
& = \frac{(-1)^{d-1}}{(d-1)!}  \blangle 1_S \ \Big|\ \CE^{(0)} \varphi_{1,0} \varphi_{-1,0} (\varphi_{1,1} + 1)^{d-1} \Fp_{1}(\pt)^{d-1} \Fp_{-1}(e)^{d-1} 1_S \brangle \\
& = \frac{ \varphi_{1,0} \varphi_{-1,0} (\varphi_{1,1} + 1)^{d-1} }{F(y,q)^2 \Delta(q)} \\
& = \frac{ G(y,q)^{d-1} }{\Delta(q)}.
\end{align*}
By the divisor equation and $\langle D(F) , \beta_h + kA \rangle = 1$ for all $h,k$,
Conjecture B is in full agreement with Theorem \ref{ellthm2} equation 1.
}

\vspace{8pt}
\noindent \textbf{(ii)} Let $A = \Fp_{-2}(\omega) \Fp_{-1}(\omega)^{d-2} 1_S$ be the class of an exceptional curve. Then,
\begin{align*}
& \blangle A\ \Big|\ \bigl( \CE^{(0)} - \frac{1}{F^2 \Delta} G^{L_0} \bigr) D(F) \brangle \\[4pt]
& = \frac{(-1)^d}{(d-1)!} \blangle 1_S \ \Big|\ \Fp_{2}(\omega) \CE^{(0)} \Fp_1(\omega)^{d-2} \Fp_{-1}(F) \Fp_{-1}(e)^{d-1} (\varphi_{1,1} + 1)^{d-2} 1_S \brangle \\
& = \frac{(-1)^d}{(d-1)!} \blangle 1_S \ \Big|\ \frac{1}{2} \CE^{(1)} \Fp_1(\omega)^{d-1} \Fp_{-1}(F) \Fp_{-1}(e)^{d-1} \varphi_{2,1}  (\varphi_{1,1} + 1)^{d-2} 1_S \brangle \\
& = - \frac{1}{2} \blangle 1_S \ \Big|\ \CE^{(1)} \Fp_{-1}(F) \varphi_{2,1}  (\varphi_{1,1} + 1)^{d-2} \brangle \\
& = - \frac{1}{2} \frac{ (-\varphi_{-1,0}) \varphi_{2,1} (\varphi_{1,1} + 1)^{d-2}}{F^2(y,q) \Delta} \\
& = - \frac{1}{2} \frac{ \Big( y \frac{d}{dy} G \Big) \cdot G^{d-2} }{\Delta}.
\end{align*}
Hence, again, Conjecture B is in full agreement with Theorem \ref{ellthm2} equation 2.

\vspace{8pt}
\noindent \textbf{(iii)} For a point $P \in S$, the incidence subscheme
\[ I(P) = \{ \xi \in \Hilb^2(S) \ |\ P \in \xi \} \]
has class $[I(P)] = \Fp_{-1}(\omega) \Fp_{-1}(e) 1_S$.
We calculate
\begin{align*}
& \blangle I(P)\ \Big|\ \bigl( \CE^{(0)} - \frac{1}{F^2 \Delta} G^{L_0} \bigr) I(P) \brangle \\[4pt]
& = - \blangle \Fp_{-1}(e) 1_S \ \Big|\ \Fp_1(\omega) \CE^{(0)} I(P) \brangle - \frac{G^2}{F^2 \Delta} \\
& = - \blangle \Fp_{-1}(e) 1_S \ \Big|\ \Big( \CE^{(0)} \Fp_{1}(\omega) (\varphi_{1,1} + 1) - \Fp_{-1}(\omega) \CE^{(2)} \varphi_{1,-1} \Big) I(P) \brangle - \frac{G^2}{F^2 \Delta} \\
& = \blangle \Fp_{-1}(e) 1_S\ \Big|\ \CE^{(0)} \Fp_{-1}(\omega) (\varphi_{1,1} + 1) 1_S \brangle \\
& \quad \quad \quad + \blangle 1_S\ \Big|\ \CE^{(2)} \Fp_{-1}(\omega) \Fp_{-1}(e) \varphi_{1,-1} 1_S \brangle - \frac{G^2}{F^2 \Delta} \\
& = \frac{ (\varphi_{1,1} + 1)^2 }{ F^2 \Delta } +  \frac{- \varphi_{-1,1} \varphi_{1,-1} }{ F^2 \Delta} - \frac{G^2}{F^2 \Delta} \\
& = \frac{ \Big( q \frac{d}{dq} F \Big)^2 }{\Delta(q)} .
\end{align*}
Hence, Conjecture B agrees with Theorem \ref{ellthm2} equation 3, case $d=2$.

\vspace{8pt}
\noindent \textbf{(iv)} For a point $P \in S$, we have
\begin{align*}
& \blangle \Fp_{-1}(F)^2 1_S \ \Big|\ \bigl( \CE^{(0)} - \frac{1}{F^2 \Delta} G^{L_0} \bigr) I(P) \brangle \\[4pt]
& = - \blangle 1_S \ \Big|\ \CE^{(2)} \varphi_{1,0}^2 I(P) \brangle \\
& = \frac{- \varphi_{1,0}^2 \varphi_{-1,1}}{F^2 \Delta} \\
& = \frac{F(y,q) \cdot q \frac{d}{dq} F(y,q) }{ \Delta(\tau) } .
\end{align*}
Hence, Conjecture B is in agreement with Theorem \ref{extra_eval}.

\subsubsection{The Hilbert scheme of $2$ points} \label{Section_Conj_in_Hilb2}
We check conjectures A, B, C, J in the case $\Hilb^2(S)$.
Conjecture A is seen to hold for $\Hilb^2(S)$ by direct calculation.
The corresponding functions $\varphi_{m,\ell}$ are given in Appendix \ref{Appendix_Numerical_Values}.
This implies Conjecture C by inspection.
Conjecture B and J hold by the following result.

\begin{thm} \label{ConjC_for_Hilb2} \label{Theorem_for_Hilb2} For all $\mu, \nu \in H^{\ast}(\Hilb^2(S))$,
\[ \langle \mu, \nu \rangle_q = \Big\langle\, \mu\ \Big|\ \Big( \CE^{(0)} - \frac{G^{L_0}}{F^2 \Delta} \Big)\, \nu\, \Big\rangle. \]
\end{thm}

\begin{thm} \label{ConjJ_for_Hilb2} Let $\mu, \nu \in H^{\ast}(\Hilb^2(S)$ be cohomology classes of the form \eqref{munu_coh_classes}.
Then,
\[ \big\langle \mu , \nu \big\rangle_q = \frac{\Phi}{\Delta(q)} \]
for a quasi-Jacobi form $\Phi$ of index $1$ and weight
\[ 2 + \sum_{i} \underline{\deg}(\gamma_i) + \sum_j \underline{\deg}(\gamma'_j). \]
\end{thm}

By Sections \ref{Examples_higher_dim_YZ} and \ref{Section_Examples_More_evaluations} above,
Theorem \ref{ConjC_for_Hilb2} holds in the cases considered in Theorems \ref{ellthm}, \ref{ellthm2} and \ref{extra_eval} respectively.
Applying the WDVV equation \eqref{dowofgefg}
successively to these base cases, one evaluates the bracket
$\langle \mu, \nu \rangle_q$ for all $\mu, \nu \in H^{\ast}(\Hilb^2(S))$
in finitely many steps.
Since for $\Hilb^2(S)$ the WDVV equation also holds for
$\CE^{(0)} - G^{L_0}/(F^2 \Delta)$,
this implies Theorem \ref{ConjC_for_Hilb2}.
Finally, Theorem \ref{ConjJ_for_Hilb2} follows now from direct inspection.

\begin{proof}[Proof of Proposition \ref{rgergegssa}]
By degenerating $(E,0)$ to the nodal elliptic curve and using the divisor equation we may rewrite $\mathsf{H}_d(y,q)$ as
\begin{equation} \label{aaasd}
 \mathsf{H}_d(y,q) = 
\sum_{k \in \BZ} \sum_{h\geq 0}
y^k q^{h-1} 
\int_{[ \overline{M}_{0,2}(\Hilb^d(S), \beta_h + kA) ]^{\text{red}}}  
(\ev_1 \times \ev_2)^{\ast} [\Delta^{[d]}] 
\end{equation}
where $[\Delta^{[d]}] \in H^{2d}( \Hilb^d(S) \times \Hilb^d(S))$ is the diagonal class.
Proposition now follows from calculating the right hand side of \eqref{aaasd} using Theorem~\ref{Theorem_for_Hilb2}.
\end{proof}

\subsubsection{The $\CA_1$ resolution.}
Let $[q^{-1}]$ be the operator that extracts the $q^{-1}$ coefficient, and let
\[ \CE_B^{\Hilb} = [q^{-1}] \CE^{\Hilb} \]
be the restriction of $\CE^{\Hilb}$ to the case of the section class $B$.
The corresponding local case was considered before in \cite{MO, MO2}.

Define operators $\CE_{B}^{(r)}$ by the relations
\begin{itemize}
\item $\big\langle 1_S\ \big|\ \CE_{B}^{(r)} 1_S \big\rangle = \frac{y}{(1+y)^2} \delta_{0r}$
\item $\big[ \Fp_{m}(\gamma), \CE_B^{(r)} \big] = \langle \gamma, B \rangle  \left( (-y)^{-m/2} - (-y)^{m/2} \right) \CE_B^{(r+m)}$
\end{itemize}
for all $m \neq 0$ and all $\gamma \in H^{\ast}(S)$, see \cite[Section 5.1]{MO2}.
Translating the results of \cite{MO,MO2} to the 
$K3$ surface leads to the following evaluation.

\begin{thm}[Maulik, Oblomkov] \label{8765} After restriction to $\CF(S)$,
\[ \CE_B^{\Hilb} + \frac{y}{(1+y)^2} {\rm id_{\CF(S)}} = \CE_{B}^{(0)} \,. \]
\end{thm}

\noindent By the numerical values of Appendix \ref{Appendix_Numerical_Values}, we expect the expansions
\begin{alignat*}{2}
 \varphi_{m,0} & = \big( (-y)^{-m/2} - (-y)^{m/2} \big) + O(q) \quad \quad && \text{ for all } m \neq 0 \\
 \varphi_{m,\ell} & = O(q) && \text{ for all } m \in \BZ, \ell \neq 0 \,.
\end{alignat*}
Because of
\[ [q^{-1}] \frac{G^{L_0}}{F^2 \Delta} = \frac{y}{(1+y)^2} \text{id}_{\CF(S)}\, , \]
we find conjectures A and B in agreement with Theorem \ref{8765}.

\appendix
\section{The reduced WDVV equation} \label{section_reducedWDVV}
Let $\Mbar_{0,4}$ be the moduli space of stable genus $0$ curves with $4$ marked points.
The boundary of $\Mbar_{0,4}$ is the union of the divisors
\begin{equation} D(12|34),\ D(14|23),\ D(13|24) \label{530} \end{equation}
corresponding to a broken curve with the respective prescribed splitting of the marked points.
Since $\Mbar_{0,4}$ is isomorphic to $\p^1$, any two of the divisors \eqref{530} are rationally equivalent.

Let $Y$ be a smooth projective variety and let $\Mbar_{0,n}(Y,\beta)$
be the moduli space of stable maps to $Y$ of genus $0$ and class $\beta$. Let 
\[ \pi : \Mbar_{0,n}(Y,\beta) \to \Mbar_{0,4} \]
be the map that forgets all but the last four points.
The pullback of the boundary divisors \eqref{530} under $\pi$ defines
rationally equivalent divisors on $\Mbar_{0,n}(Y,\beta)$.
The intersection of these divisors with curve classes obtained from the virtual class yields
relations among Gromov-Witten invariants of $Y$,
the WDVV equations \cite{FP}.
We derive the precise form of these equations
for reduced Gromov-Witten theory.
For simplicity, we restrict to the case $n = 4$.

Let $Y$ be a holomorphic symplectic variety and let
\[ \blangle \gamma_1, \ldots, \gamma_n \brangle_\beta^{\text{red}} = \int_{[ \Mbar_{0,n}(Y,\beta) ]^{\text{red}} } \ev_1^{\ast}(\gamma_1) \cup \cdots \cup \ev_n^{\ast}(\gamma_n) \]
denote the \emph{reduced} Gromov-Witten invariants
of $Y$ of genus $0$ and class $\beta \in H_2(Y;\BZ)$
with primary insertions $\gamma_1, \cdots, \gamma_m \in H^{\ast}(Y)$.

\begin{prop}\label{WDVVprop}
Let $\gamma_1, \dots, \gamma_4 \in H^{2 \ast}(Y;\BQ)$ be cohomology classes with
\[ \sum_i \deg(\gamma_i) = \mathop{\rm vdim}\nolimits \Mbar_{0,4}(Y,\beta) - 1 = \dim Y + 1, \]
where $\deg(\gamma_i)$ denotes the complex degree of $\gamma_i$. Then,
\begin{equation*}
\blangle \gamma_1, \gamma_2, \gamma_3 \cup \gamma_4 \brangle^{\text{red}}_{\beta} + \blangle \gamma_1 \cup \gamma_2, \gamma_3, \gamma_4 \brangle^{\text{red}}_{\beta}
=
\blangle \gamma_1, \gamma_4, \gamma_2 \cup \gamma_3 \brangle^{\text{red}}_{\beta} + \blangle \gamma_1 \cup \gamma_4, \gamma_2, \gamma_3 \brangle^{\text{red}}_{\beta}.
\end{equation*}
\end{prop}
\begin{proof}
Consider the fiber of $\pi$ over $D(12|34)$,
\[ D = \pi^{-1}( D(12|34) ) \,. \]
The intersection of $D$ with the class
\begin{equation} \left( \prod_{i = 1}^{4} \ev_i^{\ast}(\gamma_i) \right) \cap [\Mbar_{0,4}(Y,\beta)]^{\text{red}}. \label{wdvv_kcmvknvcv} \end{equation}
splits into a sum of integrals over the product 
\[ M' = \Mbar_{0,3}(Y,\beta_1) \times \Mbar_{0,3}(Y, \beta_2), \]
for all effective decompositions $\beta = \beta_1 + \beta_2$.

The reduced virtual class $[\Mbar_{0,4}(Y,\beta)]^{\text{red}}$ restricts to $M'$ as the sum of
\[ (\ev_3 \times \ev_3)^{\ast} \Delta_Y \cap [\Mbar_{0,3}(Y, \beta_1)]^{\text{red}} \times [\Mbar_{0,3}(Y,\beta_2)]^{\text{ord}} \]
with the same term, except for 'red' and 'red' interchanged;
here \[ \Delta_Y \in H^{2 \dim Y}(Y \times Y;\BZ) \]
is the class of the diagonal
and $[\ \cdot \ ]^{\text{ord}}$ denotes the ordinary virtual class.

Since $[\Mbar_{0,3}(Y,\beta)]^{\text{ord}} = 0$ unless $\beta = 0$, we find
\begin{align}
 \notag \int_{[ \Mbar_{0,4}(Y,\beta) ]^{\text{red}} } D \cup \prod_{i} \gamma_i & =  
\sum_{e,f} \blangle \gamma_1, \gamma_2, T_e \brangle^{\text{red}}_{\beta} g^{ef} \blangle \gamma_3, \gamma_4, T_f \brangle^{\text{ord}}_{0} + \\
 \notag & \quad \quad + \blangle \gamma_1, \gamma_2, T_e \brangle^{\text{ord}}_{0} g^{ef} \blangle \gamma_3, \gamma_4, T_f \brangle^{\text{red}}_{\beta} \\
& = \blangle \gamma_1, \gamma_2, \gamma_3 \cup \gamma_4 \brangle^{\text{red}}_{\beta} + \blangle \gamma_1 \cup \gamma_2, \gamma_3, \gamma_4 \brangle^{\text{red}}_{\beta},
\label{ifijvofj}
\end{align}
where $\{ T_e \}_e$ is a basis of $H^{\ast}(Y;\BZ)$ and $(g^{ef})_{e,f}$ is the inverse of the intersection matrix
$g_{ef} =\int_Y T_e\cup T_f$.

After comparing \eqref{ifijvofj} with the integral of \eqref{wdvv_kcmvknvcv}
over the pullback of $D(14|23)$, the proof of Proposition \ref{WDVVprop} is complete.
\end{proof}

We may use the previous proposition to define reduced quantum cohomology. 
Let $\qpar$ be a formal parameter with $\qpar^2 = 0$.
Let ${\rm Eff}_{Y}$ be the cone of effective curve class on $Y$, and for any $\beta \in {\rm Eff}_{Y}$ let
$q^{\beta}$ be the corresponding element in the semi-group algebra $\BQ[ {\rm Eff}_{Y} ]$.
Define the \emph{reduced} quantum product $\ast$ on
\[ H^{\ast}(Y;\BQ) \otimes \BQ[[ {\rm Eff}_{Y} ]] \otimes \BQ[\qpar]/\qpar^2 \,. \]
by
\[
\big\langle \gamma_1 \ast \gamma_2, \gamma_3 \big\rangle
= \big\langle \gamma_1 \cup \gamma_2, \gamma_3 \big\rangle +
\qpar \sum_{\beta > 0} q^{\beta} \blangle \gamma_1, \gamma_2, \gamma_3 \brangle_\beta^{\text{red}}
\]
for all $a,b,c \in H^{\ast}(Y)$, where $\langle \gamma_1, \gamma_2 \rangle = \int_Y \gamma_1 \cup \gamma_2$ is the standard inner product on $H^{\ast}(Y;\BQ)$
and $\beta$ runs over all non-zero effective curve classes of $Y$.
Then, Proposition~\ref{WDVVprop} implies that $\ast$ is associative.

\section{Quasi-Jacobi forms} \label{Appendix_Quasi_Jacobi_Forms}
\subsubsection{Definition}
Let $(z,\tau) \in \BC \times \BH$, and let $y = - p = - e^{2 \pi i z}$ and $q = e^{2 \pi i \tau}$.
For all expansions below, we will work in the region $|y| < 1$.

Consider the Jacobi theta functions
\[ F(z,\tau) = \frac{\vartheta_1(z,\tau)}{\eta^3(\tau)}  = (y^{1/2} + y^{-1/2}) \prod_{m \geq 1} \frac{ (1 + yq^m) (1 + y^{-1}q^m)}{ (1-q^m)^2 }, \]
the logarithmic derivative
\[ J_1(z,\tau) = y \frac{d}{dy} \log( F(y,q) ) = 
 \frac{y}{1 + y} - \frac{1}{2} - \sum_{d \geq 1} \sum_{m|d} \big( (-y)^m - (-y)^{-m} \big) q^{d},
\]
the Weierstrass elliptic function
\begin{equation}
\wp(z,\tau) = \frac{1}{12} - \frac{y}{(1+y)^2} + \sum_{d \geq 1} \sum_{m|d} m ((-y)^m - 2 + (-y)^{-m}) q^{d},
\label{wp_function_def}
\end{equation}
the derivative
\[ \wp^{\bullet}(z,\tau) = y \frac{d}{dy} \, \wp(z,\tau) = \frac{y (y-1)}{(1+y)^3} 
 + \sum_{d \geq 1} \sum_{m|d} m^2 ((-y)^m - (-y)^{-m}) q^{d}, \]
and for $k \geq 1$ the Eisenstein series
\begin{equation} E_{2k}(\tau) = 1 - \frac{4k}{B_{2k}} \sum_{d \geq 1} \Big( \sum_{m |d} m^{2k-1} \Big) q^d, \label{Eisenstein_function_def} \end{equation}
where $B_{2k}$ are the Bernoulli numbers.
Define the free polynomial algebra
\[ \mathsf{V} = \BC\big[ F(z,\tau), E_2(\tau),\, E_4(\tau),\, J_1(z,\tau),\, \wp(z,\tau),\, \wp^{\bullet}(z,\tau) \big] . \]
Define the weight and index of the generators of $\mathsf{V}$ by the following table.
Here, we list also their pole order at $z = 0$ for later use.
\begin{center}
\begin{longtable}[h!]{|l | c | c | c | c | c |}
\hline
& $F(z,\tau)$ & $E_{2k}(\tau)$ & $J_1(z,\tau)$ & $\wp(z,\tau)$ & $\wp^{\bullet}(z,\tau)$ \\
\hline
\text{pole order at $z=0$} & $0$ & $0$ & $1$ & $2$ & $3$ \\
\text{weight} & $-1$ & $2k$ & $1$ & $2$ & $3$ \\
\text{index} & $1/2$ & $0$ & $0$ & $0$ & $0$ \\
\hline
\caption{Weight and pole order at $z=0$}
\label{weight_pole_order_table}
\end{longtable}
\end{center}
\vspace{-1cm}
The grading on the generators induces a natural bigrading on $\mathsf{V}$ by weight $k$ and index $m$,
\[ \mathsf{V} = \bigoplus_{m \in (\frac{1}{2} \BZ)^{\geq 0}} \bigoplus_{k \in \BZ} \mathsf{V}_{k,m}, \]
where $m$ runs over all non-negative half-integers.

In the variable $z$, the functions
\begin{equation} E_{2k}(\tau),\ J_1(z,\tau),\ \wp(z,\tau),\ \wp^{\bullet}(z,\tau) \label{ckvmsvff} \end{equation}
can have a pole in the fundamental region
\begin{equation} \big\{ x + y \tau\  \big|\ 0 \leq x,y < 1 \big\} \label{fundamental_region} \end{equation}
only at $z = 0$.
The function $F(z,\tau)$ has a simple zero at $z = 0$
and no other zeros (or poles) in the fundamental region \eqref{fundamental_region}.

\begin{definition}
Let $m$ be a non-negative half-integer and let $k \in \BZ$.
A function
\[ f(z,\tau) \in \mathsf{V}_{k,m} \]
which is holomorphic at $z=0$ for generic $\tau$ is called a \emph{quasi-Jacobi form} of weight $k$ and index $m$.
\end{definition}

The subring $\QJac \subset \mathsf{V}$ of quasi-Jacobi forms is graded by index~$m$ and weight $k$,
\[ \QJac = \bigoplus_{m \geq 0} \bigoplus_{k  \geq -2m} \QJac_{k,m} \]
with finite-dimensional summands $\QJac_{k,m}$.

By the classical relation
\[ \big( \wp^{\bullet}(z) \big)^{2} = 4 \wp(z)^3 - \frac{1}{12} E_4(\tau) \wp(z) + \frac{1}{216} E_6(\tau). \]
we have $E_6(\tau) \in \mathsf{V}$ and therefore $E_6(\tau) \in \QJac$.
Hence, $\QJac$ contains the ring of quasi-modular forms $\BC[E_2, E_4, E_6]$.
Since the functions
\[ \varphi_{-2,1} = - F(z,\tau)^2, \quad \varphi_{0,1} = -12 F(z,\tau)^2 \wp(z,\tau), \]
lie both in $\QJac$, it follows from \cite[Theorem 9.3]{EZ} that $\QJac$ also contains the ring of weak Jacobi forms.

\begin{lemma} \label{closed_under_diff} The ring $\QJac$ is closed under differentiation by $z$ and $\tau$. \end{lemma}
\begin{proof}
We write 
\[ \partial_{\tau} = \frac{1}{2 \pi i} \frac{\partial}{\partial \tau} = q \frac{d}{dq}
 \quad \text{ and } \quad 
\partial_z = \frac{1}{2 \pi i} \frac{\partial}{\partial z} = y \frac{d}{dy}
\]
for differentiation with respect to $\tau$ and $z$ respectively.
The lemma now direct follows from the relations.
\begin{align*}
\partial_{\tau}(F) & = F \cdot \left( \frac{1}{2} J_1^{2} - \frac{1}{2} \wp - \frac{1}{12} E_2 \right),
& \partial_z(F) & = J_1 \cdot F, \\
\partial_{\tau}(J_1) & = J_1 \cdot \left( \frac{1}{12} E_2 - \wp \right) - \frac{1}{2} \wp^{\bullet},
& \partial_z(J_1) & = - \wp + \frac{1}{12} E_2, \\
\partial_{\tau}(\wp) & = 2 \wp^{2} + \frac{1}{6} \wp E_2 + J_1 \wp^{\bullet} - \frac{1}{36} E_4,
& \partial_z(\wp) & = \wp^{\bullet}, \\
\partial_{\tau}(\wp^{\bullet})
& = 6 J_1 \wp^{2} - \frac{1}{24} J_1 E_4 + 3 \wp \wp^{\bullet} + \frac{1}{4} E_2 \wp^{\bullet},
\quad \quad & \partial_z(\wp^{\bullet}) & = 6 \wp^2 - \frac{1}{24} E_4 \,. \qedhere
\end{align*}
\end{proof}

\subsubsection{Numerical values} \label{Appendix_Numerical_Values}
We present the first values of the functions $\varphi_{m,\ell}$
satisfying the conditions of Conjecture A of Section \ref{Main_conjecture_section2}.
Let $K = iF$, where $i = \sqrt{-1}$. Then,
{\allowdisplaybreaks
\small{
\begin{align*}
\varphi_{1,-1} & = K^2 \left( \frac{1}{2} J_1^{2} - \frac{1}{2} \wp - \frac{1}{12} E_2 \right) \\
\varphi_{1,0} & = -K \\
\varphi_{1,1} & = K^2 \left( \wp - \frac{1}{12} E_2 \right) \\[10pt]
\varphi_{2,-2} & = 2 K^4 \left( J_1^{4} - 2 J_1^{2} \wp - \frac{1}{12} J_1^{2} E_2 - \frac{1}{2} J_1 \wp^{\bullet} \right) \\
\varphi_{2,-1} & = 2 K^3 \left( \frac{2}{3} J_1^{3} -  J_1 \wp - \frac{1}{12} J_1 E_2 - \frac{1}{6} \wp^{\bullet} \right) \\
\varphi_{2,0} & = -2 \cdot J_1 \cdot K^{2} \\
\varphi_{2,1} & = 2 K^{3} \cdot \left( J_1 \wp - \frac{1}{12} J_1 E_2 + \frac{1}{2} \wp^{\bullet} \right) \\
\varphi_{2,2} + 1& = 2 K^{4} \cdot \left( J_1^{2} \wp - \frac{1}{12} J_1^{2} E_2 + \frac{3}{2} \wp^{2} + J_1 \wp^{\bullet} - \frac{1}{96} E_4 \right) \\[10pt]
\varphi_{3,-2} & =
3 K^5 \cdot \bigg(
\frac{9}{5} J_1^{5} - \frac{9}{2} J_1^{3} \wp - \frac{1}{8} J_1^{3} E_2 + \frac{1}{2} J_1 \wp^{2} \\
& \quad \quad \quad \quad \quad \quad \quad \quad \quad \quad + \frac{1}{24} J_1 \wp E_2 - \frac{5}{4} J_1^{2} \wp^{\bullet} + \frac{1}{180} J_1 E_4 + \frac{3}{20} \wp \wp^{\bullet}
\bigg) \\
\varphi_{3,-1} & = 3 K^4 \cdot \bigg(
\frac{9}{8} J_1^{4} - \frac{9}{4} J_1^{2} \wp - \frac{1}{8} J_1^{2} E_2 + \frac{1}{8} \wp^{2} + \frac{1}{24} \wp E_2 - \frac{1}{2} J_1 \wp^{\bullet} + \frac{1}{288} E_4
\bigg) \\
\varphi_{3,0} & = K^3 \cdot \left( -\frac{9}{2} J_1^{2} + \frac{3}{2} \wp \right) \\
\varphi_{3,1} & =
3 K^4 \cdot 
\left( \frac{3}{2} J_1^{2} \wp - \frac{1}{8} J_1^{2} E_2 + \frac{1}{2} \wp^{2} + \frac{1}{24} \wp E_2 + J_1 \wp^{\bullet} - \frac{1}{144} E_4 \right) \\
\varphi_{3,2} & = 3 K^5 \cdot \bigg( \frac{3}{2} J_1^{3} \wp - \frac{1}{8} J_1^{3} E_2 + \frac{7}{2} J_1 \wp^{2} \\
& \quad \quad \quad \quad \quad \quad \quad \quad \quad \quad
 + \frac{1}{24} J_1 \wp E_2 + \frac{7}{4} J_1^{2} \wp^{\bullet} - \frac{1}{36} J_1 E_4 + \frac{3}{4} \wp \cdot  \wp^{\bullet} \bigg) \\
\varphi_{3,3} + 1& =
3 K^6 \cdot \bigg( \frac{9}{4} J_1^{4} \wp - \frac{3}{16} J_1^{4} E_2 + \frac{15}{2} J_1^{2} \wp^{2} + \frac{1}{8} J_1^{2} \wp E_2 + 3 J_1^{3} \wp^{\bullet} \\
& \quad \quad + \frac{5}{4} \wp^{3} - \frac{1}{48} \wp^{2} E_2 - \frac{1}{16} J_1^{2} E_4 + 3 J_1 \wp \cdot \wp^{\bullet} - \frac{1}{144} \wp E_4 + \frac{1}{3} (\wp^{\bullet})^{2} \bigg) \\
\varphi_{4,0} & = K^4 \cdot \left( -\frac{32}{3} J_1^{3} + 8 J_1 \wp + \frac{2}{3} \wp^{\bullet} \right)
\end{align*}
}
\normalsize{
In the variables
\[ q = e^{2 \pi i \tau} \quad \text{ and } \quad s = (-y)^{1/2} = e^{\pi i z} \]
the first coefficients of the functions above are
}
\small{
\begin{align*}
\varphi_{1,-1} & = \left(- s^{-4} + 4 s^{-2} - 6 + 4s^{2} - s^{4}\right)q + O(q^{2}) \\
\varphi_{1,0} & = \left( s^{-1} - s \right) + \left( - s^{-3} + 3 s^{-1} - 3s + s^{3}\right)q + O(q^{2}) \\
\varphi_{1,1} & = \left( s^{-4} - 4 s^{-2} + 6 - 4s^{2} + s^{4}\right)q + O(q^{2}) \\[5pt]
\varphi_{2,-2} & = \left( -2s^{-6} + 4 s^{-4} + 2 s^{-2} - 8 + 2s^{2} + 4s^{4} - 2s^{6}\right)q + O(q^{2}) \\
\varphi_{2,-1} & = \left( -2 s^{-5}  + 6 s^{-3} - 4 s^{-1} - 4s + 6s^{3} - 2s^{5}\right)q + O(q^{2}) \\
\varphi_{2,-0} & = \left( s^{-2} - s^{2} \right) + \left( -4 s^{-4} + 8 s^{-2} - 8s^{2} + 4s^{4}\right)q + O(q^{2}) \\
\varphi_{2,1} & = \left( 2 s^{-5} - 6 s^{-3} + 4 s^{-1} + 4s - 6s^{3} + 2s^{5}\right)q + O(q^{2}) \\
\varphi_{2,2} + 1 & = 1 + \left( 2 s^{-6} - 4 s^{-4} - 2 s^{-2} + 8 - 2s^{2} - 4s^{4} + 2s^{6}\right)q + O(q^{2}) \\[5pt]
\varphi_{3,-2} & = \left(-3s^{-7} + 6s^{-5} - 3s^{-1} - 3s + 6s^5 - 3s^7 \right)q + O(q^2) \\
\varphi_{3,-1} & = \left(-3s^{-6} + 9s^{-4} - 9s^{-2} + 6 - 9s^2 + 9s^4 - 3s^6 \right)q + O(q^2) \\
\varphi_{3,0} & = \left(s^{-3} - s^{3} \right) + \left(-9s^{-5} + 18s^{-3} - 9s^{-1} + 9s - 18s^3 + 9s^5 \right)q + O(q^2) \\
\varphi_{3,1} & = \left(3s^{-6} - 9s^{-4} + 9s^{-2} - 6 + 9s^2 - 9s^4 + 3s^6 \right)q + O(q^2) \\
\varphi_{3,2} & = \left(3s^{-7} - 6s^{-5} + 3s^{-1} + 3s - 6s^5 + 3s^7 \right)q + O(q^2) \\
\varphi_{3,3} + 1 & = 1 + \left(3s^{-8} - 6s^{-6} + 3s^{-4} - 6s^{-2} + 12 - 6s^2 + 3s^4 - 6s^6 + 3s^8 \right)q + O(q^2) \\[10pt]
\varphi_{4,0} & = \left( s^{-4} - s^4 \right) + \left( -16 s^{-6} + 32s^{-4} - 16s^{-2} + 16s^2 - 32s^4 + 16s^6 \right)q + O(q^2).
\end{align*}
}
}


\newcommand{\etalchar}[1]{$^{#1}$}

\vspace{+14 pt}
\noindent
Departement Mathematik\\
ETH Z\"urich\\
georgo@math.ethz.ch


\begin{thebibliography}{HKK{\etalchar{+}}03}

\bibitem[Bea99]{Beauville2} A.~Beauville, {\em Counting rational curves on $K3$ surfaces}, Duke Math. J. {\bf 97} (1999), no. 1, 99--108. 

\bibitem[BF97]{BF}
K.~Behrend and B.~Fantechi, \emph{ The intrinsic normal cone},
\newblock Invent. Math. \textbf{128} (1997), no. 1, 45--88.


\bibitem[BL00]{BL} J.~Bryan and N.~C.~Leung, {\em The enumerative geometry of $K3$ surfaces
and modular forms}, J. Amer. Math. Soc. {\bf 13} (2000), no. 2, 371--410.

\bibitem[BOPY15]{BOPY}
J.~Bryan, G.~Oberdieck, R.~Pandharipande, and Q.~Yin, \emph{ Curve counting on
  abelian surfaces and threefolds},
  arXiv:1506.00841.

\bibitem[Cha85]{Chandra}
K.~Chandrasekharan,
\emph{ Elliptic functions}, volume 281 of \emph{Grundlehren der
  Mathematischen Wissenschaften}, Springer-Verlag, Berlin, 1985.

\bibitem[Che02]{ChenK3}
X.~Chen, \emph{ A simple proof that rational curves on {$K3$} are nodal},
\newblock Math. Ann. \textbf{ 324} (2002), no. 1, 71--104.

\bibitem[CK14]{CK14}
C.~Ciliberto and A.~L. Knutsen, \emph{ On {$k$}-gonal loci in {S}everi
  varieties on general {$K3$} surfaces and rational curves on hyperk\"ahler
  manifolds},
\newblock J. Math. Pures Appl. (9) \textbf{ 101} (2014), no. 4, 473--494.

\bibitem[EZ85]{EZ}
M.~Eichler and D.~Zagier,
\newblock \emph{ The theory of {J}acobi forms}, volume~55 of \emph{
  Progress in Mathematics},
\newblock Birkh\"auser Boston Inc., Boston, MA, 1985.

\bibitem[FP97]{FP}
W.~Fulton and R.~Pandharipande,
\newblock \emph{Notes on stable maps and quantum cohomology},
\newblock in \emph{ Algebraic geometry---{S}anta {C}ruz 1995}, volume~62 of
  Proc. Sympos. Pure Math., pages 45--96, Amer. Math. Soc.,
  Providence, RI, 1997.

\bibitem[FP05]{FPM}
C.~Faber and R.~Pandharipande, \emph{ Relative maps and tautological
  classes},
\newblock J. Eur. Math. Soc. (JEMS) \textbf{ 7} (2005), no. 1, 13--49.

\bibitem[FP13]{FP13}
C.~Faber and R.~Pandharipande,
\newblock {\em Tautological and non-tautological cohomology of the moduli space of curves},
\newblock in {\em Handbook of moduli}, Vol. I, 293--330, Adv. Lect. Math. (ALM), {\bf 24}, Int. Press, Somerville, MA, 2013.

\bibitem[Ful98]{Fulton}
W.~Fulton,
\newblock \emph{ Intersection theory}, volume~2 of \emph{ Ergebnisse der
  Mathematik und ihrer Grenzgebiete. 3. Folge. A Series of Modern Surveys in
  Mathematics},
\newblock Springer-Verlag, Berlin, second edition, 1998.

\bibitem[GHS13]{GH}
V.~Gritsenko, K.~Hulek, and G.~K. Sankaran,
\newblock Moduli of {K}3 surfaces and irreducible symplectic manifolds,
\newblock in \emph{ Handbook of moduli. {V}ol. {I}}, volume~24 of 
  Adv. Lect. Math. (ALM), pages 459--526, Int. Press, Somerville, MA, 2013.

\bibitem[Gra01]{Grab}
T.~Graber, {\em Enumerative geometry of hyperelliptic plane curves}, J. Algebraic Geom. {\bf 10} (2001), no. 4, 725--755.

\bibitem[Gro96]{GrojH}
I.~Grojnowski, \emph{ Instantons and affine algebras. {I}. {T}he {H}ilbert
  scheme and vertex operators},
\newblock Math. Res. Lett. \textbf{ 3} (1996), no. 2, 275--291.

\bibitem[HKK{\etalchar{+}}03]{MirSym}
K.~Hori, S.~Katz, A.~Klemm, R.~Pandharipande, R.~Thomas, C.~Vafa, R.~Vakil, and
  E.~Zaslow,
\newblock \emph{ Mirror symmetry}, with a preface by C.~Vafa, volume~1 of \emph{ Clay Mathematics
  Monographs},
\newblock American Mathematical Society, Providence, RI; Clay Mathematics
  Institute, Cambridge, MA, 2003.

\bibitem[HM82]{HM}
J.~Harris and D.~Mumford, \emph{ On the {K}odaira dimension of the moduli
  space of curves},
with an appendix by W.~Fulton,
\newblock Invent. Math. \textbf{ 67} (1982), no. 1, 23--88.

\bibitem[KKP14]{KKP}
S.~Katz, A.~Klemm, and R.~Pandharipande,
\emph{On the motivic stable pairs invariants of K3 surfaces},
arXiv:1407.3181.

\bibitem[KY00]{KY}
T.~Kawai and K.~Yoshioka, \emph{ String partition functions and infinite
  products},
\newblock Adv. Theor. Math. Phys. \textbf{ 4} (2000), no. 2, 397--485.

\bibitem[KL13]{KL}
Y.-H. Kiem and J.~Li, \emph{ Localizing virtual cycles by cosections},
\newblock J. Amer. Math. Soc. \textbf{ 26} (2013), no. 4, 1025--1050.

\bibitem[Leh99]{Lehn2}
M.~Lehn, \emph{ Chern classes of tautological sheaves on {H}ilbert schemes of
  points on surfaces},
\newblock Invent. Math. \textbf{ 136} (1999), no. 1, 157--207.

\bibitem[Leh04]{Lehn}
M.~Lehn,
\newblock \emph{Lectures on {H}ilbert schemes},
\newblock in \emph{ Algebraic structures and moduli spaces}, volume~38 of
  CRM Proc. Lecture Notes, 1--30, Amer. Math. Soc., Providence,
  RI, 2004.

\bibitem[LQW02]{LQW}
W.-P.~Li, Z.~Qin, and W.~Wang,
\emph{Vertex algebras and the cohomology ring structure of Hilbert schemes of points on surfaces},
Math. Ann. \textbf{324} (2002), no. 1, 105–133. 

\bibitem[LS03]{LS_K3}
M.~Lehn and C.~Sorger, \emph{ The cup product of {H}ilbert schemes for {$K3$}
  surfaces},
\newblock Invent. Math. \textbf{ 152} (2003), no. 2, 305--329.

\bibitem[MO09a]{MO}
D.~Maulik and A.~Oblomkov, \emph{ Donaldson-{T}homas theory of {${\scr
  A}_n\times \p^1$}},
\newblock Compos. Math. \textbf{ 145} (2009), no. 5, 1249--1276.

\bibitem[MO09b]{MO2}
D.~Maulik and A.~Oblomkov, \emph{ Quantum cohomology of the {H}ilbert scheme
  of points on {$\scr A_n$}-resolutions},
\newblock J. Amer. Math. Soc. \textbf{ 22} (2009), no. 4, 1055--1091.

\bibitem[MO12]{QGQC}
D.~Maulik and A.~Okounkov, \emph{ Quantum Groups and Quantum Cohomology},
arXiv:1211.1287.

\bibitem[MP13]{GWNL}
D.~Maulik and R.~Pandharipande,
\newblock \emph{Gromov-{W}itten theory and {N}oether-{L}efschetz theory},
\newblock in \emph{ A celebration of algebraic geometry}, volume~18 of
  Clay Math. Proc., pages 469--507, Amer. Math. Soc., Providence, RI,
  2013.

\bibitem[MPT10]{MPT}
D.~Maulik, R.~Pandharipande, and R.~P.~Thomas, {\em Curves on $K3$ surfaces and modular forms}, with an appendix by A.~Pixton, J. Topol. {\bf 3} (2010), no. 4, 937--996.

\bibitem[Nak97]{N2}
H.~Nakajima, \emph{ Heisenberg algebra and {H}ilbert schemes of points on
  projective surfaces},
\newblock Ann. of Math. (2) \textbf{ 145} (1997), no. 2, 379--388.

\bibitem[Nak99]{Nakajima}
H.~Nakajima,
\newblock \emph{ Lectures on {H}ilbert schemes of points on surfaces},
  volume~18 of \emph{ University Lecture Series},
\newblock American Mathematical Society, Providence, RI, 1999.

\bibitem[Obe12]{O}
G.~Oberdieck, \emph{A Serre derivative for even weight Jacobi Forms},
arXiv:1209.5628.

\bibitem[Obe15]{thesis}
G.~Oberdieck, \emph{The enumerative geometry of the Hilbert schemes of points of a K3 surface}, PhD thesis.

\bibitem[OP10b]{OP}
A.~Okounkov and R.~Pandharipande, \emph{Quantum cohomology of the {H}ilbert
  scheme of points in the plane},
\newblock Invent. Math. \textbf{179} (2010), no. 3, 523--557.

\bibitem[OP14]{K3xE}
G.~Oberdieck and R.~Pandharipande, \emph{ Curve counting on $K3 \times E\, $,
  the Igusa cusp form $\chi_{10}\, $, and descendent integration},
  arXiv:1411.1514.

\bibitem[Pon07]{Pon07}
D.~Pontoni, \emph{ Quantum cohomology of {${\rm Hilb}^2(\Bbb P^1\times\Bbb
  P^1)$} and enumerative applications},
\newblock Trans. Amer. Math. Soc. \textbf{ 359} (2007), no. 11, 5419--5448.

\bibitem[Pri12]{Pridham}
J.~P. Pridham, \emph{ Semiregularity as a consequence of Goodwillie's
  theorem},
arXiv:1208.3111.

\bibitem[PT14b]{KKV}
R.~Pandharipande and R.~P. Thomas,
\newblock \emph{The Katz-Klemm-Vafa conjecture for K3 surfaces},
arXiv:1404.6698.

\bibitem[Ros14]{Ros14}
S.~C.~F. Rose, \emph{ Counting hyperelliptic curves on an {A}belian surface
  with quasi-modular forms},
\newblock Commun. Number Theory Phys. \textbf{ 8} (2014), no. 2, 243--293.

\bibitem[STV11]{STV}
T.~Sch\"urg, B.~To\"en, and G.~Vezzosi,
\emph{Derived algebraic geometry, determinants of perfect complexes, and applications to obstruction theories for maps and complexes},
J. Reine Angew. Math. \textbf{702} (2015), 1–40. 

\bibitem[YZ96]{YZ}
S.-T. Yau and E.~Zaslow, \emph{ B{PS} states, string duality, and nodal
  curves on {$K3$}},
\newblock Nuclear Phys. B \textbf{ 471} (1996), no.3 , 503--512.

\end{thebibliography}
\end{document}